\documentclass[11pt]{article}
\usepackage[utf8]{inputenc}
\usepackage[T1]{fontenc}
\usepackage[english]{babel}
\usepackage[top=2cm,bottom=2cm,left=2cm,right=2cm]{geometry}
\usepackage[colorlinks,linkcolor=blue,anchorcolor=green,citecolor=red,urlcolor=blue,breaklinks]{hyperref}
\usepackage{amsfonts,amssymb,amsthm}
\usepackage{amsmath}
\usepackage{mathrsfs} 
\usepackage{graphicx}
\usepackage{enumerate}
\usepackage{tikz}
\usepackage{multirow}
\usepackage{array}
\usepackage{url}
\usepackage{breakurl}
\usepackage{dsfont}

\newcommand{\eps}{\varepsilon}

\newcommand{\paren}[1]{\left (#1\right )}

\newcommand{\argmin}{\mathop{\rm argmin}}

\newcommand{\mcal}[1]{\mathcal{#1}}

\newcommand{\dv}{\partial}
\newcommand{\gradi}{\nabla}
\newcommand{\R}{\mathbb{R}}
\newcommand{\Z}{\mathbb{Z}}
\newcommand{\N}{\mathbb{N}}

\newcommand{\T}{\mathcal{T}}
\newcommand{\PP}{\mathcal{P}}
\newcommand{\M}{\mathcal{M}}
\newcommand{\E}{\mathcal{E}}

\newcommand{\U}{\mathbb{U}}

\newcommand{\W}{\mathbb{W}}

\newcommand{\udd}{u^{\sharp}}
\newcommand{\tdd}{\tau^{\sharp}}

\newcommand{\pidd}{\pi^{\sharp}}

\newcommand{\Me}{\mathcal{M}^{*}}
\newcommand{\vect}[1]{\mathbb{#1}}

\renewcommand{\r}{\mcal{R}}
\newcommand{\w}{\mcal{W}}

\newcommand{\bfa}{{\bf a}}

\newcommand{\bigzero}{\mbox{\normalfont\Large\bfseries 0}}
\newcommand{\rvline}{\hspace*{-\arraycolsep}\vline\hspace*{-\arraycolsep}}

\newcommand{\dx}{\,{\rm d}x}
\def\Ind{\mathds{1}}

\newtheorem{theorem}{Theorem}[section]

\newtheorem{prop}[theorem]{Proposition}
\newtheorem{definition}{Definition}[section]
\newtheorem{remark}{Remark}[section]

\allowdisplaybreaks

\title{\textbf{\Large
A RELAXATION SCHEME FOR A HYPERBOLIC MULTIPHASE FLOW MODEL. PART I: BAROTROPIC EOS}}

\author{Khaled Saleh$^{1}$}
\date{}

\begin{document}

\maketitle

{\centering
${}^{1}$ Universit\'e de Lyon, CNRS UMR 5208, Universit\'e Lyon 1, Institut Camille Jordan, 43 bd 11 novembre 1918; F-69622 Villeurbanne cedex, France.\\
}

\begin{abstract}
This article is the first of two in which we develop a relaxation finite volume scheme for the convective part of the multiphase flow models introduced in the series of papers \cite{her-16-cla,Herard,bou-18-rel}. In the present article we focus on barotropic flows where in each phase the pressure is a given function of the density. The case of general equations of state will be the purpose of the second article. We show how it is possible to extend the relaxation scheme designed in \cite{coq-13-rob} for the barotropic Baer-Nunziato two phase flow model to the multiphase flow model with $N$ - where $N$ is arbitrarily large - phases. The obtained scheme inherits the main properties of the relaxation scheme designed for the Baer-Nunziato two phase flow model. It applies to general barotropic equations of state. It is able to cope with arbitrarily small values of the statistical phase fractions. The approximated phase fractions and phase densities are proven to remain positive and a fully discrete energy inequality is also proven under a classical CFL condition.
For $N=3$, the relaxation scheme is compared with Rusanov's scheme, which is the only numerical scheme presently available for the three phase flow model (see \cite{bou-18-rel}).
For the same level of refinement, the relaxation scheme is shown to be much more accurate than Rusanov's scheme, and for a given level of approximation error, the relaxation scheme is shown to perform much better in terms of computational cost than Rusanov's scheme.
Moreover, contrary to Rusanov's scheme which develops strong oscillations when approximating vanishing phase solutions, the numerical results show that the relaxation scheme remains stable in such regimes.
\end{abstract}



\section{Introduction}\label{sec:intro}

A multiphase flow is a flow involving the simultaneous presence of materials with different states or phases (for instance gas-liquid-solid mixtures) or materials in the same state or phase but with different chemical properties (for instance non miscible liquid-liquid mixtures).
The modeling and numerical simulation of multiphase flows is a relevant approach for a detailed investigation of some patterns occurring in many industrial sectors. In the oil and petroleum industry for instance, multiphase flow modeling is needed for the understanding of pipe flows where non miscible oil, liquid water and gas and possibly solid particles are involved. In the chemical industry some synthesis processes are based on three phase chemical reactors. In the metalworking industry, some cooling processes involve multiphase flows. 

\medskip
In the nuclear industry, many applications involve multiphase flows such as accidental configurations that may arise in pressurized water reactors, among which the Departure from Nucleate Boiling (DNB) \cite{dnb}, the Loss
of Coolant Accident (LOCA) \cite{loca}, the re-flooding phase following a LOCA or the Reactivity Initiated Accident (RIA) \cite{ria}, that all involve two phase liquid-vapor flows. Other accidental configurations involve three phase flows such as the steam explosion, a phenomenon consisting in violent boiling or flashing of water into steam, occurring when the water is in contact with hot molten metal particles of ``corium'': a liquid mixture of nuclear fuel, fission products, control rods, structural materials, etc.. resulting from a core meltdown. The corium is fragmented in droplets the order of magnitude of which is 100$\mu m$. This allows a very rapid heat transfer to the surrounding water in a time less than the characteristic time of the pressure relaxation associated with water evaporation, hence a dramatic increase of pressure and the possible apparition of shock and rarefaction pressure waves that may also damage the reactor structure and cause a containment failure. The passage of the pressure wave through the pre-dispersed metal creates flow forces which further fragment the melt, increasing the interfacial area between corium droplets and liquid water, hence resulting in rapid heat transfer, and thus sustaining the process. We refer the reader to \cite{ber-00-vap} and the references therein in order to have a better understanding of that problem, and also to the recent paper \cite{mei-14-cha}.

\medskip
The modeling and numerical simulation of the steam explosion is an open topic up to now. Since the sudden increase of vapor concentration results in huge pressure waves including shock and rarefaction waves, compressible multiphase flow models with unique jump conditions and for which the initial-value problem is well posed are mandatory. Some modeling efforts have been provided in this direction in \cite{her-16-cla,Herard,bou-18-rel,Muller-Hantke}. The $N$-phase flow models developed therein consist in an extension to $N\geq 3$ phases of the well-known Baer-Nunziato two phase flow model \cite{bae-86-two}. They consist in $N$ sets of partial differential equations (PDEs) accounting for the evolution of phase fraction, density, velocity and energy of each phase. As in the Baer-Nunziato model, the PDEs are composed of a hyperbolic first order convective part consisting in $N$ Euler-like systems coupled through non-conservative terms and zero-th order source terms accounting for pressure, velocity and temperature relaxation phenomena between the phases. It is worth noting that the latter models are quite similar to the classical two phase flow models in \cite{bo-08-com,bdz-99-two,gav-02-mat}. We emphasize that the models considered here are only suitable for non miscible fluids. We refer to \cite{mat-18-the,mat-19-thr} and the references therein for the modeling of flows involving miscible fluids such as gas-gas mixtures.

\medskip
The existing approach to approximate the admissible weak solutions of these models consists in a fractional step method that treats separately convective effects and relaxation source terms. The present work is only concerned with the numerical treatment of the convective part. For the numerical treatment of the relaxation source terms in the framework of barotropic equations of state, we refer the reader to \cite{bou-18-rel}. Up to now, Rusanov's scheme is the only numerical scheme available for the simulation of the convective part of the considered multiphase flow model (see \cite{bou-18-rel}). Rusanov's scheme is well known for its simplicity but also for its poor accuracy due to the very large associated numerical viscosity. Another drawback of Rusanov's scheme observed when simulating two phase flows is its lack of robustness in the regimes of vanishing phases occurring when one (or more) phase is residual and the associated phase fraction is close to zero. Since accuracy and robustness are critical for the reliable simulation of 2D and 3D phenomena arising in multiphase flows such as the steam explosion, one must develop dedicated Riemann solver for these models.

\medskip
The aim of this work is to develop a relaxation finite volume scheme for the barotropic multiphase flow model introduced in \cite{her-16-cla}. In particular, we show how it is possible to extend the relaxation scheme designed in \cite{coq-13-rob} for the barotropic Baer-Nunziato two phase flow model to the multiphase flow model with $N$ - where $N$ is arbitrarily large - phases. The obtained scheme inherits the main properties of the relaxation scheme designed for the Baer-Nunziato model. Since it is a Suliciu type relaxation scheme, it applies to any barotropic equation of state, provided that the pressure is an increasing function of the density (see \cite{bou-04-non,Suliciu1,Suliciu2}). The scheme is able to cope with arbitrarily small values of the statistical phase fractions, which are proven to remain positive as well as the phase densities. Finally, a fully discrete energy inequality is also proven under a classical CFL condition.

\medskip
In \cite{coq-17-pos}, the relaxation scheme has shown to compare well with two of the most popular existing schemes available for the full Baer-Nunziato model (with energy equations), namely Schwendeman-Wahle-Kapila's Godunov-type scheme \cite{SWK} and Tokareva-Toro's HLLC scheme \cite{TT}. Still for the Baer-Nunziato model, the relaxation scheme also has shown a higher precision and a lower computational cost (for comparable accuracy) than Rusanov's scheme.
Regarding the multiphase flow model considered in the present paper, the relaxation scheme is compared to Rusanov's scheme, the only scheme presently available.
As expected, for the same level of refinement, the relaxation scheme is shown to be much more accurate than Rusanov's scheme, and for a given level of approximation error, the relaxation scheme is shown to perform much better in terms of computational cost than Rusanov's scheme.
Moreover, the numerical results show that the relaxation scheme is much more stable than Rusanov's scheme which develops strong oscillations in vanishing phase regimes.

\medskip
The relaxation scheme described here is restricted to the simulations of flows with subsonic relative speeds, \emph{i.e.} flows for which the difference between the material velocities of the phases is less than the monophasic speeds of sound. For multiphase flow simulations in the nuclear industry context, this is not a restriction, but it would be interesting though to extend the present scheme to sonic and supersonic flows.

\medskip
For the sake of concision and simplicity, this work is only concerned with barotropic equations of states. However, as it was done for the two phase Baer-Nunziato model in \cite{coq-17-pos}, an extension of the relaxation scheme to the full multiphase flow model with energy equations is within easy reach. This will be the purpose of a companion paper.

\medskip
The paper is organized as follows. Section \ref{sec:model} is devoted to the presentation of the multiphase flow model. In Section \ref{sec:riemsolv} we explain how to extend the relaxation Riemann solver designed in \cite{coq-13-rob} to the multiphase flow model and Section \ref{secnumtest} is devoted to the numerical applications on the three phase flow model. In addition to a convergence and CPU cost study in Test-case 1, we simulate in Test-cases 2 and 3 vanishing phase configurations where two of the three phases have nearly disappeared in some space region. In particular, Test-case 3 is dedicated to the interaction of a gas shock wave with a lid of rigid particles.

\section{The multiphase flow model}\label{sec:model}

We consider the following system of partial differential equations (PDEs) introduced in \cite{her-16-cla} for the modeling of the evolution of $N$ distinct compressible phases in a one dimensional space: for $k=1,..,N$, $x\in\R$ and $t>0$:  
\begin{subequations}
\label{sys:multi0}
\begin{align}
\label{sys:multi:void}
& \dv_t \alpha_k + \mathscr{V}_I(\U)\dv_x \alpha_k = 0,  \\
\label{sys:multi:mass}
& \dv_t \paren{\alpha_k \rho_k} + \dv_x \paren{\alpha_k \rho_k u_k}= 0,  \\
\label{sys:multi:mom}
& \textstyle \dv_t \paren{\alpha_k \rho_k u_k} + \dv_x \paren{\alpha_k \rho_k u_k^2+\alpha_k p_k(\rho_k)}+\sum_{\substack{l=1 \\ l\neq k}}^N \mathscr{P}_{kl}(\U)\dv_x \alpha_l= 0.
\end{align}
\end{subequations}

The model consists in $N$ coupled Euler-type systems.
The quantities $\alpha_k$, $\rho_k$ and $u_k$ represent the mean statistical fraction, the mean density and the mean velocity in phase $k$ (for $k=1,..,N$). The quantity $p_k$ is the pressure in phase $k$. We assume barotropic pressure laws for each phase so that the pressure is a given function of the density $\rho_k \mapsto p_k(\rho_k)$ with the classical assumption that $p_k'(\rho_k)>0$. The mean statistical fractions and the mean densities are positive and the following saturation constraint holds everywhere at every time: $\sum_{k=1}^N \alpha_k =1$.
Thus, among the $N$ equations \eqref{sys:multi:void}, $N-1$ are independent and the main unknown $\U$ is expected to belong to the physical space:
\begin{multline*}
\Omega_{\U}=
\Big \lbrace 
\U=\paren{\alpha_1,..,\alpha_{N-1},\alpha_1\rho_1,..,\alpha_N\rho_N,\alpha_1\rho_1 u_1,..,\alpha_N\rho_N u_N}^T\in\R^{3N-1}, \\
\text{such that} \ 0 < \alpha_1,..,\alpha_{N-1} < 1 \ \text{and} \ \alpha_k\rho_k >0 \ \text{for all} \ k=1,..,N  
\Big \rbrace.
\end{multline*}
Following \cite{her-16-cla}, several closure laws can be given for the so-called interface velocity $\mathscr{V}_I(\U)$ and interface pressures $\mathscr{P}_{kl}(\U)$. Throughout the whole paper, we make the following choice :
\begin{equation}
 \label{multi:vi:pi}
\mathscr{V}_I(\U) = u_1, \quad \text{and} \quad
\left \lbrace 
 \begin{array}{lll}
  \text{for $k=1$}, \quad & \mathscr{P}_{1l}(\U) = p_l(\rho_l), \quad & \text{for $l=2,..,N$}\\
  \text{for $k\neq 1$}, \quad & \mathscr{P}_{kl}(\U) = p_k(\rho_k),  \quad& \text{for $l=1,..,N, \, l\neq k$}.
 \end{array}
\right.
\end{equation}

With this particular choice, observing that the saturation constraint gives $\sum_{l=1,l\neq k}^N\dv_x\alpha_l=-\dv_x\alpha_k$ for all $k=1,..,N$ the momentum equations \eqref{sys:multi:mom} can be simplified as follows:
\begin{align}
\label{multi:mom1} &
 \textstyle \dv_t \paren{ \alpha_1\rho_1 u_1} + \dv_x  \paren{\alpha_1\rho_1 u_1^2+ \alpha_1 p_1(\rho_1)} + \sum_{l=2}^N  p_l(\rho_l)  \dv_x \alpha_l =0, \\[2ex]
\label{multi:momk} & \dv_t \paren{ \alpha_k\rho_k u_k} + \dv_x  \paren{\alpha_k\rho_k u_k^2+ \alpha_k p_k(\rho_k)} - p_k(\rho_k)  \dv_x \alpha_k =0, & k=2,..,N.
\end{align}

\medskip
\begin{remark}
When $N=2$, system \eqref{sys:multi0} is the convective part of the Baer-Nunziato two phase flow model \cite{bae-86-two}.
This model is thus an extension of the Baer-Nunziato two phase flow model to $N$ (possibly $\geq 3$) phases. 
As for the Baer-Nunziato model, in the areas where all the statistical fractions $\alpha_k$ are constant in space, system \eqref{sys:multi0} consists in $N$ independent Euler systems weighted by the statistical fraction of the corresponding phase. These Euler systems are coupled through non-conservative terms which are active only in the areas where the statistical fractions gradients are non zero. 
\end{remark}

\begin{remark}
The choice $\mathscr{V}_I(\U) = u_1$ is classical for the two phase flow model when phase 1 is dispersed and phase 2 prevails in the fluid. It is then natural to take an interfacial velocity which is equal to the material velocity of the dispersed phase. For three phase flows, the choice $\mathscr{V}_I(\U) = u_1$ has been made in \cite{bou-18-rel}. When simulating the steam explosion phenomenon \cite{ber-00-vap,mei-14-cha}, it corresponds to taking an interfacial velocity equal to the material velocity of the corium particles. 
\end{remark}

\medskip
The following proposition characterizes the wave structure of system \eqref{sys:multi0}:
\begin{prop}
\label{prop:spectre:multi}
With the closure laws \eqref{multi:vi:pi}, system \eqref{sys:multi0} is weakly hyperbolic on $\Omega_{\U}$ : it admits the following $3N-1$ real eigenvalues: $\sigma_1(\U) =..=\sigma_{N-1}(\U) =u_1$, $\sigma_{N-1+k}(\U) =u_k-c_k(\rho_k)$ for $k=1,..,N$ and $\sigma_{2N-1+k}(\U) =u_k+c_k(\rho_k)$ for $k=1,..,N$, 
where $c_k(\rho_k) = \sqrt{p_k'(\rho_k)}$. The corresponding right eigenvectors are linearly independent if, and only if,
\begin{equation}
\label{hypfail}
\alpha_k \neq 0, \quad \forall k=1,..,N \qquad \text{and} \qquad |u_1-u_k| \neq c_k(\rho_k), \quad \forall k=2,..,N.
\end{equation}
The characteristic field associated with $\sigma_1(\U) ,..,\sigma_{N-1}(\U) $ is linearly degenerate while the characteristic fields associated with $\sigma_{N-1+k}(\U) $ and $\sigma_{2N-1+k}(\U) $ for $k=1,..,N$ are genuinely non-linear.
\end{prop}

\medskip

\begin{proof}
 The proof can be found in \cite{Note-multi}.
\end{proof}

\medskip

\begin{remark}
The system is not hyperbolic in the usual sense because when (\ref{hypfail}) is not satisfied, the right eigenvectors do not span the whole space $\R^{3N-1}$. Two possible phenomena may cause a loss of the strict hyperbolicity: an interaction between the linearly degenerate field of velocity $u_1$ with one of the acoustic fields of the phase $k$ for $k=2,..,N$, and vanishing values of one of the phase fractions $\alpha_k$, $k=1,..,N$. In the physical configurations of interest in the present work (such as three phase flows in nuclear reactors), the flows have strongly subsonic relative velocities, \emph{i.e.} a relative Mach number much smaller than one:
\begin{equation}
M_k=\frac{|u_1-u_k|}{c_k(\rho_k)} << 1, \qquad \forall k=2,..,N,
\end{equation}
so that resonant configurations corresponding to wave interaction between acoustic fields and the $u_1$-contact discontinuity are unlikely to occur.
In addition, following the definition of the admissible physical space $\Omega_{\U}$, one never has $\alpha_k=0$. However,  $\alpha_k=0$ is to be understood in the sense $\alpha_k\to 0$ since one aim of this work is to construct a robust enough numerical scheme that could handle all the possible values of $\alpha_k\in(0,1)$, $k=1,..,N$, especially, arbitrarily small values.
\end{remark}

An important consequence of the closure law $\mathscr{V}_I(\U) = u_1$ is the linear degeneracy of the field associated with the eigenvalue $\sigma_1(\U)=..=\sigma_{N-1}(\U)=u_1$. This allows to define solutions with discontinuous phase fractions through the Riemann invariants of this linear field. Indeed, as proven in \cite{Note-multi}, there are $2N$ independent Riemann invariants associated with this field which is enough to parametrize the integral curves of the field since the multiplicity of the eigenvalue is $N-1$. This can be done as long as the system is hyperbolic \emph{i.e.} as long as \eqref{hypfail} is satisfied, which prevents the interaction between shock waves and the non conservative products in the model. 

\medskip
An important consequence of the closure law \eqref{multi:vi:pi} for the interface pressures  $\mathscr{P}_{kl}(\U)$ is the existence of an additional conservation law for the smooth solutions of \eqref{sys:multi0}.
Defining the specific internal energy of phase $k$, $e_k$ by $e_k'(\rho_k)=p_k(\rho_k)/\rho_k^2$ and the specific total energy of phase $k$ by $E_k=u_k^2/2 +  e_k(\rho_k)$, the smooth solutions of \eqref{sys:multi0} satisfy the following identities:
\begin{align}
\label{multi:ener1} &
 \textstyle \dv_t \paren{ \alpha_1\rho_1 E_1} + \dv_x  \paren{\alpha_1\rho_1 E_1 u_1+ \alpha_1 p_1(\rho_1)u_1} + u_1 \sum_{l=2}^N  p_l(\rho_l)  \dv_x \alpha_l =0, \\[2ex]
\label{multi:enerk} & \dv_t \paren{ \alpha_k\rho_k E_k} + \dv_x  \paren{\alpha_k\rho_k E_k u_k+ \alpha_k p_k(\rho_k)u_k} - u_1 p_k(\rho_k)  \dv_x \alpha_k =0, & k=2,..,N.
\end{align}

In \cite{Note-multi}, the mappings $\U\mapsto(\alpha_k\rho_k E_k)(\U)$ are proven to be (non strictly) convex for all $k=1,..,N$. Since, as long as the system is hyperbolic, the gradients of $\alpha_k$ are supported away from shock waves, it is natural to use theses mappings as mathematical entropies of system \eqref{sys:multi0} and select the physical non-smooth weak solutions of \eqref{sys:multi0} as those which satisfy he following entropy inequalities in the weak sense:
\begin{align}
\label{multi:ener1:ineq} &
 \textstyle \dv_t \paren{ \alpha_1\rho_1 E_1} + \dv_x  \paren{\alpha_1\rho_1 E_1 u_1+ \alpha_1 p_1(\rho_1)u_1} + u_1 \sum_{l=2}^N  p_l(\rho_l)  \dv_x \alpha_l \leq 0, \\[2ex]
\label{multi:enerk:ineq} & \dv_t \paren{ \alpha_k\rho_k E_k} + \dv_x  \paren{\alpha_k\rho_k E_k u_k+ \alpha_k p_k(\rho_k)u_k} - u_1 p_k(\rho_k)  \dv_x \alpha_k \leq 0, & k=2,..,N.
\end{align}
If a shock appears in phase 1, inequality \eqref{multi:ener1:ineq} is strict and if a shock appears in phase $k$ for some $k\in\lbrace2,..,N\rbrace$ inequality \eqref{multi:enerk:ineq} is strict.
Summing for $k=1,..,N$, the entropy weak solutions of \eqref{sys:multi0} are seen to satisfy the following total energy inequality:
\begin{equation}
 \label{multi:enertot}
 \textstyle \dv_t \paren{\sum_{k=1}^N \alpha_k\rho_k E_k} + \dv_x  \paren{\sum_{k=1}^N \paren{\alpha_k\rho_k E_k u_k+ \alpha_k p_k(\rho_k)u_k}}  \leq 0. 
\end{equation}
Obviously, for smooth solutions, \eqref{multi:enertot} is an equality.

\section{A relaxation approximate Riemann solver}
\label{sec:riemsolv}
As for the Baer-Nunziato two phase flow model, system \eqref{sys:multi0} has genuinely non-linear fields associated with the phasic acoustic waves, which makes the construction of an exact Riemann solver very difficult.
Following similar steps as in \cite{coq-13-rob}, we introduce a relaxation approximation of the multiphase flow model \eqref{sys:multi0} which is an enlarged system involving $N$ additional unknowns $\T_k$, associated with linearizations of the phasic pressure laws. These linearizations are designed to get a quasilinear enlarged system, shifting the initial non-linearity from the convective part to a stiff relaxation source term. The relaxation approximation is based on the idea that the solutions of the original system are formally recovered as the limit of the solutions of the proposed enlarged system, in the regime of a vanishing relaxation coefficient $\eps>0$. For a general framework on relaxation schemes we refer to \cite{CGPIR,CGS,bou-04-non}.

\medskip
We propose to approximate the solutions of \eqref{sys:multi0} by the solutions of the following Suliciu relaxation type model (see \cite{bou-04-non,Suliciu1,Suliciu2}) in the limit $\eps\to0$:
\begin{equation}
\label{sys:multi:relax0}
 \dv_t \W^\eps + \dv_x {\bf g}(\W^\eps)+ {\bf d}(\W^\eps)\dv_x\W^\eps =\dfrac{1}{\eps} \mathcal{R}(\W^{\eps}), \quad x\in\R, \, t>0,
\end{equation}
where the relaxation unknown is now expected to belong to the following enlarged phase space: 
\begin{multline*}
\Omega_{\W}=
\Big \lbrace 
\W=\paren{\alpha_1,..,\alpha_{N-1},\alpha_1\rho_1,..,\alpha_N\rho_N,\alpha_1\rho_1 u_1,..,\alpha_N\rho_N u_N,\alpha_1\rho_1 \T_1,..,\alpha_N\rho_N \T_N}^T\in\R^{4N-1}, \\
\text{such that} \ 0 < \alpha_1,..,\alpha_{N-1} < 1,  \ \alpha_k\rho_k >0 \ \text{and} \ \alpha_k\rho_k\T_k >0 \  \text{for} \ k=1,..,N  
\Big \rbrace,
\end{multline*}
and where:
\begin{equation*}
{\bf g}(\W)=\left [
\begin{matrix}
 0 \\
 \vdots \\
 0 \\
 \alpha_1 \rho_1 u_1 \\
 \vdots \\
 \alpha_N \rho_N u_N\\
 \alpha_1\rho_1 u_1^2 + \alpha_1 \pi_1\\
 \vdots \\
 \alpha_N\rho_N u_N^2 + \alpha_N \pi_N\\
 \alpha_1 \rho_1 \T_1 u_1\\
 \vdots \\
 \alpha_N \rho_N \T_N u_N
\end{matrix}  \right ], \quad {\bf d}(\W)\dv_x\W=
\left [
\begin{matrix}
 u_1 \dv_x\alpha_1\\
 \vdots \\
 u_1 \dv_x\alpha_{N-1}\\
 0 \\
 \vdots \\
 0 \\
 \sum_{\substack{l=1 \\ l\neq 1}}^N\Pi_{1l}(\W)\dv_x\alpha_l\\
 \vdots \\
 \sum_{\substack{l=1 \\ l\neq N}}^N\Pi_{Nl}(\W)\dv_x\alpha_l\\
 0\\
 \vdots \\
 0
\end{matrix}  \right ]\quad \mathcal{R}(\W)=
\left [
\begin{matrix}
 0\\
 \vdots \\
 0\\
 0 \\
 \vdots \\
 0 \\
 0\\
 \vdots \\
 0\\
 \alpha_1 \rho_1 (\tau_1 -\T_1)\\
 \vdots \\
 \alpha_N \rho_N (\tau_N -\T_N)
\end{matrix}  \right ].
\end{equation*}
The saturation constraint is still valid:
\begin{equation}
 \label{multi:saturation}
 \sum_{k=1}^N \alpha_k =1.
\end{equation}
For each phase $k=1,..,N$, $\tau_k=\rho_k^{-1}$ is the specific volume of phase $k$ and the pressure $\pi_k$ is a (partially) linearized pressure $\pi_k(\tau_k,\T_k)$, the equation of state of which is defined by:
\begin{equation}
\label{multi:press_relax}
\pi_k(\tau_k,\T_k)=\PP_k(\T_k) + a_k^2(\T_k-\tau_k), \quad k=1,..,N,
\end{equation}
where $\tau \mapsto \PP_k(\tau) := p_k(\tau^{-1})$ is the pressure of phase $k$ seen as a function of the specific volume $\tau=\rho^{-1}$. Accordingly with \eqref{multi:vi:pi} the relaxation interface pressure $\Pi_{kl}(\W)$ is defined by:
\begin{equation}
 \label{multi:relax:pi}
 \left \lbrace 
 \begin{array}{lll}
  \text{for $k=1$}, \quad & \Pi_{1l}(\W) = \pi_l(\tau_l,\T_l), \quad & \text{for $l=2,..,N$}\\
  \text{for $k\neq 1$}, \quad & \Pi_{kl}(\W) = \pi_k(\tau_k,\T_k),  \quad& \text{for $l=1,..,N, \, l\neq k$}.
 \end{array}
 \right.
\end{equation}

When $N=2$, system \eqref{sys:multi:relax0} is exactly the same relaxation approximation introduced in \cite{coq-13-rob} for the Baer-Nunziato two phase flow model. In the formal limit $\eps \to 0$, the additional variable $\T_k$ tends towards the specific volume $\tau_k$, and the linearized pressure law $\pi_k(\tau_k,\T_k)$ tends towards the original non-linear pressure law $p_k(\rho_k)$, thus recovering system \eqref{sys:multi0} in the first $3N-1$ equations of \eqref{sys:multi:relax0}. 
The constants $(a_k)_{k=1,..,N}$ in (\ref{multi:press_relax}) are positive parameters that must be taken large enough so as to satisfy the following sub-characteristic condition (also called Whitham's condition):
\begin{equation}
\label{whitham}
a_k^2 > -\frac{d\PP_k}{d\tau_k}(\T_k), \quad k=1,..,N,
\end{equation}
for all the values $\T_k$ encountered in the solution of \eqref{sys:multi:relax0}. Performing a Chapman-Enskog expansion, one can see that Whitham's condition expresses that system \eqref{sys:multi:relax0} is a viscous perturbation of system \eqref{sys:multi0} in the regime of small $\eps$.

\medskip
At the numerical level, a fractional step method is commonly used in the implementation of relaxation methods: the first step is a time-advancing step using the solution of the Riemann problem for the convective part of \eqref{sys:multi:relax0}:
\begin{equation}
\label{sys:multi:relax}
 \dv_t \W + \dv_x {\bf g}(\W)+ {\bf d}(\W)\dv_x\W =0, \quad x\in\R, \, t>0,
\end{equation}
while the second step consists in an instantaneous relaxation towards
the equilibrium system by imposing $\T_k = \tau_k$ in the solution
obtained by the first step. This second step is equivalent to sending
$\eps$ to $0$ instantaneously. As a consequence, we now focus on constructing an exact Riemann solver for the homogeneous convective system \eqref{sys:multi:relax}. Let us first state the main mathematical properties of system \eqref{sys:multi:relax}.

\medskip
The linearization \eqref{multi:press_relax} is designed so that system \eqref{sys:multi:relax} has only linearly degenerate fields, thus making the resolution of the Riemann problem for \eqref{sys:multi:relax} easier than for the original system \eqref{sys:multi0}. We have the following proposition:

\begin{prop}
\label{prop:spectre:multi:relax}
System \eqref{sys:multi:relax} is weakly hyperbolic on $\Omega_\W$ in the following sense. It admits the following $4N-1$ real eigenvalues: $\sigma_1(\W)=..=\sigma_{N-1}(\W)=u_1$, $\sigma_{N-1+k}(\W)=u_k-a_k \tau_k$, $\sigma_{2N-1+k}(\W)=u_k+a_k\tau_k$ and $\sigma_{3N-1+k}(\W)=u_k$ for $k=1,..,N$. All the characteristic fields are linearly degenerate and the corresponding right eigenvectors are linearly independent if, and only if,
\begin{equation}
\label{multi:hypfail}
\alpha_k \neq 0, \quad \forall k=1,..,N \qquad \text{and} \qquad |u_1-u_k| \neq a_k\tau_k, \quad \forall k=2,..,N.
\end{equation}
\end{prop}

\medskip

\begin{proof}
 The proof is given in Appendix \ref{app:sec:spectre:multi:relax}.
\end{proof}

\medskip
\begin{remark}
Here again, one never has $\alpha_k=0$ for $\W\in\Omega_\W$. However, $\alpha_k=0$ is to be understood in the sense $\alpha_k\to 0$.
\end{remark}

We also have the following properties:
\begin{prop}\label{prop:ener:multi:relax}
The smooth solutions as well as the entropy weak solutions of \eqref{sys:multi:relax} satisfy the following phasic energy equations:
\begin{align}
\label{multi:relax:ener1} &
 \textstyle \dv_t \paren{ \alpha_1\rho_1 \E_1} + \dv_x  \paren{\alpha_1\rho_1 \E_1 u_1+ \alpha_1 \pi_1 u_1} + u_1 \sum_{l=2}^N  \pi_l  \dv_x \alpha_l =0, \\[2ex]
\label{multi:relax:enerk} & \dv_t \paren{ \alpha_k\rho_k \E_k} + \dv_x  \paren{\alpha_k\rho_k \E_k u_k+ \alpha_k \pi_k u_k} - u_1 \pi_k  \dv_x \alpha_k =0, & k=2,..,N,
\end{align}
where
\begin{equation*}
 \mathcal{E}_k:= \mathcal{E}_k(u_k,\tau_k,\T_k) = \dfrac{u_k^2}{2} + e_k(\T_k) + \dfrac{\pi_k^2(\tau_k,\T_k) - \PP_k^2(\T_k)}{2a_k^2}, \qquad k=1,..,N.
\end{equation*}
Summing for $k=1,..,N$, the smooth solutions and the entropy weak solutions of \eqref{sys:multi:relax} are seen to satisfy an additional conservation law:
\begin{equation}
 \label{multi:relax:enertot}
 \textstyle \dv_t \paren{\sum_{k=1}^N \alpha_k\rho_k \E_k} + \dv_x  \paren{\sum_{k=1}^N \alpha_k\rho_k \E_k u_k+ \alpha_k \pi_k u_k}  =0. 
\end{equation}
Under Whitham's condition \eqref{whitham}, to be met for all the $\T_k$ under consideration, the following Gibbs principles are satisfied for $k=1,..,N$:
\begin{equation}
\label{gibbsAE}
 \tau_k = \argmin_{\T_k} \lbrace \mathcal{E}_k(u_k,\tau_k,{\T_k})\rbrace, \quad \text{and} \quad \mathcal{E}_k(u_k,\tau_k,{\tau_k})=E_k(u_k,\tau_k),
 \end{equation}
where $E_k(u_k,\tau_k)=u_k^2/2+e_k(\rho_k)$.
\end{prop}

\medskip

\begin{proof}
The proof of \eqref{multi:relax:ener1} and \eqref{multi:relax:enerk} follows from classical manipulations. From the phasic mass and momentum equations we first derive the evolution equation satisfied by $u_k^2/2$. We then derive an equation for $\pi_k(\tau_k,\T_k)$, using the mass equation of phase $k$ and the advection equation of $\T_k$. Combining these two equations and the fact that $\T_k$ is advected, we obtain \eqref{multi:relax:ener1} and \eqref{multi:relax:enerk}. The proof of Gibbs principle follows from an easy study of the function $\T_k \mapsto \mathcal{E}_k(u_k,\tau_k,\T_k)$.
\end{proof}

\medskip

\begin{remark}
Since all the characteristic fields of system \eqref{sys:multi:relax} are linearly degenerate, the mixture energy equation \eqref{multi:relax:enertot} is expected to be satisfied for not only smooth but also weak solutions. However, as we will see later when constructing the solutions of the Riemann problem, in the stiff cases of vanishing phases where one of the left or right phase fractions $\alpha_{k,L}$ or $\alpha_{k,R}$ is close to zero for some $k\in\lbrace 1,..,N\rbrace$, ensuring positive values of the densities of phase $k$ requires an extra dissipation of the mixture energy by the computed solution.
\end{remark}

\subsection{The relaxation Riemann problem}

Let $(\W_L,\W_R)$ be two constant states in $\Omega_\W$ and consider the Riemann problem for system \eqref{sys:multi:relax} with the following initial condition:
\begin{equation}
 \label{sys:multi:relax:CI}
   \W(x,t=0)= 
  \left \lbrace
 \begin{array}{ll}
  \W_L, & \text{if $x<0$},\\
  \W_R, & \text{if $x>0$}.
 \end{array}
\right.
 \end{equation}

\subsubsection{Definition of the solutions to the Riemann problem}

Following Proposition \ref{prop:spectre:multi:relax}, a solution to the Riemann problem \eqref{sys:multi:relax}-\eqref{sys:multi:relax:CI} is expected to be a self-similar function composed of constant intermediate states separated by waves which are contact discontinuities associated with the system's eigenvalues. Since the phase fractions are transported by the material velocity of phase 1, the non-conservative products involving the phase fraction gradients are only active across this wave and the phases are independent away from this wave. In particular, for a fixed $k$ in $\lbrace 2,..,N \rbrace$, the phase $k$ quantities may change across the contact discontinuities associated with the eigenvalues $u_k-a_k\tau_k$, $u_k$, $u_k-a_k\tau_k$ and $u_1$ and are constant across the other waves. For the applications aimed at by this work, we are only interested in solutions which have a subsonic wave ordering:
\[
 |u_1-u_k| < a_k\tau_k, \qquad \forall k=2,..,N.
\]
Consequently, in the self-similar Riemann solution, the propagation velocity $u_1$ lies in-between the acoustic waves of all the other phases. Moreover, ensuring the positivity of the phase 1 densities also requires the material wave $u_1$ to lie in between the acoustic waves of phase 1.

\begin{center}
\begin{tikzpicture}[scale=2.5]
\small
\tikzstyle{axes}=[thin,>=latex]
\begin{scope}[axes]
\draw (-1.6,0)   node {
	\begin{tikzpicture}[scale=2.5]
	\draw[->] (-1,0)--(1,0) node[right=3pt] {$x$};
        \draw[->] (0,0)--(0,1) node[left=3pt] {$t$};
	\draw [very thick,color=red] (0,0) -- (30:1cm) node[color=black,above] {$u_{1}+a_1\tau_{1}$};
        \draw [very thick,color=red] (0,0) -- (110:1cm) node[color=black,above] {$u_1$};
	\draw [very thick,color=red] (0,0) -- (155:1cm) node[color=black,above] {$u_{1}-a_1\tau_{1}$};
	\draw (0,-0.2) node[] {Wave structure for phase $1$.};
	\end{tikzpicture}
};
\draw (1.6,0)   node {
	\begin{tikzpicture}[scale=2.5]
	\draw[->] (-1,0)--(1,0) node[right=3pt] {$x$};
        \draw[->] (0,0)--(0,1) node[left=3pt] {$t$};
	\draw [very thick,color=red] (0,0) -- (30:1cm) node[color=black,above] {$u_{k}+a_k\tau_{k}$};
        \draw [very thick,color=red] (0,0) -- (80:1cm) node[color=black,above] {$u_k$};
	\draw [dashed,very thick,color=red] (0,0) -- (110:1cm) node[color=black,above] {$u_1$};
	\draw [very thick,color=red] (0,0) -- (155:1cm) node[color=black,above] {$u_{k}-a_k\tau_{k}$};
	\draw (0,-0.2) node[] {Wave structure for phase $k$, $k=2,..,N$.};%
	\end{tikzpicture}
}; 
\end{scope}
\normalsize
\end{tikzpicture}
\end{center}

We now introduce the definition of solutions to the Riemann problem \eqref{sys:multi:relax}-\eqref{sys:multi:relax:CI}, which is a straightforward extension of the definition of solutions in the case of two phase flows (see \cite[Definition 4.1]{coq-13-rob}).
\begin{definition}
 \label{def_sol}
Let $(\W_{L},\W_{R})$ be two states in $\Omega_{\W}$. A solution to the Riemann problem \eqref{sys:multi:relax}-\eqref{sys:multi:relax:CI} \textbf{with subsonic wave ordering} is a self-similar mapping $\W(x,t)=\W_{\rm Riem}(x/t;\W_{L},\W_{R})$ where the function $\xi \mapsto \W_{\rm Riem}(\xi;\W_{L},\W_{R})$ satisfies the following properties: 
\begin{enumerate}
 \item[(i)] $\W_{\rm Riem}(\xi;\W_{L},\W_{R})$ is a piecewise constant function, composed of $3N$ waves (associated with the eigenvalues $u_k-a_k\tau_k$, $u_k$ and $u_k+a_k\tau_k$ for $k=1,..,N$) separating $3N+1$ constant intermediate states belonging to $\Omega_{\W}$, and such that:
\begin{equation}
\label{limits}
\begin{aligned}
 & \xi < \min\limits_{k=1,..,N} \left \lbrace u_{k,L}-a_{k}\tau_{k,L}\right \rbrace \Longrightarrow 
 \W_{\rm Riem}(\xi;\W_{L},\W_{R})=\W_L, \\
 & \xi > \max\limits_{k=1,..,N} \left \lbrace u_{k,R}+a_{k}\tau_{k,R}\right \rbrace \Longrightarrow \W_{\rm Riem}(\xi;\W_{L},\W_{R})=\W_R.
\end{aligned}
\end{equation}
 \item[(ii)] There exist $u_1^*\in\R$ and $\Pi^*=(\pi_2^*,..,\pi_N^*)\in\R^{N-1}$ which depend on $(\W_L,\W_R)$ such that the function $\W(x,t)=\W_{\rm Riem}(x/t;\W_{L},\W_{R})$ satisfies the following system of PDEs in the distributional sense:
\begin{equation}
\label{sys:multi:relax:solve}
\dv_t \W + \dv_x {\bf g}(\W)+ {\bf D}^*(\W_L,\W_R)\delta_0(x-u_1^*t)=0,
\end{equation}
with  
$$
\begin{array}{r@{\,=\,(\,}c@{\,,\,}c@{\,,\,}c@{\,,\,}c@{\,,\,}c@{\,,\,}c@{\,,\,}c@{\,,\,}c@{\,,\,}c@{\,,\,}c@{\,,\,}c@{\,,\,}c@{\,,\,}l@{\,)^T}}
{\bf D}^*(\W_L,\W_R) & u_1^* \Delta\alpha_1 & .. & u_1^* \Delta\alpha_{N-1} & 0 & .. & 0 & \sum_{l=2}^N\pi_l^*\Delta\alpha_l & -\pi_2^*\Delta\alpha_2& .. & -\pi_N^*\Delta\alpha_N & 0 & .. & 0
\end{array}
$$
where for $k=1,..,N$, $\Delta \alpha_k=\alpha_{k,R}-\alpha_{k,L}$.
 
\item[(iii)] Furthermore, the function $\W(x,t)=\W_{\rm Riem}(x/t;\W_{L},\W_{R})$ also satisfies the following energy equations in the distributional sense:
\begin{align}
\label{multi:relax:ener1:bis} &
 \textstyle \dv_t \paren{ \alpha_1\rho_1 \E_1} + \dv_x  \paren{\alpha_1\rho_1 \E_1 u_1+ \alpha_1 \pi_1 u_1} + u_1^* \sum_{l=2}^N  \pi_l^*  \dv_x \alpha_l =0, \\[1ex]
\label{multi:relax:enerk:bis} & \dv_t \paren{ \alpha_k\rho_k \E_k} + \dv_x  \paren{\alpha_k\rho_k \E_k u_k+ \alpha_k \pi_k u_k} - u_1^* \pi_k^*  \dv_x \alpha_k =-\mathcal{Q}_k(u_1^*,\W_L,\W_R)\delta_0(x-u_1^*t), & k=2,..,N.
\end{align}
where $\mathcal{Q}_k(u_1^*,\W_L,\W_R)$ is a nonnegative number. 

\item[(iv)] The solution has a \textbf{subsonic wave ordering} in the following sense: 
\begin{equation}
 \label{subsol}
 u_{k,L}-a_k\tau_{k,L} < u_1^* < u_{k,R}+a_k\tau_{k,R}, \qquad \forall k=2,..,N.
\end{equation}
\end{enumerate}
\end{definition}

\begin{remark}
Following \eqref{multi:relax:pi}, there are $N-1$ interface pressures corresponding to the phase pressures $(\pi_2,..,\pi_N)$. Moreover, the saturation constraint \eqref{multi:saturation} gives $\sum_{l=1,l\neq k}^N\dv_x\alpha_l=-\dv_x\alpha_k$ for all $k=1,..,N$, which justifies the simplified form of the non-conservative product ${\bf D}^*(\W_L,\W_R)\delta_0(x-u_1^*t)$.
\end{remark}

\begin{remark}
\label{remark:kinetic1}
Equation \eqref{multi:relax:enerk:bis} is a relaxed version of \eqref{multi:relax:enerk} in which the energy of phase $k$ is allowed to be dissipated across the $u_1$-wave despite the linear degeneracy of the associated field. As explained in \cite{coq-13-rob} for the relaxation approximation of the Baer-Nunziato two phase flow model, allowing such a dissipation may be necessary when an initial phase fraction $\alpha_{k,L}$ or $\alpha_{k,R}$ is close to zero, in order to ensure the positivity of all the intermediate densities. 
\end{remark}

\subsubsection{The resolution strategy: a fixed-point research}

Following the ideas developed in \cite{coq-13-rob}, the resolution of the Riemann problem \eqref{sys:multi:relax}-\eqref{sys:multi:relax:CI} is based on a fixed-point research which formally amounts to iterating on a two step procedure involving the pair  $\left ( u_1^*,\Pi^* \right )\in\R\times\R^{N-1}$. We first remark that system \eqref{sys:multi:relax:solve} can be written in the following form: 
\begin{equation*}
\begin{aligned}
& \dv_t \alpha_1 + u_1^* \dv_x \alpha_1 = 0,  \\
(\mcal{S}_1) \qquad & \dv_t \paren{\alpha_1 \rho_1} + \dv_x \paren{\alpha_1 \rho_1 u_1}= 0,  \\
& \textstyle \dv_t \paren{\alpha_1 \rho_1 u_1} + \dv_x \paren{\alpha_1 \rho_1 u_1^2+\alpha_1 \pi_1(\tau_1,\T_1)}+\sum_{l=2}^N \pi_l^*\dv_x \alpha_l= 0,\\
& \dv_t \paren{\alpha_1 \rho_1 \T_1} + \dv_x \paren{\alpha_1 \rho_1 \T_1 u_1}= 0,\\[2ex]
\text{and for $k=2,..,N$:} & \\[1ex]
& \dv_t \alpha_k + u_1^* \dv_x \alpha_k = 0,  \\
(\mcal{S}_k) \qquad & \dv_t \paren{\alpha_k \rho_k} + \dv_x \paren{\alpha_k \rho_k u_k}= 0,  \\
& \textstyle \dv_t \paren{\alpha_k \rho_k u_k} + \dv_x \paren{\alpha_k \rho_k u_k^2+\alpha_k \pi_k(\tau_k,\T_k)}- \pi_k^*\dv_x \alpha_k= 0,\\
& \dv_t \paren{\alpha_k \rho_k \T_k} + \dv_x \paren{\alpha_k \rho_k \T_k u_k}= 0.
\end{aligned}
\end{equation*}
 
\bigskip

\noindent \textbf{First step:} The family of interface pressures $\Pi^*=(\pi_2^*,..,\pi_N^*)\in\R^{N-1}$ defining the non-conservative products $\pi_k^*\dv_x\alpha_k=\pi_k^* \Delta \alpha_k \delta_0(x-u_1^*t)$ for $k=2,..,N$ are assumed to be known. Hence, system $(\mcal{S}_1)$, which gathers the governing equations for phase 1, is completely independent of the other phases since the non-conservative terms can now be seen as a \emph{known source term} and the Riemann problem for $(\mcal{S}_1)$ can therefore be solved independently of the other phases. Observe that system $(\mcal{S}_1)$ is very similar to \cite[System (4.16)]{coq-13-rob} in the two phase flow framework. This Riemann problem is easily solved since $(\mcal{S}_1)$ is a hyperbolic system with a source term which is a Dirac mass supported by the kinematic wave of velocity $u_1^*$. Hence, there is no additional wave due to the source term.

\medskip
Consequently, knowing a prediction of the interface pressures $\Pi^*=(\pi_2^*,..,\pi_N^*)\in\R^{N-1}$, one can explicitly compute the value of the kinematic speed $u_1^*$ by solving the Riemann problem associated with phase 1. This first step enables to define a function:
\begin{equation}
\label{fonction_f}
\mathscr{F}[\W_{L},\W_{R};a_1]:
\left \lbrace
\begin{array}{lll}
  \R^{N-1} & \longrightarrow & \R \\
  \Pi^*=(\pi_2^*,..,\pi_N^*)& \longmapsto & u_1^*.
\end{array}
\right.
\end{equation}

\noindent \textbf{Second step:} The advection velocity $u_1^*$ of the phase fractions $\alpha_k$ is assumed to be a known constant. Thus, for all $k=2,..,N$, the governing equations for phase $k$, gathered in system $(\mcal{S}_k)$, form a system which is independent of all the other systems $(\mcal{S}_l)$ with $l=1,..,N$, $l\neq k$. 
In addition to the kinematic velocity $u_k$ and the acoustic speeds $u_k\pm a_k\tau_k$, the Riemann problem for $(\mcal{S}_k)$ involves an \textit{additional wave} whose \textit{known} constant velocity is $u_1^*$. This wave is weighted with an \textit{unknown} weight $\pi_k^* \Delta \alpha_k$ (only for the momentum equation) which is calculated by solving the Riemann problem for $(\mcal{S}_k)$ and then applying Rankine-Hugoniot's jump relation to the momentum equation for the traveling wave of speed $u_1^*$. Here again, we observe that system $(\mcal{S}_k)$ is the exact same system already solved in the two phase flow framework (see \cite[System (4.20)]{coq-13-rob}). This is what justifies the straightforward extension of the relaxation scheme to the multiphase flow model. Solving all these Riemann problems for $(\mcal{S}_k)$ with $k=2,..,N$, this second step allows to define a function :
\begin{equation}
\label{fonction_g}
\mathscr{G}[\W_{L},\W_{R};(a_k)_{k=2,..,N}]:
\left \lbrace
\begin{array}{lll}
  \R & \longrightarrow & \R^{N-1} \\
   u_1^* & \longmapsto & \Pi^*=(\pi_2^*,..,\pi_N^*)  .
\end{array}
\right.
\end{equation}

\noindent \textbf{Fixed-point research:} Performing an iterative procedure on these two steps actually boils down to the following fixed-point research.

\medskip
\textit{Find $u_1^*$ in the interval $\paren{\max\limits_{k=1,..,N} \left \lbrace u_{k,L}-a_{k}\tau_{k,L}\right \rbrace,\min\limits_{k=1,..,N} \left \lbrace u_{k,R}+a_{k}\tau_{k,R}\right \rbrace } $ such that} 
\begin{equation}
\label{fp0}
u_1^* = \Big ( \mathscr{F}[\W_{L},\W_{R};a_1] \circ  \mathscr{G}[\W_{L},\W_{R};(a_k)_{k=2,..,N}] \Big )(u_1^* ).
\end{equation}
The interval in which $u_1^*$ must be sought corresponds to the subsonic condition (\ref{subsol}) on the one hand, and to the positivity of the intermediate states of phase 1 on the other hand.

\medskip
Let us now introduce some notations which depend explicitly and solely on the initial states $(\W_L,\W_R)$ and on the relaxation parameters $(a_k)_{k=1,..,N}$. For $k=1,..,N$:
\begin{equation}
\label{diese}
\begin{aligned}
\udd_k &= \dfrac{1}{2} \left (u_{k,L}+u_{k,R} \right )-\dfrac{1}{2a_k} \left (\pi_k(\tau_{k,R},\T_{k,R}) - \pi_k(\tau_{k,L},\T_{k,L}) \right ),
\\
\pidd_k &= \dfrac{1}{2} \left ( \pi_k(\tau_{k,R},\T_{k,R}) + \pi_k(\tau_{k,L},\T_{k,L}) \right )- \dfrac{a_k}{2} \left (u_{k,R}- u_{k,L} \right ),
\\
\tdd_{k,L} &= \tau_{k,L} + \dfrac{1}{a_k}(\udd_k - u_{k,L}), 
\\
\tdd_{k,R} &= \tau_{k,R} - \dfrac{1}{a_k}(\udd_k - u_{k,R}).
\end{aligned}
\end{equation} 

\begin{remark}
\label{rem:diese}
Observe that with these definitions, $u_{k,L}-a_{k}\tau_{k,L}=\udd_k-a_k \tdd_{k,L}$ and $u_{k,R}+a_{k}\tau_{k,R}=\udd_k+a_k \tdd_{k,R}$.
\end{remark}

\subsubsection{First step of the fixed-point procedure: solving phase 1}
\label{subsec:multi:step1}
In this first step, the interface pressures $\Pi^*=(\pi_2^*,..,\pi_N^*)\in\R^{N-1}$ defining the non-conservative products $\pi_k^*\dv_x\alpha_k=\pi_k^* \Delta \alpha_k \delta_0(x-u_1^*t)$ for $k=2,..,N$ are first assumed to be known and one solves the Riemann problem for $(\mcal{S}_1)$ with the initial condition
\begin{equation}
\label{multi:step1:CI}
\W_1(x,t=0)=
\left\lbrace
\begin{array}{ll}
\W_{1,L}\qquad \textnormal{if} \qquad x<0,\\
\W_{1,R} \qquad \textnormal{if} \qquad x>0,
\end{array}
        \right.
\end{equation}
where $\W_1 = (\alpha_1,\alpha_1\rho_1,\alpha_1\rho_1 u_1,\alpha_1\rho_1\T_1)$ denotes the state vector for phase 1, and $(\W_{1,L},\W_{1,R})$ are the restriction of the complete initial data $(\W_{L},\W_{R})$ to the phase 1 variables.
Observe that the non-conservative products $\pi_k^*\dv_x\alpha_k$ are not ambiguous here since for all $k=2,..,N$, $\pi_k^*$ is a \textbf{known constant}. System $(\mcal{S}_1)$ is very similar to \cite[System (4.16)]{coq-13-rob} encountered in the two phase flow framework, and the resolution of the corresponding Riemann problem follows from the exact same steps. Therefore, we only state the main results and the reader is referred to \cite[Section 4.4]{coq-13-rob} for the proofs.

\medskip
The following proposition characterizes the convective behavior of system $(\mcal{S}_1)$.
\begin{prop}
System $(\mcal{S}_1)$ is a hyperbolic system with linearly degenerate fields associated with the eigenvalues $u_1-a_1 \tau_1$, $u_1$ and $u_1+a_1 \tau_1$. The eigenvalue $u_{1}$ has multiplicity 2.
\end{prop}

\begin{proof}
The proof is similar to that of Proposition \ref{prop:spectre:multi:relax} which is given in Appendix \ref{app:sec:spectre:multi:relax}.
\end{proof}

We have the following well-posedness result for the governing equations of phase 1. The Riemann problem $(\mcal{S}_1)$-\eqref{multi:step1:CI} differs from the Riemann problem in the two phase case in  \cite[Prop. 4.4]{coq-13-rob} only by the value of the source term $\sum_{l=2}^N \pi_l^*\dv_x \alpha_l$. Hence, the proof follows from similar steps as the proof of \cite[Prop. 4.4]{coq-13-rob}.

\begin{prop}
\label{multi:prop:step1}
Assume that the parameter $a_1$ is such that $\tdd_{1,L} >0$ and $\tdd_{1,R} >0$. Then the Riemann problem $(\mcal{S}_1)$-\eqref{multi:step1:CI} admits a unique solution the intermediate states of which are defined by:
\begin{center}
\begin{tikzpicture}[scale=3]
\small
\tikzstyle{axes}=[thin,>=latex]
\begin{scope}[axes]
        \draw[->] (-1,0)--(1,0) node[right=3pt] {$x$};
        \draw[->] (0,0)--(0,1) node[left=3pt] {$t$};
	\draw [very thick,color=red] (0,0) -- (20:1cm) node [color=black, above=1pt] {$\qquad u_{1,R}+a_1\tau_{1,R}$};
        \draw [very thick,color=red] (0,0) -- (80:0.8cm) node [color=black, above=2pt] {$u_1^*$};
	\draw [very thick,color=red] (0,0) -- (150:1cm) node [color=black, above=1pt] {$u_{1,L}-a_1\tau_{1,L}$};
	\draw (-0.6,0.1) node[] {$\W_{1,L}$};
	\draw (-0.25,0.6) node[] {$\W_{1}^-$};
	\draw (0.4,0.6) node[] {$\W_{1}^+$};
	\draw (0.7,0.1) node[] {$\W_{1,R}$};
\end{scope}
\normalsize
\end{tikzpicture}
\end{center}

\begin{equation}
\label{eq:intstate:1}
\begin{array}{llll}
\alpha_{1}^- = \alpha_{1,L},
&\quad \tau_1^- = \tdd_{1,L} +\frac{1}{a_1} (u_1^*-\udd_1),
&\quad u_1^+=u_1^*,
&\quad\T_1^- = \T_{1,L}, \\[2ex]
\alpha_{1}^+ = \alpha_{1,R},
&\quad \tau_1^+ = \tdd_{1,R} -\frac{1}{a_1} (u_1^*-\udd_1),
&\quad u_1^+=u_1^*,
&\quad \T_1^+ = \T_{1,R},
\end{array}
\end{equation}
where
\begin{equation}
\label{multi:u1star}
 u_1^* =  \udd_1 + \dfrac{1}{a_1(\alpha_{1,L}+\alpha_{1,R})} \sum_{l=2}^N(\pidd_1-\pi_l^*)\Delta \alpha_l.
\end{equation}
The intermediate densities $\rho_{1}^-$ and $\rho_{1}^+$ are positive if and only if
\begin{equation}
\label{pos1}
 u_{1,L}-a_1\tau_{1,L} < u_1^* < u_{1,R}+a_1\tau_{1,R}.
\end{equation}
Moreover, this unique solution satisfies the following energy equation in the usual weak sense:
\begin{equation}
\label{enerphi1}
 \textstyle \dv_t \paren{ \alpha_1\rho_1 \E_1} + \dv_x  \paren{\alpha_1\rho_1 \E_1 u_1+ \alpha_1 \pi_1 u_1} + u_1^* \sum_{l=2}^N  \pi_l^*  \dv_x \alpha_l =0.
\end{equation} 
\end{prop}

\begin{proof}
The proof consists in solving a system with the unknowns $\W_1^-$ and $\W_1^+$. All the fields are linearly degenerated. The components of these state vectors are linked to each other through the Riemann invariants of the $u_1^*$-wave. The quantity $\sum_{l=2}^N \pi_l^*\dv_x \alpha_l$ is a given source term in the right hand side of the Rankine-Hugoniot relation associated with the momentum equation for this wave. $\W_1^-$ and $\W_{1,L}$ are connected through the Riemann invariants of the $\lbrace u_1-a_1\tau_1 \rbrace$-wave.  $\W_1^+$ and $\W_{1,R}$ are connected through the Riemann invariants of the $\lbrace u_1+a_1\tau_1\rbrace$-wave. We refer to \cite[Prop. 4.4]{coq-13-rob} for more details.
\end{proof}

The expression of $u_1^*$ given in equation \eqref{multi:u1star} defines the function $\mathscr{F}[\W_{L},\W_{R};a_1]$ introduced in \eqref{fonction_f}, since $u_1^*$ is expressed as a function of $\Pi^*=(\pi_2^*,..,\pi_N^*)\in\R^{N-1}$. A convenient reformulation of \eqref{multi:u1star} is the following:
\begin{equation}
\label{multi:step1:jumpmom}
\sum_{l=2}^N \pi_l^* \Delta \alpha_l =  \sum_{l=2}^N \pidd_1 \Delta \alpha_l - \theta_1(u_1^*), 
\end{equation}
where
\begin{equation}
\label{multi:def:theta1}
\theta_1(u_1^*) =  a_1\paren{\alpha_{1,L}+\alpha_{1,R}} (u_1^* -  \udd_1). 
\end{equation}

\subsubsection{Second step of the fixed-point procedure: solving phase $k$, for $k=2,..,N$}
\label{subsec:multi:stepk}
In this second step, the transport velocity $u_1^*$ of the phase fractions $\alpha_k$ is assumed to be known, while the vector of interface pressures $\Pi^*=(\pi_2^*,..,\pi_N^*)\in\R^{N-1}$ defining the non-conservative products $\pi_k \dv_x \alpha_k = \pi_k^* \Delta \alpha_k \delta_0(x-u_1^*t)$ is an unknown that must be calculated by solving the $N-1$ independent Riemann problems for $(\mcal{S}_k)$, for $k=2,..,N$ with the initial condition
\begin{equation}
\label{multi:stepk:CI}
\W_k(x,t=0)=
\left\lbrace
\begin{array}{ll}
\W_{k,L}\qquad \textnormal{if} \qquad x<0,\\
\W_{k,R} \qquad \textnormal{if} \qquad x>0,
\end{array}
        \right.
\end{equation}
where $\W_k = (\alpha_k,\alpha_k\rho_k,\alpha_k\rho_k u_k,\alpha_k\rho_k\T_k)$ denotes the state vector for phase $k$, and $(\W_{k,L},\W_{k,R})$ are the restriction of the complete initial data $(\W_{L},\W_{R})$ to the phase $k$ variables. System $(\mcal{S}_k)$ is the exact same system as \cite[System (4.20)]{coq-13-rob} encountered in the two phase flow framework, and the resolution of the corresponding Riemann problem follows from the exact same steps. Here again, we only state the main results and the reader is referred to \cite[Section 4.5]{coq-13-rob} for the detailed proofs.

\medskip
Once the resolution is done, applying Rankine-Hugoniot's jump relation to the momentum equation of $(\mcal{S}_k)$ gives the expression of $\pi_k^* \Delta \alpha_k$. 

\begin{prop}
For every $k=2,..,N$, system $(\mcal{S}_k)$ admits four real eigenvalues that are $u_k-a_k \tau_k$, $u_k$, $u_k+a_k \tau_k$ and $u_1^*$. All the fields are linearly degenerate and the system is hyperbolic if, and only if $|u_k-u_1^*| \neq  a_k\tau_k$. 
\end{prop}

\begin{proof}
The proof is similar to that of Proposition \ref{prop:spectre:multi:relax} which is given in Appendix \ref{app:sec:spectre:multi:relax}.
\end{proof}

We search for Riemann solutions which comply with the subsonic relative speed constraint \eqref{subsol}. Such solutions are of three types depending on the relative wave ordering between the eigenvalues $u_k$ and $u_1^*$:
\begin{center}
\begin{tikzpicture}[scale=2.5]
\small
\tikzstyle{axes}=[thin,>=latex]
\begin{scope}[axes]
\draw (-2.3,0)   node {
	\begin{tikzpicture}[scale=2]
	\draw[->] (-1,0)--(1,0) node[right=3pt] {$x$};
        \draw[->] (0,0)--(0,1) node[left=3pt] {$t$};
	\draw [very thick,color=red] (0,0) -- (30:1cm) node[color=black,above] {$u_{k,R}+a_k\tau_{k,R}$};
        \draw [very thick,color=red] (0,0) -- (80:1cm) node[color=black,above] {$u_k^*$};
	\draw [dashed,very thick,color=red] (0,0) -- (120:1cm) node[color=black,above] {$u_1^*$};
	\draw [very thick,color=red] (0,0) -- (155:1cm) node[color=black,above] {$u_{k,L}-a_k\tau_{k,L}$};
	\draw (-0.6,0.1) node[] {$\W_{k,L}$};
	\draw (-0.4,0.4) node[] {$\W_{k}^-$};
	\draw (-0.15,0.6) node[] {$\W_{k}^+$};
	\draw (0.3,0.4) node[] {$\W_{k,R*}$};
	\draw (0.7,0.1) node[] {$\W_{k,R}$};
	\draw (0,-0.2) node[] {Wave ordering $u_k>u_1$};%
	\end{tikzpicture}
};
\draw (0,0)   node {
	\begin{tikzpicture}[scale=2]
	\draw[->] (-1,0)--(1,0) node[right=3pt] {$x$};
        \draw[->] (0,0)--(0,1) node[left=3pt] {$t$};
	\draw [very thick,color=red] (0,0) -- (30:1cm) node[color=black,above] {$u_{k,R}+a_k\tau_{k,R}$};
        \draw [very thick,color=red] (0,0) -- (120:1cm) node[color=black,above] {$u_k^*$};
	\draw [dashed,very thick,color=red] (0,0) -- (80:1cm) node[color=black,above] {$u_1^*$};
	\draw [very thick,color=red] (0,0) -- (155:1cm) node[color=black,above] {$u_{k,L}-a_k\tau_{k,L}$};
	\draw (-0.6,0.1) node[] {$\W_{k,L}$};
	\draw (-0.1,0.6) node[] {$\W_{k}^-$};
	\draw (0.3,0.4) node[] {$\W_{k}^+$};
	\draw (-0.35,0.3) node[] {$\W_{k,L*}$};
	\draw (0.7,0.1) node[] {$\W_{k,R}$};
	\draw (0,-0.2) node[] {Wave ordering $u_k<u_1$};%
	\end{tikzpicture}
};
\draw (2.3,0)   node {
	\begin{tikzpicture}[scale=2]
	\draw[->] (-1,0)--(1,0) node[right=3pt] {$x$};
        \draw[->] (0,0)--(0,1) node[left=3pt] {$t$};
	\draw [very thick,color=red] (0,0) -- (30:1cm) node[color=black,above] {$u_{k,R}+a_k\tau_{k,R}$};
        \draw [very thick,color=red] (0,0) -- (80:1cm) node[color=black,above] {$u_k^*=u_1^*$};
	\draw [very thick,color=red] (0,0) -- (155:1cm) node[color=black,above] {$u_{k,L}-a_k\tau_{k,L}$};
	\draw (-0.6,0.1) node[] {$\W_{k,L}$};
	\draw (-0.15,0.6) node[] {$\W_{k}^-$};
	\draw (0.3,0.4) node[] {$\W_{k}^+$};
	\draw (0.7,0.1) node[] {$\W_{k,R}$};
	\draw (0,-0.2) node[] {Wave ordering $u_k=u_1$};%
	\end{tikzpicture}
};
\end{scope}
\normalsize
\end{tikzpicture}
\end{center}

We denote $\xi\mapsto\W_k(\xi;\W_{k,L},\W_{k,R})$ the self-similar mapping defining this solution and we may now recall the following result stated in \cite[Prop. 4.8]{coq-13-rob} in the framework of the Baer-Nunziato two phase flow model.
\begin{prop}
\label{multi:prop:stepk}
Assume that $a_k$ is such that $\tdd_{k,L} >0$ and $\tdd_{k,R} >0$ and define for $(\nu,\omega)\in\R_+^{*}\times\R_+^*$ the two-variable function:
\begin{equation}
\label{thefamousM0}
\M_0(\nu,\omega) = \dfrac{1}{2} \left( \dfrac{1+\omega^2}{1-\omega^2} \left(
1+
\dfrac{1}{\nu} \right ) - \sqrt{\left(
\dfrac{1+\omega^2}{1-\omega^2}\right )^2 \left( 1+ \dfrac{1}{\nu}
\right )^2 -\dfrac{4}{\nu} }\right ), 
\end{equation}
which can be extended by continuity to $\omega=1$ by setting $\M_0(\nu,1)=0$.
\begin{itemize}
 \item The Riemann problem $(\mcal{S}_k)$-\eqref{multi:stepk:CI} admits self-similar solutions $\xi\mapsto\W_k(\xi;\W_{k,L},\W_{k,R})$ with the subsonic wave ordering $u_k-a_k \tau_k < u_1^*< u_k < u_k+a_k \tau_k$, if and only if
\begin{equation}
\udd_k-u_1^* \geq 0 \qquad \text{and} \qquad \udd_k-a_k \tdd_{k,L} < u_1^*.
\end{equation}
These solutions are parametrized by a real number $\M$
and the intermediate states are given by:
\begin{equation}
\label{eq:intstate:k}
\begin{aligned}
\tau_k^- &= \tdd_{k,L} \dfrac{ 1 - \Me_k}{1 -\M},   &u_k^- =& \ u_1^*+ a_k\M
\tau_k^-,
&\T_k^- =& \ \T_{k,L},  \\
\tau_k^+ &=\tdd_{k,L} \dfrac{ 1 + \Me_k }{1 + \nu_k \M},   &u_k^+ =& \ u_1^*+\ \nu_k
a_k \M
\tau_k^+, &\T_k^+ =& \ \T_{k,L},  \\
\tau_{k,R*} &= \tdd_{k,R} + \tdd_{k,L} \dfrac{\Me_k - \nu_k \M}{1 + \nu_k \M},
&u_{k,R*} =&  \ u_1^*+ \nu_k a_k \M \tau_k^+,  &\T_{k,R*} =& \ \T_{k,R}. 
\end{aligned}
\end{equation}
where
$\nu_k = \dfrac{\alpha_{k,L}}{\alpha_{k,R}}$, $\Me_k=\dfrac{\udd_k-u_1^*}{a_k \tdd_{k,L}}$, and $\M$ varies in the interval $(0 ,\M_0(\nu_k,\omega_k)]$ where $\omega_k=\dfrac{1-\Me_k}{1+\Me_k}$.
These solutions satisfy the following energy identity:
\begin{equation}
\label{enerphik}
\dv_t \paren{ \alpha_k\rho_k \E_k} + \dv_x  \paren{\alpha_k\rho_k \E_k u_k+ \alpha_k \pi_k u_k} - u_1^* \pi_k^*  \dv_x \alpha_k =-\mathcal{Q}_k(u_1^*,\W_L,\W_R)\delta_0(x-u_1^*t),
\end{equation}
where $\mathcal{Q}_k(u_1^*,\W_L,\W_R)\geq 0$. When $\M = \M_0(\nu_k,\omega_k)$ one has $\mathcal{Q}_k(u_1^*,\W_L,\W_R)= 0$  and when $0<\M < \M_0(\nu_k,\omega_k)$ one has $\mathcal{Q}_k(u_1^*,\W_L,\W_R) > 0$. 

\item The Riemann problem $(\mcal{S}_k)$-\eqref{multi:stepk:CI} admits solutions with the subsonic wave ordering $u_k-a_k \tau_k <  u_k < u_1^* < u_k+a_k \tau_k$, if and only if
\begin{equation}
\udd_k-u_1^* \leq 0 \qquad \text{and} \qquad \udd_k-a_k \tdd_{k,L} < u_1^*.
\end{equation}
By the Galilean invariance of system $(\mcal{S}_k)$ written in the mowing frame of speed $u_1^*$ (see \cite[System (4.33)]{coq-13-rob}), such a solution is given by 
$
\xi\mapsto \mcal{V}\W_k(2u_1^*-\xi;\mcal{V}\W_{k,R},\mcal{V}\W_{k,L})
$
where the operator $\mcal{V}$ changes the relative velocities $u_k-u_1^*$ into their opposite values 
\[
\mcal{V}: (\alpha_k,\alpha_k\rho_k,\alpha_k\rho_k u_k,\alpha_k\rho_k\T_k) \mapsto (\alpha_k,\alpha_k\rho_k,\alpha_k\rho_k (2u_1^*-u_k),\alpha_k\rho_k\T_k).
\]
The solution also satisfies an energy identity similar to \eqref{enerphik}.

\item The Riemann problem $(\mcal{S}_k)$-\eqref{multi:stepk:CI} admits solutions with the subsonic wave ordering $u_k-a_k \tau_k <  u_k = u_1^* < u_k+a_k \tau_k$, if and only if
\begin{equation}
\udd_k-u_1^* =0.
\end{equation}
The intermediate states are obtained by passing to the limit as $\M_k^*\to 0$ in the expressions given in the case of the wave ordering $u_1^*<u_k$.
The solution also satisfies an energy identity similar to \eqref{enerphik}.
\end{itemize}
\end{prop}

\medskip
\begin{proof}
The proof consists in solving a system with the unknowns $\W_k^-$ and $\W_k^+$ and $\W_{k,R*}$ (or $\W_{k,L*}$). All the fields are linearly degenerated. The components of $\W_k^-$ and $\W_k^+$ are linked to each other through the Riemann invariants of the $u_1^*$-wave. $\W_k^-$ and $\W_{k,L}$ are connected through the Riemann invariants of the $\lbrace u_k-a_k\tau_k\rbrace $-wave. $\W_k^+$ and $\W_{k,R*}$ are connected through the Riemann invariants of the $u_k$-wave. $\W_{k,R*}$ and $\W_{k,R}$ are connected through the Riemann invariants of the $\lbrace u_k+a_k\tau_k \rbrace$-wave. We refer to \cite[Prop. 4.8]{coq-13-rob} for more details. 
\end{proof}

\medskip

\begin{remark}
As mentioned in Remark \ref{remark:kinetic1}, taking $\mathcal{Q}_k(u_1^*,\W_L,\W_R)>0$ may be necessary when an initial phase fraction $\alpha_{k,L}$ or $\alpha_{k,R}$ is close to zero, in order to ensure the positivity of all the intermediate states densities. In the case of the wave ordering $u_k-a_k \tau_k < u_1^*< u_k < u_k+a_k \tau_k$ for instance, ensuring the positivity of $\tau_{k,R*}$ in the regime $\nu_k >> 1$ (\emph{i.e.} $\alpha_{k,R}\to 0$) may require taking $0<\M < \M_0(\nu_k,\omega_k)$ which implies $\mathcal{Q}_k(u_1^*,\W_L,\W_R) > 0$. The precise choice of $\M$ in the interval $(0,\M_0(\nu_k,\omega_k))$ made in \cite[Section 4.5.2]{coq-13-rob} to ensure the positivity of $\tau_{k,R*}$ is still valid and will not be detailed here.
\end{remark}

Now that the solution of the Riemann problem $(\mcal{S}_k)$-\eqref{multi:stepk:CI} has been computed, we may
apply Rankine-Hugoniot's jump relation to the momentum equation of $(\mcal{S}_k)$ in order to determine the expression of the non-conservative product $\pi_k^* \Delta \alpha_k$ with respect to the given parameter $u_1^*$, thus defining the function $\mathscr{G}[\W_{L},\W_{R};(a_k)_{k=2,..,N}]$ introduced in \eqref{fonction_g}. We obtain the following expression which is directly taken form \cite[Eq. (4.51)]{coq-13-rob}):
\begin{equation}
\label{multi:stepk:jumpmom}
\pi_k^* \Delta \alpha_k =  \pidd_k \Delta \alpha_k + \theta_k(u_1^*),
\end{equation}
where the function $\theta_k$ is defined by:
\begin{multline}
\label{multi:def:thetak}
\theta_k(u_1^*) =  a_k\paren{\alpha_{k,L}+\alpha_{k,R}} (u_1^* -  \udd_k) \\
+2a_k^2
\left \lbrace
\begin{array}{ll}
\alpha_{k,L} \, \tdd_{k,L} \, \M_0\paren{\frac{\alpha_{k,L}}{\alpha_{k,R}},\dfrac{1-\Me_k}{1+\Me_k}}, & \quad \text{with \ $\Me_k=\frac{\udd_k-u_1^*}{a_k\tdd_{k,L}}$  \ if \ $\udd_k\geq u_1^*$},\\[3ex] 
\alpha_{k,R} \, \tdd_{k,R} \, \M_0\paren{\frac{\alpha_{k,R}}{\alpha_{k,L}},\dfrac{1-\Me_k}{1+\Me_k}}, & \quad \text{with \ $\Me_k=\frac{\udd_k-u_1^*}{a_k\tdd_{k,R}}$ \ if \ $\udd_k\leq u_1^*$}.
\end{array}
\right.
\end{multline}

\medskip
\begin{remark}
\label{rem:kinetic2}
This function $\theta_k$ corresponds to an energy preserving solution (with $\mathcal{Q}_k(u_1^*,\W_L,\W_R)=0$) assuming all the intermediate densities are positive. If one has to dissipate energy in order to ensure the positivity of the densities, function $\theta_k$ must be slightly modified (see \cite[Section 4.5.2]{coq-13-rob} for more details). 
\end{remark}

\subsubsection{Solution to the fixed-point problem}
Solving the fixed-point \eqref{fp0} amounts to \textit{re-coupling} phase 1 with the other phases which have been decoupled for a separate resolution. This is done by equalizing the two expressions obtained for $\sum_{k=2}^N \pi_k^* \Delta \alpha_k$ in the first step \eqref{multi:step1:jumpmom} on the one hand and in the second step \eqref{multi:stepk:jumpmom} (after summation over $k=2,..,N$) on the other hand. We obtain that $u_1^*$ must solve the following scalar fixed-point problem : 
\begin{equation}
\label{fpTheta}
 \Theta(u_1^*)=\sum_{k=2}^N (\pidd_1-\pidd_k) \Delta \alpha_k,
\end{equation}
where the function $\Theta$ is defined by $\Theta(u)= \theta_1(u)+..+\theta_N(u)$.

\medskip
We have the following theorem which states a necessary and sufficient condition for the existence of solutions to the Riemann problem \eqref{sys:multi:relax}-\eqref{sys:multi:relax:CI}. One important fact is that this condition can be \emph{explicitly tested} against the initial data $(\W_L,\W_R)$. 

\begin{theorem}
\label{multi:THEtheorem}
Let be given a pair of admissible initial states $(\W_{L},\W_{R}) \in \Omega_{\W} \times \Omega_{\W}$ and assume that the parameter $a_k$ is such that $\tdd_{k,L} >0$ and $\tdd_{k,R} > 0$ for all $k=1,..,N$. The Riemann problem \eqref{sys:multi:relax}-\eqref{sys:multi:relax:CI} admits  solutions in the sense of Definition \ref{def_sol} if, and only if, the following condition holds:
\begin{equation}
\label{TheCondition}
  \Theta\Big(\max\limits_{k=1,..,N} \left \lbrace u_{k,L}-a_{k}\tau_{k,L}\right \rbrace \Big )< \sum_{k=2}^N (\pidd_1-\pidd_k) \Delta \alpha_k < \Theta\Big(\min\limits_{k=1,..,N} \left \lbrace u_{k,R}+a_{k}\tau_{k,R}\right \rbrace\Big).
\end{equation}
The intermediate states of this solution are given in Propositions \ref{multi:prop:step1} and \ref{multi:prop:stepk}  where $u_1^*$ is the unique real number in the interval $\Big(\max\limits_{k=1,..,N} \left \lbrace u_{k,L}-a_{k}\tau_{k,L}\right \rbrace,\min\limits_{k=1,..,N} \left \lbrace u_{k,R}+a_{k}\tau_{k,R}\right \rbrace \Big)$ satisfying \eqref{fpTheta}. 
\end{theorem}

\begin{remark}
\label{rem:akgrand}
When simulating real industrial applications, the relaxation Riemann solver used for the convective effects will be associated with another step for the treatment of zero-th order terms enforcing the return to pressure, velocity (and possibly temperature) equilibrium between the phases. Hence, the pressure disequilibrium between the phases in the initial states is usually expected to be small, which yields small values of the quantities $\pidd_1-\pidd_k$. Hence, in most applications, condition \eqref{TheCondition} is expected to be satisfied. However, even away from pressure equilibrium, it is easy to observe that assumption \eqref{TheCondition} is always satisfied if the parameters $(a_k)_{k=1,..,N}$ are taken large enough. Indeed, denoting $\bfa=(a_1,..,a_N)$, one can prove that:
\[
 \begin{aligned}
  &\Theta\Big(\max\limits_{k=1,..,N} \left \lbrace u_{k,L}-a_{k}\tau_{k,L}\right \rbrace \Big ) \underset{|\bfa|\to+\infty}{\leq} -C_L|\bfa|^2, \\
  &\sum_{k=2}^N (\pidd_1-\pidd_k) \Delta \alpha_k \underset{|\bfa|\to+\infty}{=} \mathcal{O}(|\bfa|),\\
  &\Theta\Big(\min\limits_{k=1,..,N} \left \lbrace u_{k,R}+a_{k}\tau_{k,R}\right \rbrace\Big)\underset{|\bfa|\to+\infty}{\geq} C_R|\bfa|^2,
 \end{aligned}
\]

where $C_L$ and $C_R$ are two positive constants depending on $(\W_L,\W_R)$.
\end{remark}

\begin{proof}[Proof of Theorem \ref{multi:THEtheorem}]
In order to ease the notations, let us denote 
\[
  c_L= \max\limits_{k=1,..,N} \left \lbrace u_{k,L}-a_{k}\tau_{k,L}\right \rbrace, \quad \text{and} \quad c_R=\min\limits_{k=1,..,N} \left \lbrace u_{k,R}+a_{k}\tau_{k,R}\right \rbrace.
\]
Let us prove that each of the functions $\theta_k$ is a continuous and strictly increasing function on the open interval $(c_L,c_R)$.
The function $\theta_1$ defined in \eqref{multi:def:theta1} is clearly continuous and strictly increasing on this interval. Let us now consider $\theta_k$ for some $k\in\lbrace 2,..,N \rbrace$. We only consider the energy preserving case for which $\theta_k$ is defined in \eqref{multi:def:thetak} (see Remark \ref{rem:kinetic2} and \cite[Section 4.6.2]{coq-13-rob} for the general case). For $u_1^*\in(c_L,c_R)$, we have $\Me_k\in(-1,1)$ (see \eqref{multi:def:thetak} and Remark \eqref{rem:diese}) and therefore $w_k=(1-\Me_k)/(1+\Me_k)\in(0,+\infty)$. The function $\omega\mapsto\M_0(\nu,\omega)$ defined in \eqref{thefamousM0} is continuous on $\R_+^*$, which implies that the function $\theta_k$ defined in \eqref{multi:def:thetak} is continuous on the interval $(c_L,c_R)$. Let us now differentiate $\theta_k$.

For $u_1^*\in(c_L,\udd_k)$, we have, denoting $\nu_k = \dfrac{\alpha_{k,L}}{\alpha_{k,R}}$, $\Me_k=\dfrac{\udd_k-u_1^*}{a_k \tdd_{k,L}}$, $\omega_k=\dfrac{1-\Me_k}{1+\Me_k}$:
\[
 \theta_k'(u_1^*) 
= a_k\paren{\alpha_{k,L}+\alpha_{k,R}} + 2 a_k^2 \, \alpha_{k,L} \, \tdd_{k,L}  \frac{\dv\M_0}{\dv \omega} (\nu_k,\omega_k) \cdot \dfrac{d\omega_k}{d\Me_k} \cdot \dfrac{d\Me_k}{d u_1^*} \\
\]
with $\dfrac{d\omega_k}{d\Me_k}=-\dfrac{2}{(1+\Me_k)^2}=-\dfrac{(1+\omega_k)^2}{2}$ and $\dfrac{d\Me_k}{d u_1^*}=-\dfrac{1}{a_k \tdd_{k,L}}$. Hence, we obtain:
\[
  \frac{1}{a_k \, \alpha_{k,R}}\theta_k'(u_1^*) = 1+\nu_k+\nu_k (1+\omega_k)^2 \frac{\dv\M_0}{\dv \omega} (\nu_k,\omega_k).
\]
It is not difficult to prove that the right hand side of this equality is positive which implies that $\theta_k$ is strictly increasing on the interval $(c_L,\udd_k)$. Actually this exact computation has already been done in the framework of the Baer-Nunziato model (see \cite[Eq. (4.65)]{coq-13-rob} for the details). A similar computation proves that $\theta_k'$ is also positive on the interval $(\udd_k,c_R)$. We obtain that all the functions $\theta_k$ are continuous and strictly increasing on the open interval $(c_L,c_R)$ and so is $\Theta= \theta_1+..+\theta_N$. The result of Theorem \ref{multi:THEtheorem} follows from the intermediate value theorem.

\end{proof}

\subsection{The relaxation finite volume scheme and its properties}

We now derive a finite volume scheme for the approximation of the entropy weak solutions of a Cauchy problem associated with system \eqref{sys:multi0}.
For simplicity in the notations, we assume constant positive time and space steps $\Delta
t$ and $\Delta x$. The space is partitioned into cells $\R=\bigcup_{j \in \Z}  C_j$ where $C_j= [x_{j-\frac{1}{2}},x_{j+\frac{1}{2}}[$ with
$x_{j+\frac{1}{2}}=(j+\frac 12) \Delta x$ for all $j$ in $\Z$. 
We also introduce the discrete intermediate times $t^{n}=n\Delta t, \ n
\in \N$. The approximate solution at time $t^{n}$ is a piecewise constant function whose
value on each cell $C_{j}$ is a
constant value denoted by $\U_{j}^{n}$. We assume that $\Delta t$ and $\Delta x$ satisfy the CFL condition :
\begin{equation}
\label{cfl}
\frac{\Delta t}{\Delta x}
\ \max\limits_{k=1,..,N}
\ \max\limits_{j\in\Z}
\ \max 
\left \lbrace |
(u_{k}-a_{k}\tau_{k})^n_j|,|
(u_{k}+a_{k}\tau_{k})^n_{j+1}|\right \rbrace < \frac{1}{2}.
\end{equation}

The Finite Volume relaxation scheme reads (see \cite{coq-13-rob} for more details):
\begin{equation}
\label{fvscheme}
\U_{j}^{n+1} = \U_{j}^{n} - \dfrac{\Delta t}{\Delta x}
\left (\mathbf{F}^{-}(\U_{j}^{n},\U_{j+1}^{n}) -
\mathbf{F}^{+}(\U_{j-1}^{n},\U_{j}^{n})  \right),
\end{equation}
where the numerical fluxes are computed thanks to the exact Riemann solver $\W_{\rm Riem}(\xi;\W_{L},\W_{R})$ constructed for the relaxation system:
\begin{equation*}
\begin{aligned}
&\mathbf{F}^{-}(\vect{U}_{L},\vect{U}_{R}) = \mathscr{P} 
\textbf{g} \left (\W_{\rm Riem} \left (0^-;\mathscr{M}(\vect{U}_L),\mathscr{M}(\vect{U}_R)\right) \right ) +\mathscr{P} 
\mathbf{D}^* \left(\mathscr{M}(\vect{U}_L),\mathscr{M}(\vect{U}_R)\right) \Ind_{\left \lbrace u_{1}^* < 0 \right \rbrace}, \\
&\mathbf{F}^{+}(\vect{U}_{L},\vect{U}_{R}) = \mathscr{P} 
\textbf{g} \left (\W_{\rm Riem} \left (0^+;\mathscr{M}(\vect{U}_L),\mathscr{M}(\vect{U}_R)\right) \right ) -\mathscr{P} 
{\bf D}^* \left(\mathscr{M}(\vect{U}_L),\mathscr{M}(\vect{U}_R)\right) \Ind_{\left \lbrace u_{1}^* > 0 \right \rbrace}.
\end{aligned}
\end{equation*}
The non-conservative part of the flux ${\bf D}^*(\W_L,\W_R)$ is defined in Definition \ref{def_sol} and the mappings $\mathscr{M}$ and $\mathscr{P}$ are given by:
\begin{align*}
& \mathscr{M}: \left \lbrace \begin{array}{cccl}
		  &  \Omega_\U & \longrightarrow & \Omega_\W  \\
		  &  (x_i)_{i=1,..,3N-1} & \longmapsto & \paren{x_1,x_2,..,x_{3N-1},x_1,x_2,..,x_{N-1},1-\sum_{i=1}^{N-1}x_i}. 
\end{array} \right. \\[2ex]
& \mathscr{P}: \left \lbrace \begin{array}{cccl}
		  &  \Omega_\W & \longrightarrow & \Omega_\U  \\
		  &  (x_i)_{i=1,..,4N-1} & \longmapsto & (x_1,x_2,..,x_{3N-1}). 
\end{array} \right.
\end{align*}

\medskip

At each interface $x_{j+\frac 12}$, the relaxation Riemann solver $\W_{\rm Riem}(\xi;\mathscr{M}(\vect{U}_{j}^n),\mathscr{M}(\vect{U}_{j+1}^n))$ depends on the family of relaxation parameters $(a_k)_{k=1,..,N}$ which must be chosen so as to ensure the conditions stated in the existence Theorem \ref{multi:THEtheorem}, and to satisfy some stability properties. Observe that one might take different relaxation parameters $(a_k)_{k=1,..,N}$ for each interface, which amounts to approximating system \eqref{sys:multi0} by a different relaxation approximation at each interface, which is more or less diffusive depending on how large are the local parameters. Further discussion on the practical computation of these parameters is postponed to the appendix \ref{sec:flux}, as well as the detailed description of the computation of the numerical fluxes $\mathbf{F}^{\pm}(\vect{U}_{L},\vect{U}_{R})$.

\medskip
\begin{remark}[The method is valid for all barotropic e.o.s.]
\label{rem:any eos}
The Riemann solution $\W_{\rm Riem}(\xi;\W_L,\W_R)$ only depends on the quantities $\udd_k$, $\pidd_k$, $\tdd_{k,L}$ and $\tdd_{k,R}$ defined in \eqref{diese} and on the left and right phase fractions $\alpha_{k,L}$ and $\alpha_{k,R}$ for $k=1,..,N$. Indeed, the solution of the fixed-point problem \eqref{fpTheta} only depends on these quantities and so do the intermediate states (see \eqref{eq:intstate:1} and \eqref{eq:intstate:k}). Therefore, the dependence of the Riemann solution $\W_{\rm Riem}(\xi;\W_L,\W_R)$ on the barotropic equation of state occurs only through the computation of $\pi_k(\tau_{k,L},\T_{k,L})$ and $\pi_k(\tau_{k,R},\T_{k,R})$. For $(\W_L,\W_R)=(\mathscr{M}(\vect{U}_j^n),\mathscr{M}(\vect{U}_{j+1}^n))$, we have $\T_{k,L}=(\tau_{k})_j^n,$ and $\T_{k,R}=(\tau_{k})_{j+1}^n$ and thus $\pi_k(\tau_{k,L},\T_{k,L})=p_k((\rho_{k})_j^n)$ and $\pi_k(\tau_{k,R},\T_{k,R})=p_k((\rho_{k})_{j+1}^n)$ for all $k=1,..,N$. These quantities can be computed for any barotropic e.o.s. at the beginning of each time step.
\end{remark}

\medskip
We may now state the following theorem, which gathers the main properties of this scheme, and which constitutes the main result of the paper.

\begin{theorem}
The finite volume scheme \eqref{fvscheme} for the multiphase flow model has the following properties:
 
 \begin{itemize}
  \item \textbf{Positivity:} Under the CFL condition \eqref{cfl}, the scheme preserves positive values of the phase fractions and densities: for all $n\in\N$, if ($\U_j^n\in\Omega_{\U}$ for all $j\in\Z$), then $0 <(\alpha_{k})_j^{n+1} < 1$ and $(\alpha_k \rho_k)_j^{n+1} >0$ for all $k=1,..,N$ and all $j\in\Z$, \emph{i.e.} ($\U_j^{n+1}\in \Omega_{\U}$ for all $j\in\Z$).

  \item \textbf{Conservativity:} The discretizations of the partial masses $\alpha_k\rho_k,\,k=1,..,N$, and the total mixture momentum $\sum_{k=1}^N \alpha_k\rho_k u_k$ are conservative.
  
  \item \textbf{Discrete energy inequalities.} Assume that the relaxation parameters $(a_k)_{j+\frac 12}^n,\, k=1,..,N$ satisfy Whitham's condition at each time step and each interface, \emph{i.e} that for all $k=1,..,N$, $n\in\N$, $j\in\Z$, $(a_k)_{j+\frac 12}^n$ is large enough so that
\begin{equation}
\label{whithambis}
((a_k)_{j+\frac 12}^n)^2 > -\frac{d\PP_k}{d\tau_k}(\T_k),
\end{equation}
for all $\T_k$ in the solution $\xi\mapsto \W_{\rm Riem}(\xi;\mathscr{M}(\vect{U}_{j}^n),\mathscr{M}(\vect{U}_{j+1}^n))$. Then, the values $\vect{U}_{j}^n,\, j\in\Z,\,n\in\N$, computed by the scheme satisfy the following discrete energy inequalities, which are discrete counterparts of the energy inequalities \eqref{multi:ener1:ineq} and \eqref{multi:enerk:ineq} satisfied by the exact entropy weak solutions of the model:
\begin{align}
& \label{ener_prop_1}
\begin{array}{ll}
  (\alpha_1\rho_1 E_1)(\vect{U}_j^{n+1}) \leq (\alpha_1\rho_1 E_1)(\vect{U}_j^{n}) & \displaystyle - \frac{\Delta t}{\Delta x} \left (  (\alpha_1\rho_1 \E_1 u_1+\alpha_1 \pi_1 u_1)_{j+\frac 12}^{n}- (\alpha_1\rho_1 \E_1 u_1+\alpha_1 \pi_1 u_1)_{j-\frac 12}^{n}\right) \\[2ex] 
   & \displaystyle -\frac{\Delta t}{\Delta x}  \Ind_{\left \lbrace (u_1^*)_{j-\frac 12}^n \geq 0 \right \rbrace }(u_1^*)_{j- \frac 12}^n \, \sum_{l=2}^N(\pi_l^*)_{j- \frac 12}^n\left ( (\alpha_l)_{j}^n- (\alpha_l)_{j-1}^n \right ) \\[2ex]
   & \displaystyle-\frac{\Delta t}{\Delta x}  \Ind_{\left \lbrace (u_1^*)_{j+\frac 12}^n \leq 0 \right \rbrace }(u_1^*)_{j+ \frac 12}^n \, \sum_{l=2}^N(\pi_l^*)_{j+ \frac 12}^n \left ( (\alpha_l)_{j+1}^n- (\alpha_l)_{j}^n \right ),
   \end{array}  
\\
& \nonumber \text{and for $k=2,..,N$}:
\\
& \label{ener_prop_k}
\begin{array}{ll}
  (\alpha_k\rho_k E_k)(\vect{U}_j^{n+1}) \leq (\alpha_k\rho_k E_k)(\vect{U}_j^{n}) & \displaystyle - \frac{\Delta t}{\Delta x} \left (  (\alpha_k\rho_k \E_k u_k+\alpha_k \pi_k u_k)_{j+\frac 12}^{n}- (\alpha_k\rho_k \E_k u_k+\alpha_k \pi_k u_k)_{j-\frac 12}^{n}\right) \\[2ex] 
   & \displaystyle +\frac{\Delta t}{\Delta x}  \Ind_{\left \lbrace (u_1^*)_{j-\frac 12}^n \geq 0 \right \rbrace }(u_1^*)_{j- \frac 12}^n \, (\pi_k^*)_{j- \frac 12}^n\left ( (\alpha_k)_{j}^n- (\alpha_k)_{j-1}^n \right ) \\[2ex]
   & \displaystyle+\frac{\Delta t}{\Delta x}  \Ind_{\left \lbrace (u_1^*)_{j+\frac 12}^n \leq 0 \right \rbrace }(u_1^*)_{j+ \frac 12}^n \, (\pi_k^*)_{j+ \frac 12}^n \left ( (\alpha_k)_{j+1}^n- (\alpha_k)_{j}^n \right ),
   \end{array}
\end{align}
where  for $j\in\Z$, $(\alpha_k\rho_k \E_k u_k+\alpha_k \pi_k u_k)_{j+\frac 12}^{n}=(\alpha_k\rho_k \E_k u_k+\alpha_k \pi_k u_k) \left (\W_{\rm Riem}(0^+;\mathscr{M}(\vect{U}_{j}^n),\mathscr{M}(\vect{U}_{j+1}^n)) \right )$ is the right hand side trace of the phasic energy flux evaluated at $x_{j+\frac 12}$.
 \end{itemize}
\end{theorem}

\begin{proof}
A classical reformulation of the approximate Riemann solver allows to see that $\vect{U}_j^{n+1}$ is the cell-average over $C_j$ of the following function at $t=t^{n+1}$:
\begin{equation}
\label{app}
 \vect{U}_{app}(x,t):=\sum_{j\in\Z} \mathscr{P} \vect{W}_{\rm Riem}\left (\frac{x-x_{j+\frac 12}}{t-t^n};\mathscr{M}(\vect{U}_{j}^n),\mathscr{M}(\vect{U}_{j+1}^n) \right ) \Ind_{[x_j,x_{j+1}]}(x).
\end{equation}
Hence, the positivity property on the phase fractions and phase densities is a direct consequence of Theorem \ref{multi:THEtheorem} and Definition \ref{def_sol} which states the positivity of the densities in the relaxation Riemann solution. For this purpose, energy dissipation \eqref{multi:relax:enerk:bis} across the $u_1$-contact discontinuity may be necessary for enforcing this property when the ratio $\frac{(\alpha_k)^n_{j}}{(\alpha_k)^n_{j+1}}$ (or its inverse) is large for some $j\in\Z$. 

The conservativity of the discretization of the mass equation is straightforward. That of the discretization of the total momentum is a consequence of the fact that $u_1^*$ is the solution of the fixed-point problem \eqref{fpTheta}.

Let us now prove the discrete energy inequalities \eqref{ener_prop_k} for phases $k=2,..,N$ satisfied by the scheme under Whitham's condition \eqref{whithambis}. Assuming the CFL condition \eqref{cfl}, the solution of \eqref{sys:multi:relax} over $[x_{j-\frac 12},x_{j+\frac 12}]\times[t^n,t^{n+1}]$ is the function
\begin{multline}
\vect{W}(x,t):=\vect{W}_r\left (\frac{x-x_{j-\frac 12}}{t-t^n};\mathscr{M}(\vect{U}_{j-1}^n),\mathscr{M}(\vect{U}_{j}^n) \right ) \Ind_{[x_{j-\frac 12},x_{j}]}(x) \\ +\vect{W}_r\left (\frac{x-x_{j+\frac 12}}{t-t^n};\mathscr{M}(\vect{U}_{j}^n),\mathscr{M}(\vect{U}_{j+1}^n) \right ) \Ind_{[x_{j},x_{j+\frac 12}]}(x).
\end{multline}
According to Definition \ref{def_sol}, this function satisfies the phase k energy equation:
\begin{equation}
 \label{local_ener_relax_k}
\begin{aligned}
\dv_t (\alpha_k \rho_k \E_k) &+ \dv_x (\alpha_k \rho_k \E_k u_k + \alpha_k \pi_k u_k) \\
&- (u_1^*\pi_k^*)_{j-\frac 12}^n \left ( (\alpha_k)_{j}^n- (\alpha_k)_{j-1}^n \right ) \delta_{0}\left (x-x_{j-\frac12}-(u_1^*)_{j-\frac 12}^n(t-t^n) \right ) \\
&- (u_1^*\pi_k^*)_{j+\frac 12}^n \left ( (\alpha_k)_{j+1}^n- (\alpha_k)_{j}^n \right ) \delta_{0}\left (x-x_{j+\frac12}-(u_1^*)_{j+\frac 12}^n(t-t^n) \right )   \\
&=- (\mathcal{Q}_k)_{j-\frac 12}^n \delta_{0}\left (x-x_{j-\frac12}-(u_1^*)_{j-\frac 12}^n(t-t^n) \right ) 
- (\mathcal{Q}_k)_{j+\frac 12}^n\delta_{0} \left (x-x_{j+\frac12}-(u_1^*)_{j+\frac 12}^n(t-t^n) \right ),
 \end{aligned}
\end{equation}
where for $i\in\Z$, we have denoted $(\mathcal{Q}_k)_{i-\frac 12}^n=\mathcal{Q}_k\left ((u_1^*)_{i-\frac 12}^n,\mathscr{M}(\vect{U}_{i-1}^n),\mathscr{M}(\vect{U}_{i}^n) \right )$.
Integrating this equation over $]x_{j-\frac 12},x_{j+\frac 12}[\times[t^n,t^{n+1}]$ and dividing by $\Delta x$ yields:
\begin{equation}
\label{enerk_proof}
 \begin{array}{ll}
  \displaystyle \frac{1}{\Delta x}\int_{x_{j-\frac 12}}^{x_{j+\frac 12}}(\alpha_k\rho_k \E_k)(\vect{W}(x,t^{n+1}))\dx 
  &\leq (\alpha_k\rho_k \E_k)(\mathscr{M}(\vect{U}_j^{n})) \\
  & \displaystyle - \frac{\Delta t}{\Delta x}  (\alpha_k\rho_k \E_k u_k+\alpha_k \pi_k u_k) \left ( \vect{W}_{\rm Riem}\left (0^-;\mathscr{M}(\vect{U}_{j}^n),\mathscr{M}(\vect{U}_{j+1}^n) \right ) \right )\\[2ex]
  &\displaystyle  + \frac{\Delta t}{\Delta x} (\alpha_k\rho_k \E_k u_k+\alpha_k \pi_k u_k) \left (\vect{W}_{\rm Rieam}\left (0^+;\mathscr{M}(\vect{U}_{j-1}^n),\mathscr{M}(\vect{U}_{j}^n) \right ) \right ) \\[2ex]
  & \displaystyle +\frac{\Delta t}{\Delta x}  \Ind_{\left \lbrace (u_1^*)_{j-\frac 12}^n \geq 0 \right \rbrace }(u_1^* \, \pi_k^*)_{j- \frac 12}^n\left ( (\alpha_k)_{j}^n- (\alpha_k)_{j-1}^n \right ) \\[2ex]
  & \displaystyle+\frac{\Delta t}{\Delta x}  \Ind_{\left \lbrace (u_1^*)_{j+\frac 12}^n \leq  0 \right \rbrace }(u_1^* \, \pi_k^*)_{j+ \frac 12}^n \left ( (\alpha_k)_{j+1}^n- (\alpha_k)_{j}^n \right ),
   \end{array}
\end{equation}
because $(\mathcal{Q}_k)_{j- \frac 12}^n\geq0$ and $(\mathcal{Q}_k)_{j+ \frac 12}^n \geq 0$. Since the initial data is at equilibrium: $\vect{W}(x,t^{n})=\mathscr{M}(\vect{U}_j^{n})$ for all $x\in C_j$ (\emph{i.e.}  $(\T_k)_j^n$ is set to be equal to $(\tau_k)_j^n$) one has $(\alpha_k\rho_k \E_k)(\mathscr{M}(\vect{U}_j^{n}))=(\alpha_k\rho_k E_k)(\vect{U}_j^{n})$ according to Proposition \ref{prop:ener:multi:relax}. Applying the Rankine-Hugoniot jump relation to \eqref{local_ener_relax_k} across the line $\lbrace (x,t),x=x_{j+\frac 12}, \, t>0\rbrace$, yields:
\begin{multline*}
(\alpha_k\rho_k \E_k u_k+\alpha_k \pi_k u_k) \left ( \vect{W}_{\rm Riem}\left (0^-;\mathscr{M}(\vect{U}_{j}^n),\mathscr{M}(\vect{U}_{j+1}^n) \right ) \right ) \\
= (\alpha_k\rho_k \E_k u_k+\alpha_k \pi_k u_k) \left ( \vect{W}_{\rm Riem}\left (0^+;\mathscr{M}(\vect{U}_{j}^n),\mathscr{M}(\vect{U}_{j+1}^n) \right ) \right ) + (\mathcal{Q}_k)_{j+ \frac 12}^n \Ind_{\left \lbrace (u_1^*)_{j+\frac 12}^n =  0 \right \rbrace }.
 \end{multline*}
Hence, since $(\mathcal{Q}_k)_{j+ \frac 12}^n\geq0$, for the interface $x_{j+\frac 12}$, taking the trace of $(\alpha_k\rho_k \E_k u_k+\alpha_k \pi_k u_k)$ at $0^+$ instead of $0^-$ in \eqref{enerk_proof} only improves the inequality. Furthermore, assuming that the parameter $a_k$ satisfies Whitham's condition \eqref{whithambis}, the Gibbs principle stated in \eqref{gibbsAE} holds true so that:
$$
\frac{1}{\Delta x}\int_{x_{j-\frac 12}}^{x_{j+\frac 12}}(\alpha_k\rho_k E_k)(\vect{U}_{app}(x,t^{n+1}))\dx \leq \frac{1}{\Delta x}\int_{x_{j-\frac 12}}^{x_{j+\frac 12}}(\alpha_k\rho_k \E_k)(\vect{W}(x,t^{n+1}))\dx.
$$
Invoking the convexity of the mapping $\vect{U}\mapsto(\alpha_k\rho_k E_k)(\vect{U})$ (see \cite{Note-multi}), Jensen's inequality implies that
$$
(\alpha_k\rho_k E_k)(\vect{U}_j^{n+1}) \leq \frac{1}{\Delta x}\int_{x_{j-\frac 12}}^{x_{j+\frac 12}}(\alpha_k\rho_k E_k)(\vect{U}_{app}(x,t^{n+1}))\dx,
$$
which yields the desired discrete energy inequality for phase k. The proof of the discrete energy inequality for phase 1 follows similar steps. 
\end{proof}

\section{Numerical results}
\label{secnumtest}
In this section, we present three test cases on which the performances of the relaxation scheme are illustrated. We only consider the three phase flow model (\emph{i.e.} $N=3$). In the first two test cases,
the thermodynamics of the three phases are given by ideal gas pressure laws for $k=1,2,3$ :
\begin{equation}
\label{gpeos}
 p_k(\rho_k)=\kappa_k\rho_k^{\gamma_k},
\end{equation}
and we consider the approximation of the solutions of two different Riemann problems. In the third test case, we consider the simulation of a shock tube apparatus, where a gas shock wave interacts with a lid of rigid particles. This third case is also simulated with the three phase flow model although it is a two phase flow. The thermodynamics of the particle phase is given by a barotropic stiffened gas \emph{e.o.s.}.  

\medskip
We recall that the scheme relies on a relaxation Riemann solver which requires solving a fixed-point problem in order to compute, for every cell interface $x_{j+\frac 12}$, the zero of a scalar function (see eq. \eqref{fpTheta}). Newton's method is used in order to compute this solution. Usually, convergence is achieved within three iterations.



\subsection{Test-case 1: a Riemann problem with all the waves}
In this test-case, the thermodynamics of all three phases are given by barotropic \emph{e.o.s.} \eqref{gpeos} with the parameters given in Table \ref{Table_eos1}.
\begin{table}[ht!]
\centering
\begin{tabular}{|ccc|}
\hline
$(\kappa_1,\gamma_1)$ 	& $(\kappa_2,\gamma_2)$	& $(\kappa_3,\gamma_3)$	 \\
\hline
$(1,3)$			& $(10,1.4)$		& $(1,1.6)$    \\
\hline
\end{tabular}
\protect \parbox[t]{13cm}{\caption{E.o.s parameters for Test 1.\label{Table_eos1}}}
\end{table}
The wave pattern for phase 1 consists of a left-traveling rarefaction wave, a phase fraction discontinuity of velocity $u_1$ and a right-traveling shock. For phase 2 the wave pattern is composed of a left-traveling shock, the phase fraction discontinuity, and a right-traveling rarefaction wave. Finally, the wave pattern for phase 3 is composed of a left-traveling shock, the phase fraction discontinuity, and a right-traveling shock. The $u_1$-contact discontinuity separates two regions denoted $-$ and $+$ respectively on the left and right sides of the discontinuity (see Figure \ref{Fig_struct}).

\begin{figure}[ht!]
\centering
\begin{center}
\begin{tikzpicture}[scale=4]
\tikzstyle{axes}=[thin,>=latex]
\begin{scope}[axes]
        \draw[->] (-1,0)--(1,0) node[right=3pt] {$x$};
        \draw[->] (0,0)--(0,1) node[left=3pt] {$t$};
        \draw [very thick,color=red] (0,0) -- (20:1cm);
	\draw [very thick,color=red] (0,0) -- (25:1cm);
	\draw [very thick,color=red] (0,0) -- (30:1cm);
        \draw [dashed, very thick,color=red] (0,0) -- (80:0.8cm) node [color=black, above=2pt] {$u_1$};
	\draw [very thick,color=red] (0,0) -- (145:1cm);
	\draw [very thick,color=red] (0,0) -- (150:1cm);
	\draw [very thick,color=red] (0,0) -- (155:1cm);
	\draw (-1.1,0.7) node[] {$u_{1}-c_1$};
	\draw (-1.1,0.6) node[] {$u_{2}-c_2$};
	\draw (-1.1,0.5) node[] {$u_{3}-c_3$};
	\draw (1.1,0.6) node[] {$u_{1}+c_1$};
	\draw (1.1,0.5) node[] {$u_{2}+c_2$};
	\draw (1.1,0.4) node[] {$u_{3}+c_3$};
	\draw (-0.6,0.1) node[] {$\U_{L}$};
	\draw (-0.25,0.6) node[] {$\U^-$};
	\draw (0.4,0.6) node[] {$\U^+$};
	\draw (0.7,0.1) node[] {$\U_{R}$};
\end{scope}
\end{tikzpicture}
\end{center}
\protect \parbox[t]{13cm}{\caption{Structure of a Riemann solution, notations for the intermediate states.\label{Fig_struct}}}
\end{figure}
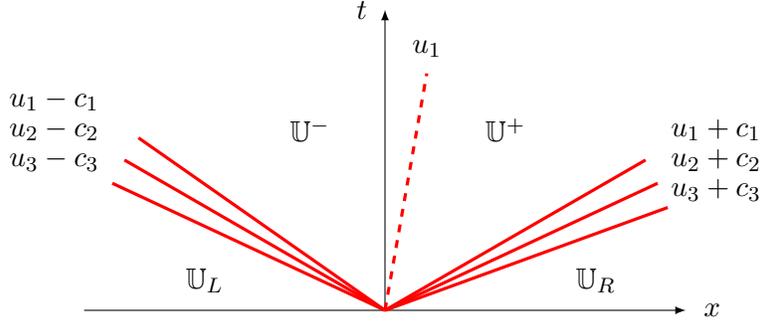

\begin{table}[ht!]
\centering
\begin{tabular}{|c|cccc|}
\hline
		& Region $L$ 	& Region $-$	& Region $+$	& Region $R$ \\
\hline
$\alpha_1$ 	&$0.9$		&$0.9$		&$0.4$		&$0.4$	   \\
$\alpha_2$ 	&$0.05$		&$0.05$		&$0.4$		&$0.4$	   \\
$\rho_1$	&$2.5$		&$2.0$		&$2.06193$	&$1.03097$  \\
$u_1$		&$-0.56603$	&$0.3$		&$0.3$		&$-1.62876$   \\
$\rho_2$	&$0.2$		&$1.0$		&$1.00035$	&$1.25044$	   \\
$u_2$		&$6.18311$	&$0.2$		&$0.28750$	&$1.14140$	   \\
$\rho_3$	&$0.5$		&$1.0$		&$1.19853$	&$.59926$	   \\
$u_3$		&$0.31861$	&$-0.5$		&$0.13313$	&$-0.73119$	   \\
\hline
\end{tabular}
\protect \parbox[t]{13cm}{\caption{Test-case 1: left, right and intermediate states of the exact solution.\label{Table_TC1}}}
\end{table}

The relaxation scheme is compared with Rusanov's scheme, which is the only numerical scheme presently available for the three-phase flow model (see \cite{bou-18-rel}).
In Figure \ref{Figcase1}, the approximate solution computed with the relaxation scheme is compared with both the exact solution and the approximate solution obtained with Rusanov's scheme (a Lax-Friedrichs type scheme). The results show that unlike Rusanov's scheme, the relaxation method correctly captures the intermediate states even for this rather coarse mesh of $100$ cells. This coarse mesh is a typical example of an industrial mesh, reduced to one direction, since $100$ cells in 1D correspond to a $10^6$-cell mesh in 3D. It appears that the contact discontinuity is captured more sharply by the relaxation method than by Rusanov's scheme for which the numerical diffusion is larger. In addition, the velocity of the contact discontinuity is not well estimated for the phase 2 variables with such a coarse mesh. We can also see that for the phase 2 and phase 3 variables, there are no oscillations as one can see for Rusanov's scheme: the curves are monotone between the intermediate states. The intermediate states for the phases 2 and 3 are not captured by Rusanov's scheme whereas the relaxation scheme gives a rather good estimation, even for the narrow state in phase 3 between the contact discontinuity and the right-traveling shock.
These observations assess that, for the same level of refinement, the relaxation method is much more accurate than Rusanov's scheme.

\medskip
A mesh refinement process has also been implemented in order to check numerically the convergence of the method, as well as its performances in terms of CPU-time cost. For this purpose, we compute the discrete $L^1$-error between the approximate solution and the exact one at the final time $T_{\rm max}=N\Delta t=0.05$, normalized by the discrete $L^1$-norm of the exact solution:
\begin{equation*}
 E(\Delta x) = \dfrac{\sum_{j}|\phi_j^N-\phi_{ex}(x_j,T_{\rm max})| \Delta x}{\sum_{j}|\phi_{ex}(x_j,T_{\rm max})| \Delta x},
\end{equation*}
where $\phi$ is any of the \emph{non-conservative} variables $(\alpha_1, \alpha_2, \rho_1, u_1,\rho_2, u_2,\rho_3, u_3)$. The calculations have been implemented on several meshes composed of $100 \times 2^n$ cells with $n=0,1,..,10$ (knowing that the domain size is $L=1$). In Figure \ref{Figcase1bis}, the error $E(\Delta x)$ at the final time $T_{\rm max}=0.05$, is plotted against $\Delta x$ in a $log-log$ scale for both schemes. 
We can see that all the errors converge towards zero with at least the expected order of $\Delta x^{1/2}$. Note that $\Delta x^{1/2}$ is only an asymptotic order of convergence, and in this particular case, one would have to implement the calculation on more refined meshes in order to reach the theoretically expected order of $\Delta x^{1/2}$.

\medskip
Figure \ref{Figcase1ter} displays the error on the non-conservative variables with respect to the CPU-time of the calculation expressed in seconds for both the relaxation scheme and Rusanov's scheme. Each point of the plot corresponds to one single calculation for a given mesh size. One can see that, for all the variables except $\rho_1$ and $u_1$, if one prescribes a given level of the error, the computational time needed to reach this error with Rusanov's scheme is higher than that needed by the relaxation scheme. On some variables, the gain of time can be spectacular. For instance, for the same error on the phase 1 fraction $\alpha_1$, the gain in computational cost is forty times when using the relaxation method rather than Rusanov's scheme which is a quite striking result. Indeed, even if Rusanov's scheme is known for its poor performances in terms of accuracy, it is also an attractive scheme for its reduced complexity. This means that the better accuracy of the relaxation scheme (for a fixed mesh) widely compensates for its (relative) complexity.

\begin{figure}[ht!]
\begin{center}
\begin{tabular}{ccc}
$\alpha_1$ & $\alpha_2$ & \\[1ex]
\includegraphics[width=5.5cm,height=4.3cm]{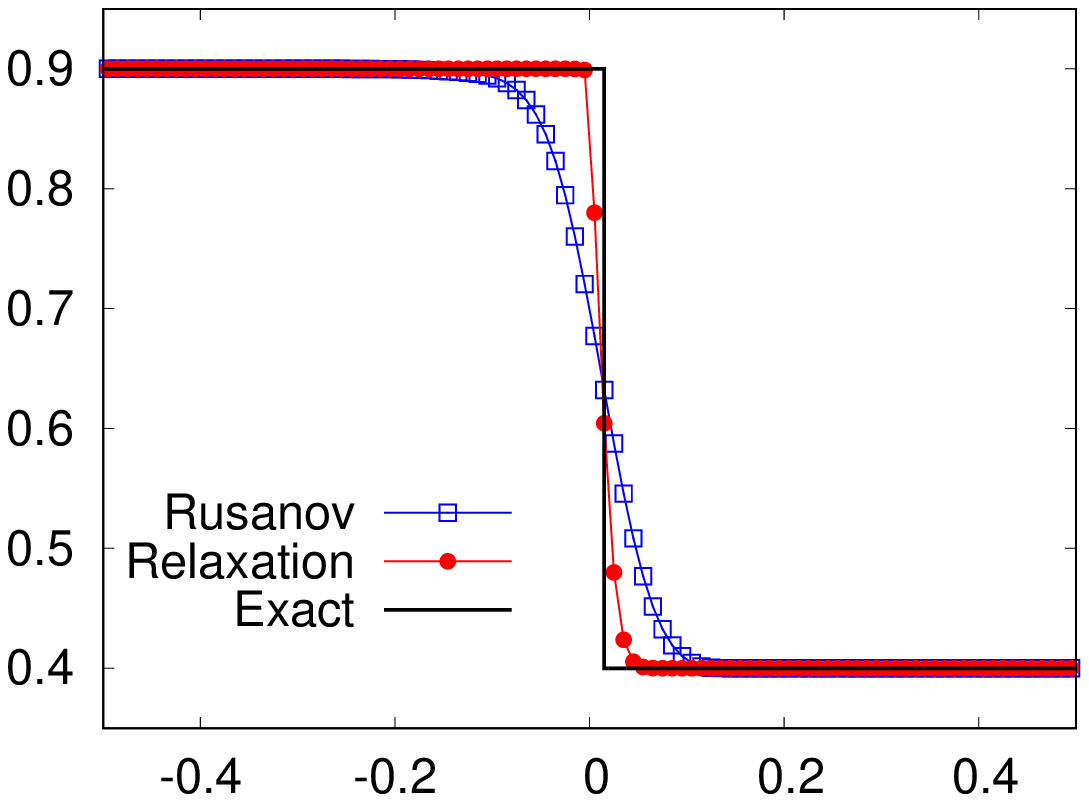}&
\includegraphics[width=5.5cm,height=4.3cm]{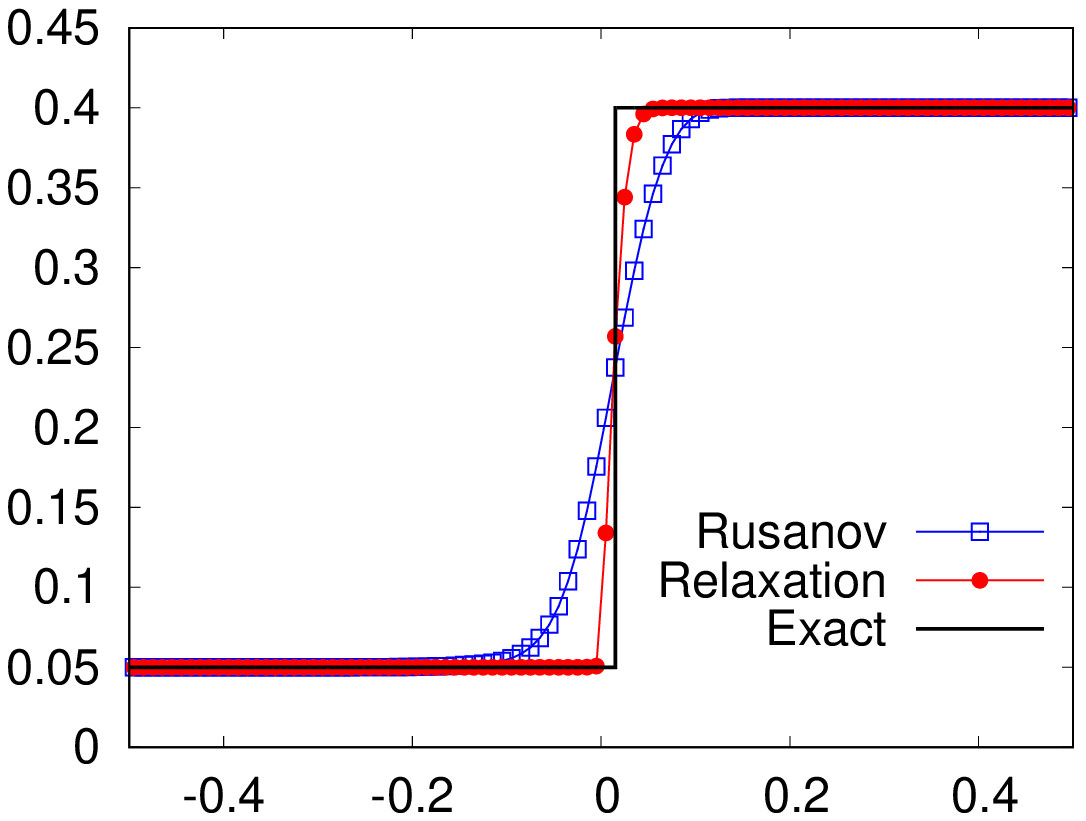}&  \\[3ex]
$u_1$ & $u_2$ & $u_3$ \\[1ex]
\includegraphics[width=5.5cm,height=4.3cm]{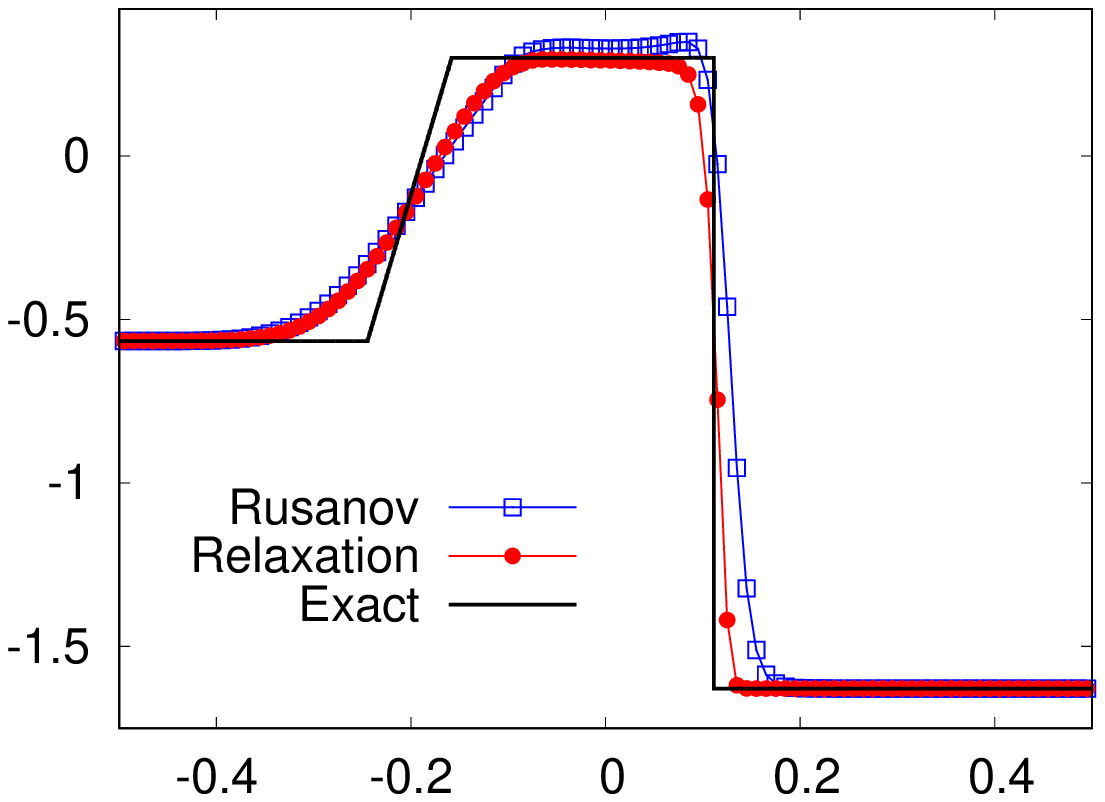} &
\includegraphics[width=5.5cm,height=4.3cm]{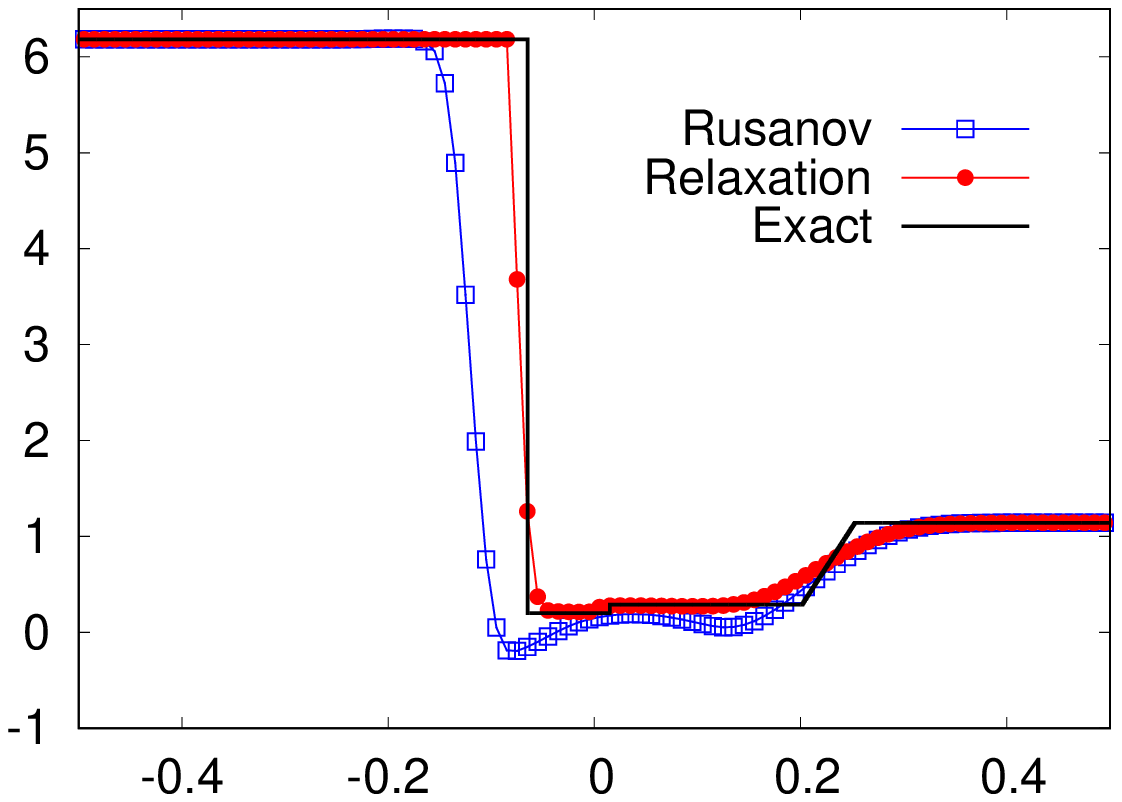} &
\includegraphics[width=5.5cm,height=4.3cm]{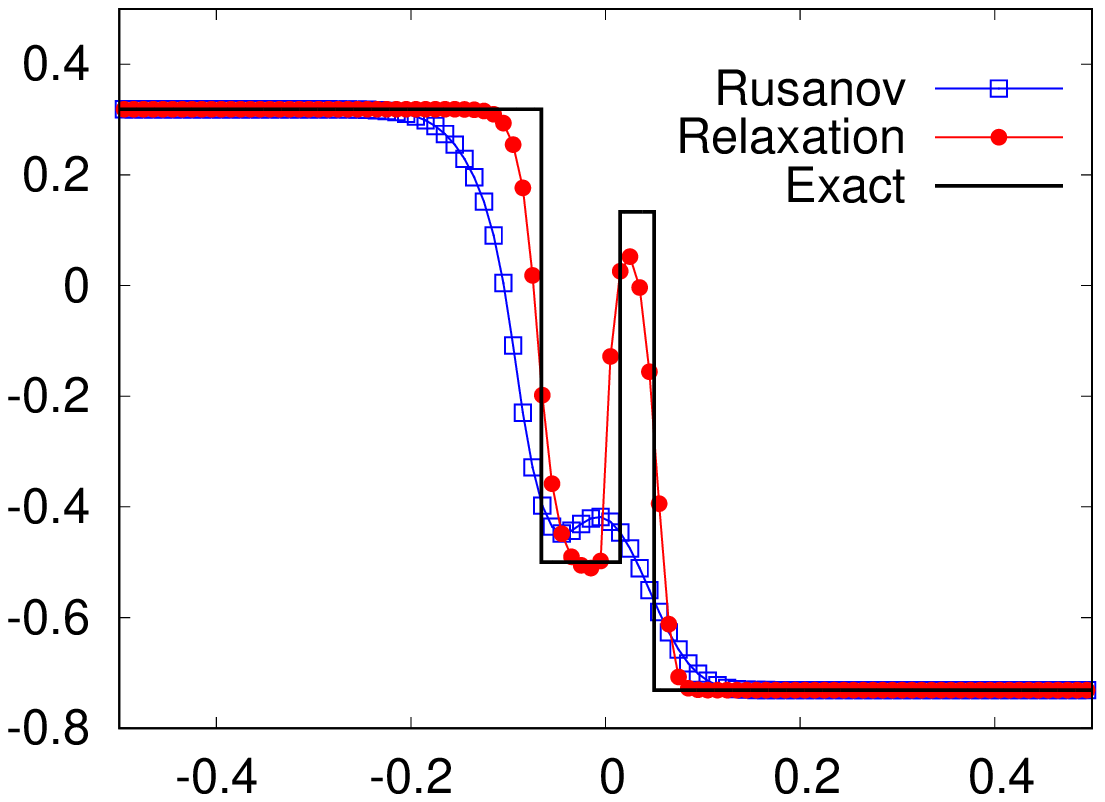}\\[3ex] 
$\rho_1$ & $\rho_2$ & $\rho_3$  \\[1ex]
\includegraphics[width=5.5cm,height=4.3cm]{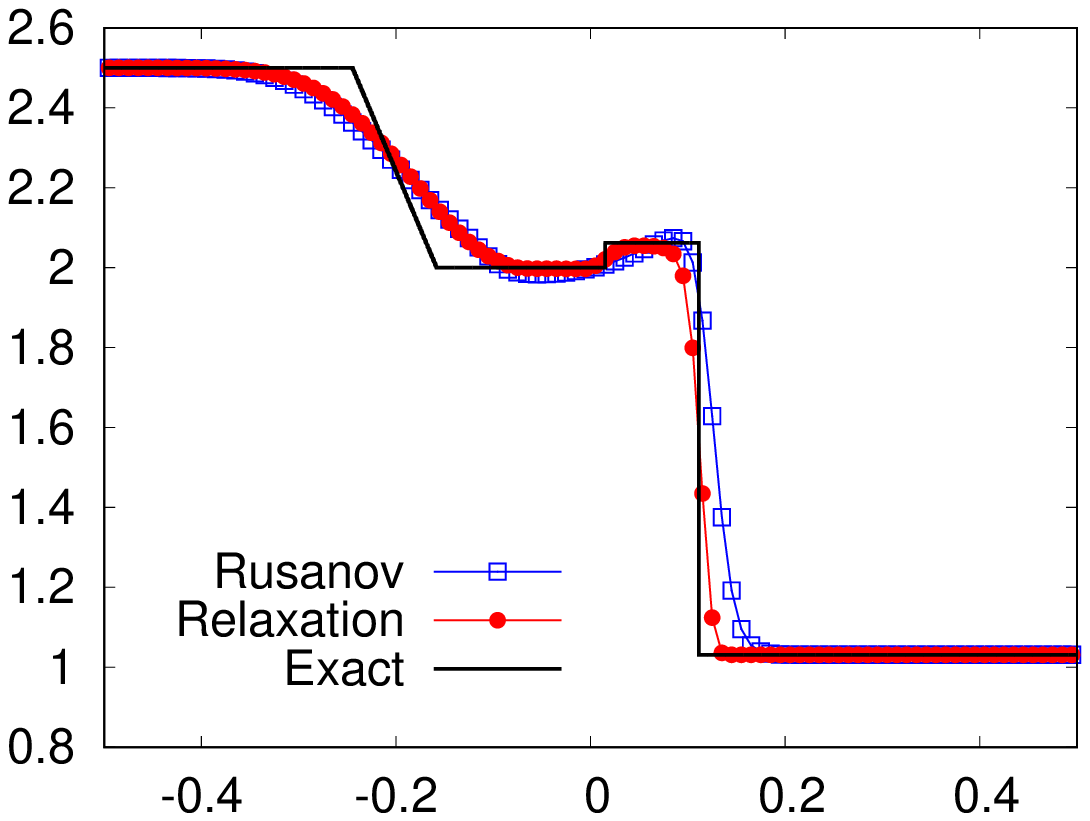} &
\includegraphics[width=5.5cm,height=4.3cm]{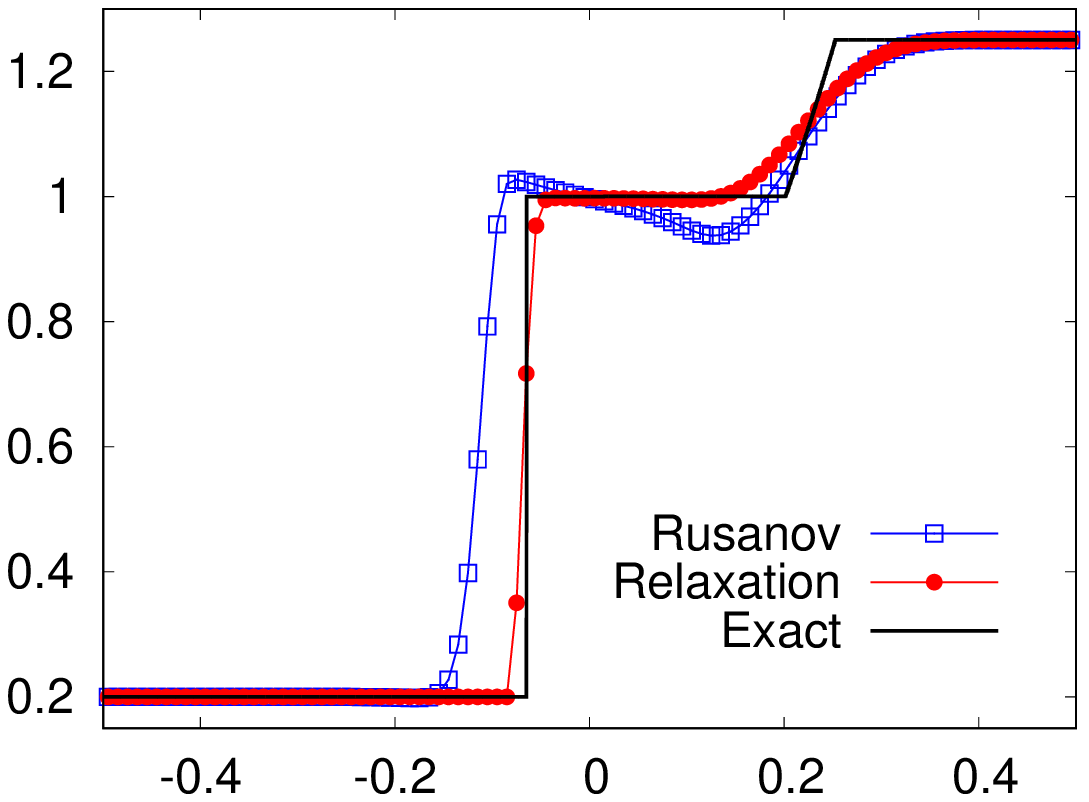} &
\includegraphics[width=5.5cm,height=4.3cm]{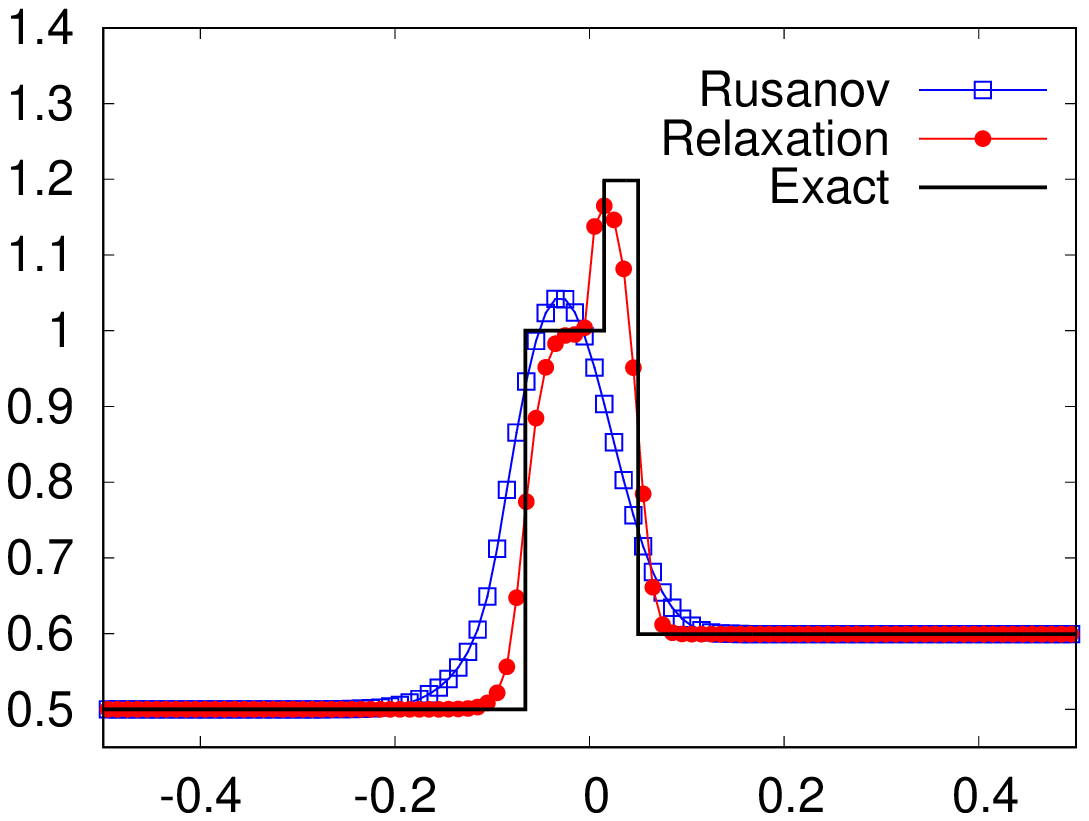} 
\end{tabular}
\protect \parbox[t]{17cm}{\caption{Test-case 1: space variations of the physical variables at the final time $T_{\rm max}=0.05$. Mesh size: $100$ cells.\label{Figcase1}}}
\end{center}
\normalsize
\end{figure}

\begin{figure}[ht!]
\begin{center}
\begin{tabular}{cc}
Rusanov's scheme & Relaxation scheme  \\[1ex]
\includegraphics[width=7cm,height=6cm]{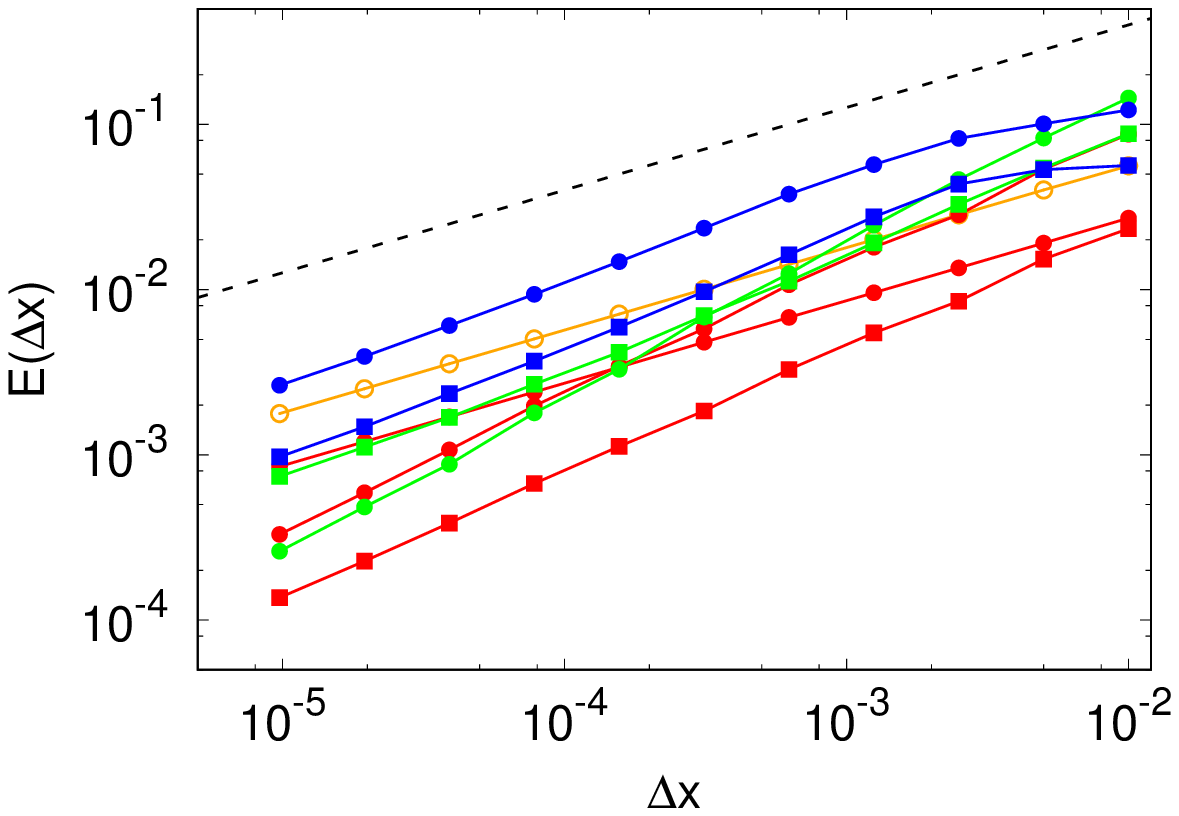}&
\includegraphics[width=10cm,height=6cm]{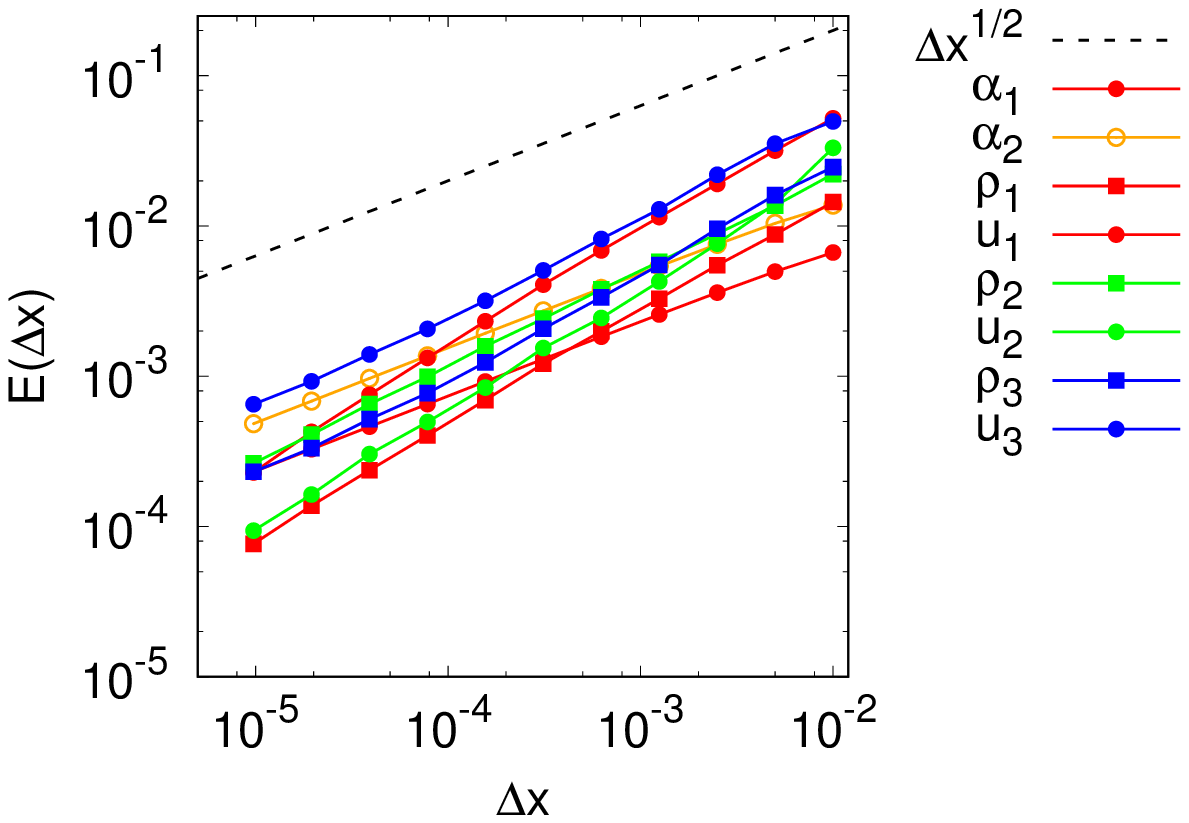}
\end{tabular}
\protect \parbox[t]{17cm}{\caption{Test-case 1: $L^1$-Error with respect to $\Delta x$ for the relaxation scheme and Rusanov's scheme.\label{Figcase1bis}}}
\end{center}
\end{figure}

\begin{figure}[ht!]
\begin{center}
\begin{tabular}{ccc}
$\alpha_1$ & $\alpha_2$ & \\[1ex]
\includegraphics[width=5.5cm,height=4.3cm]{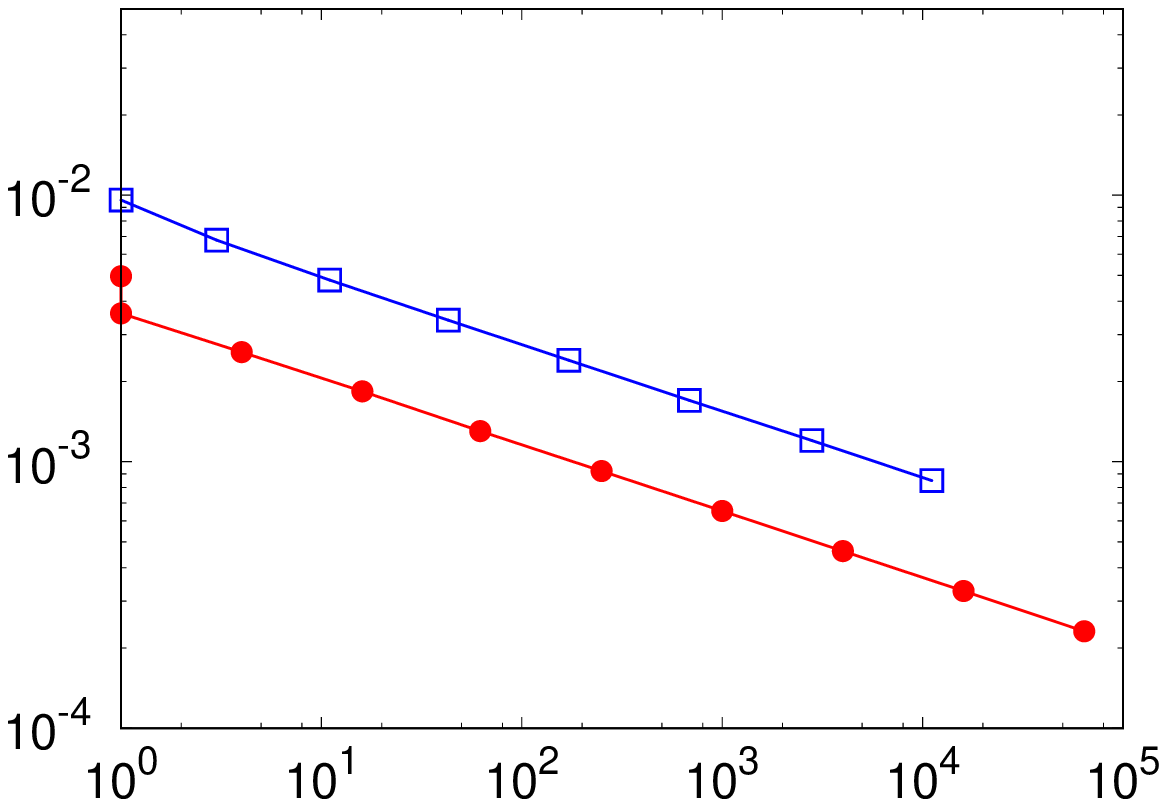}&
\includegraphics[width=5.5cm,height=4.3cm]{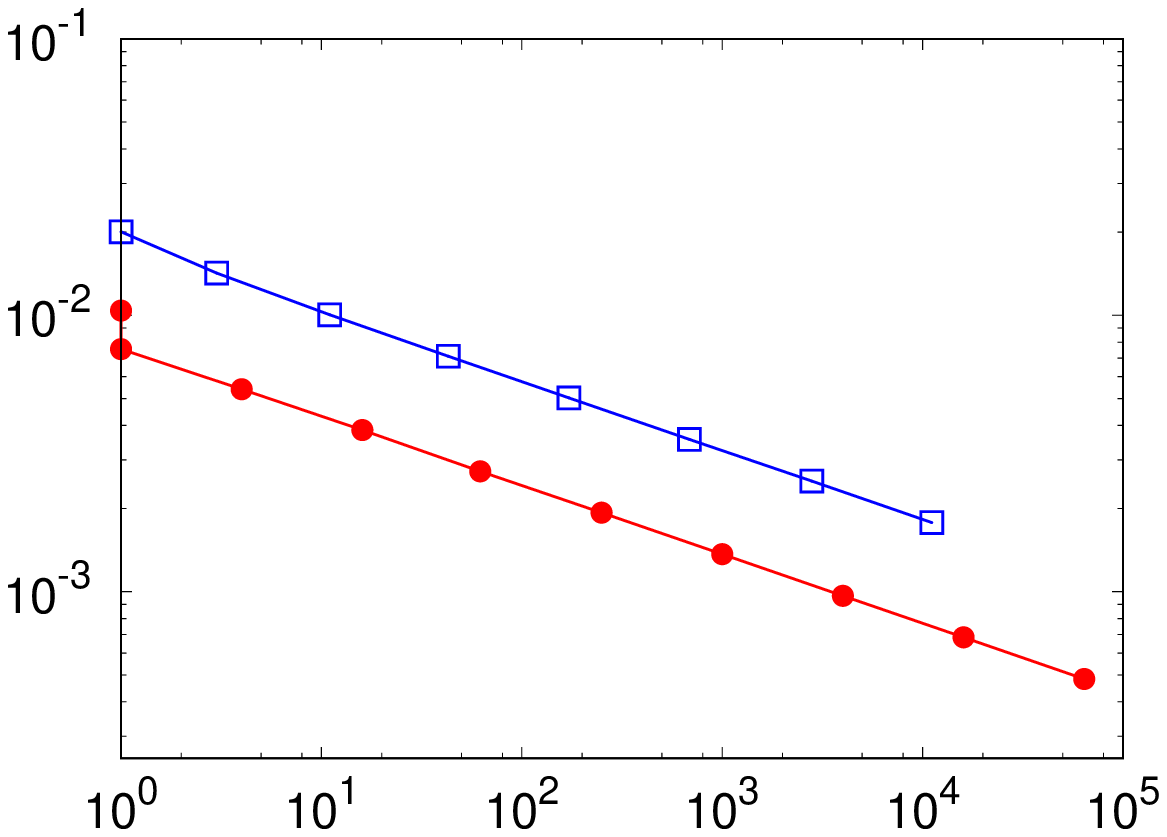}&  \\[3ex]
$u_1$ & $u_2$ & $u_3$ \\[1ex]
\includegraphics[width=5.5cm,height=4.3cm]{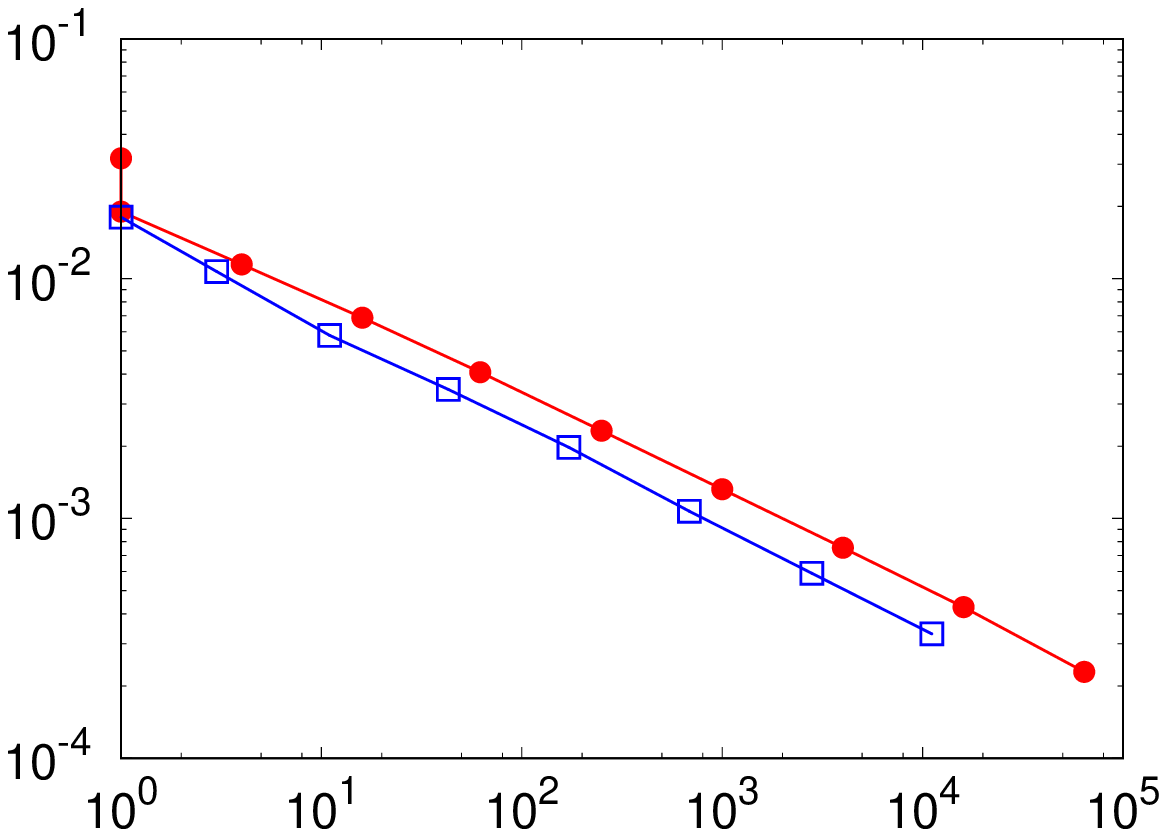} &
\includegraphics[width=5.5cm,height=4.3cm]{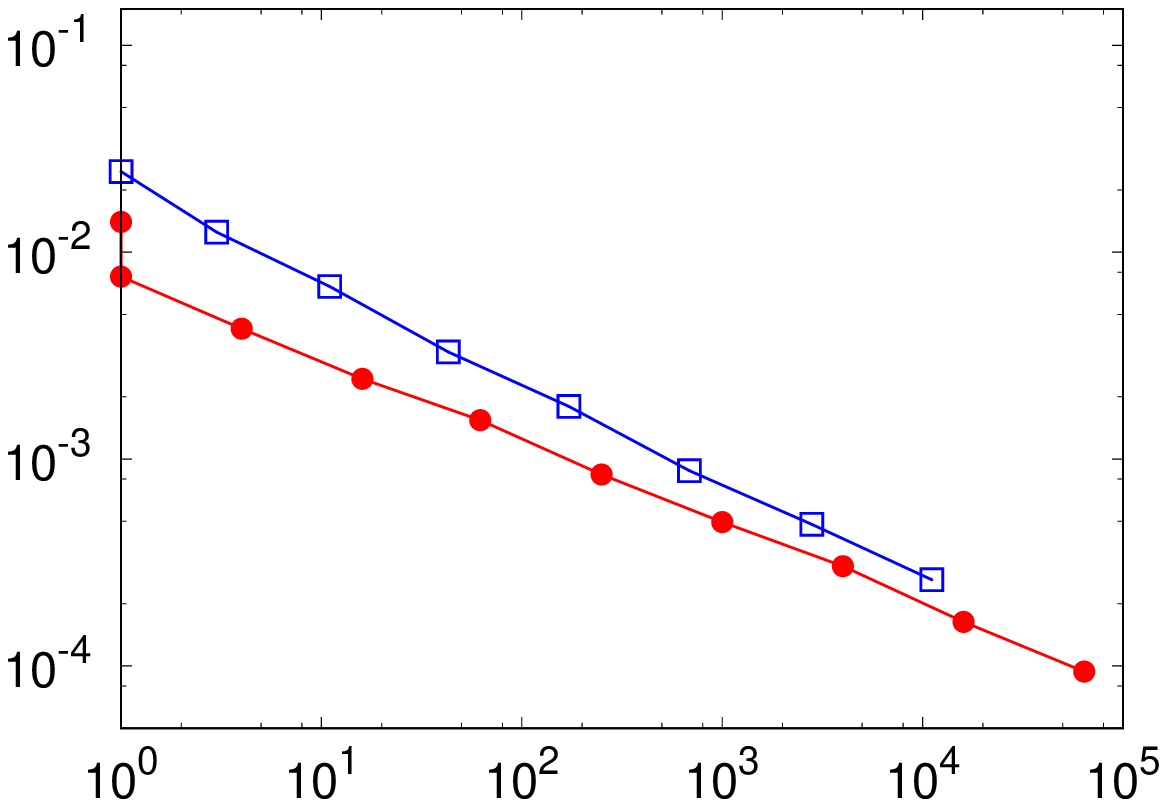} &
\includegraphics[width=5.5cm,height=4.3cm]{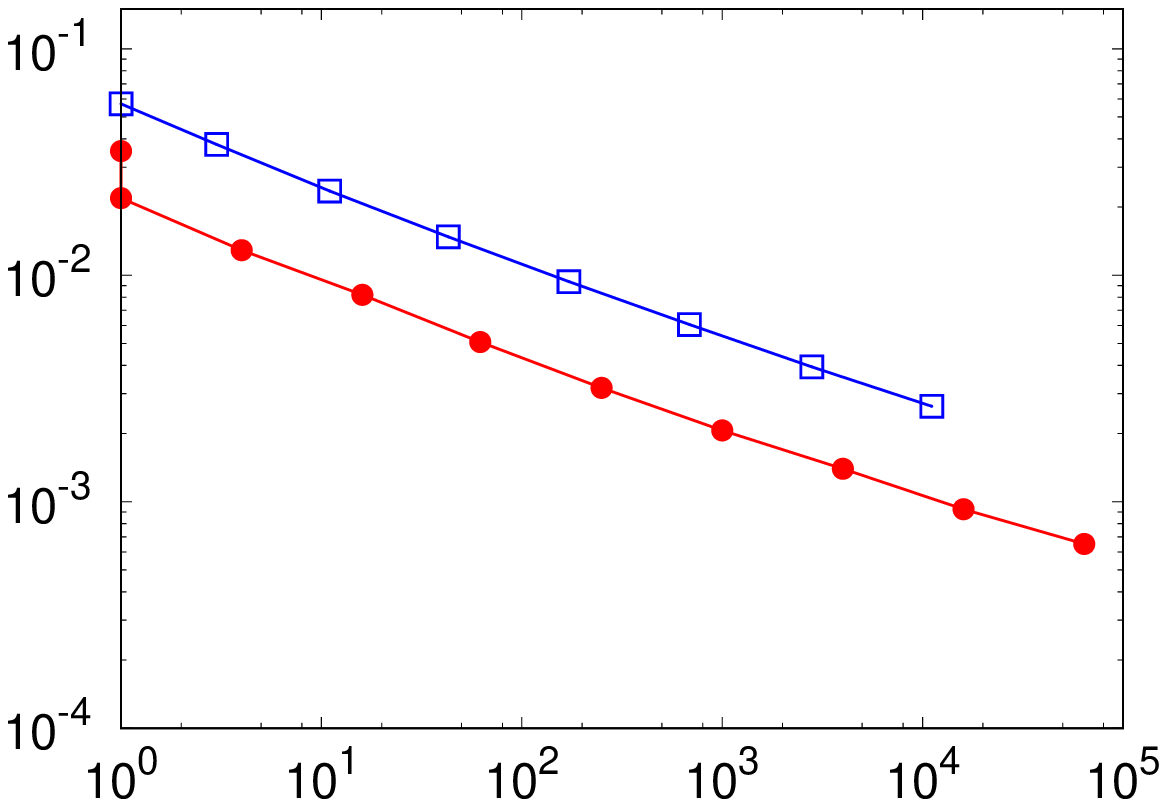}\\[3ex] 
$\rho_1$ & $\rho_2$ & $\rho_3$  \\[1ex]
\includegraphics[width=5.5cm,height=4.3cm]{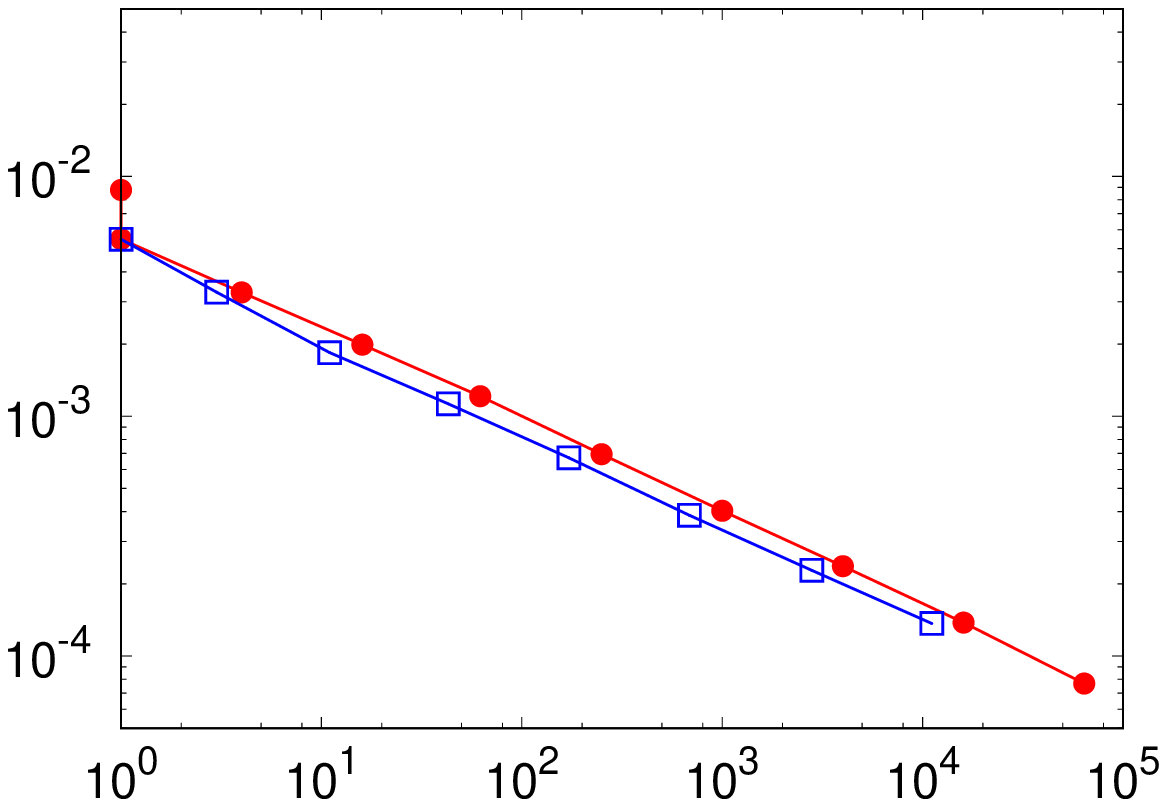} &
\includegraphics[width=5.5cm,height=4.3cm]{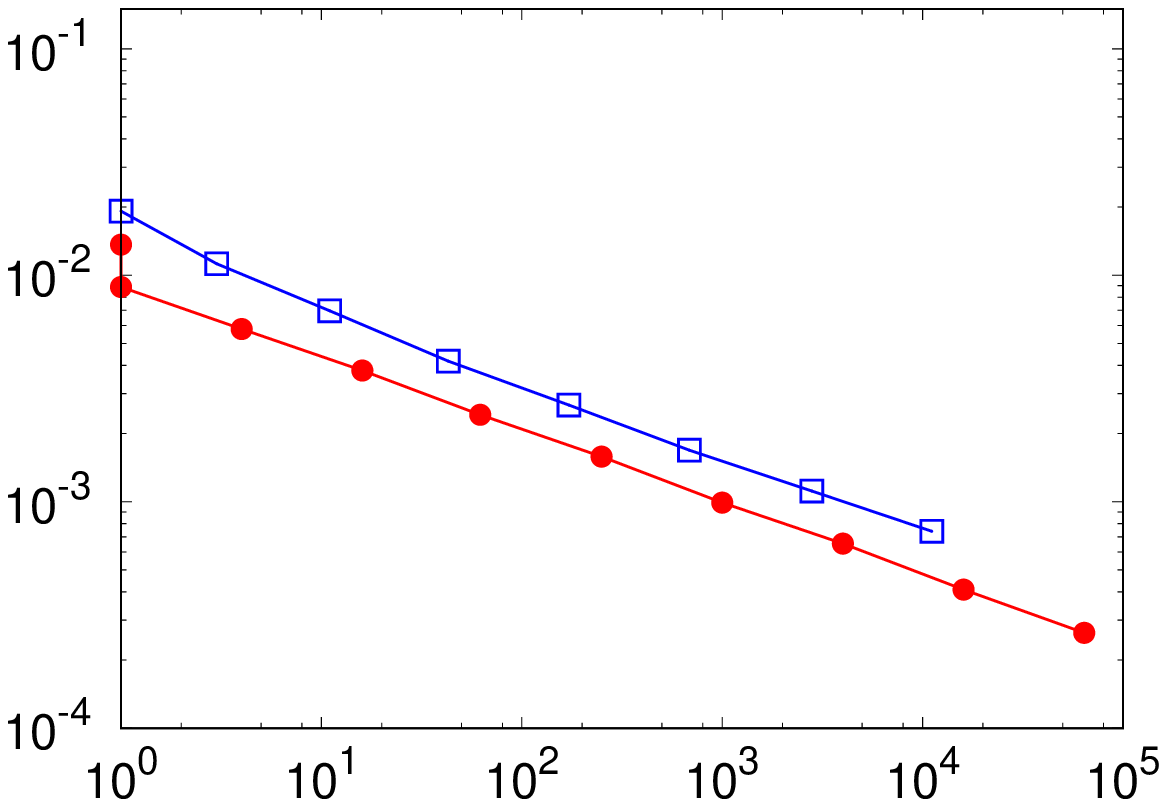} &
\includegraphics[width=5.5cm,height=4.3cm]{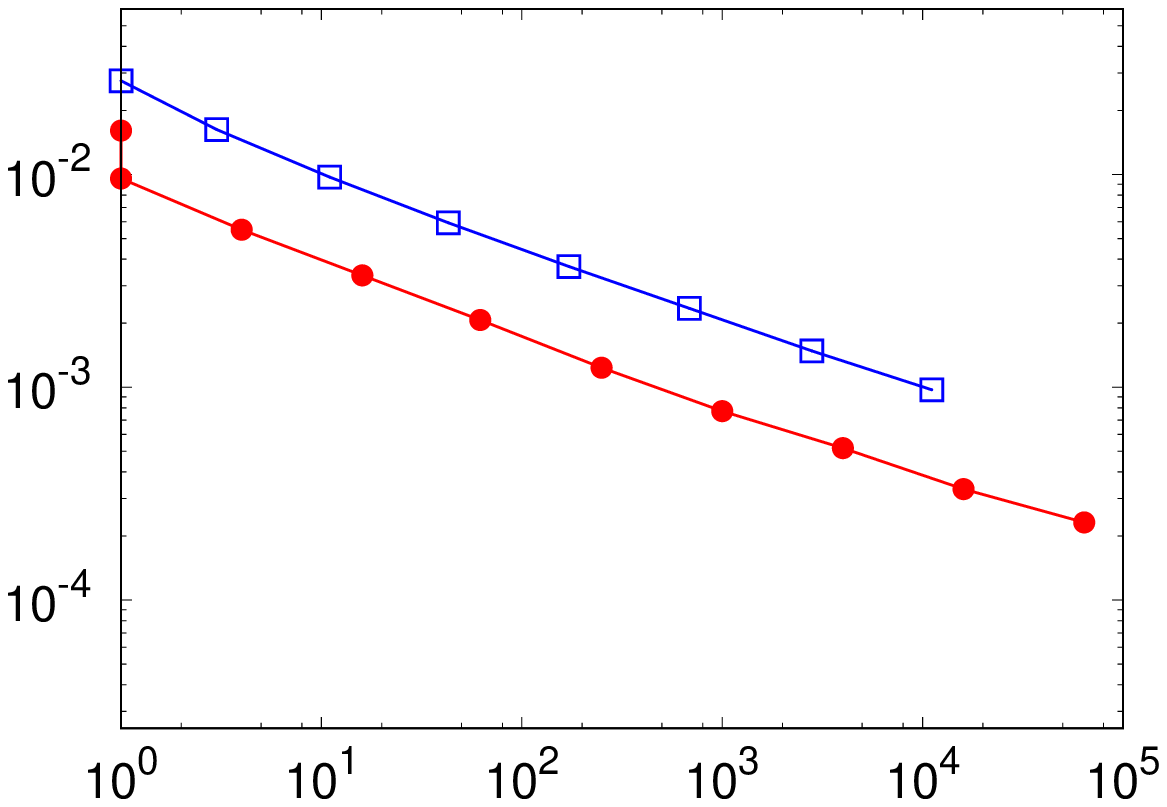} 
\end{tabular}
\protect \parbox[t]{17cm}{\caption{Test-case 1: $L^1$-Error with respect to computational cost (in seconds) for the relaxation scheme (bullets, red line) and Rusanov's scheme (squares, blue line).\label{Figcase1ter}}}
\end{center}
\end{figure}


\subsection{Test-case 2: a Riemann problem with a coupling between a single phase region and a mixture region}
In this test-case, the thermodynamics of all three phases are still given by barotropic \emph{e.o.s.} \eqref{gpeos} with the parameters given in Table \ref{Table_eos2}.
\begin{table}[ht!]
\centering
\begin{tabular}{|ccc|}
\hline
$(\kappa_1,\gamma_1)$ 	& $(\kappa_2,\gamma_2)$	& $(\kappa_3,\gamma_3)$	 \\
\hline
$(1,3)$			& $(10,1.4)$		& $(5,1.6)$    \\
\hline
\end{tabular}
\protect \parbox[t]{13cm}{\caption{E.o.s parameters for Test 2.\label{Table_eos2}}}
\end{table}

Here, we consider a Riemann problem in which two phases vanish in one of the initial states, which means that the corresponding phase fractions are equal to zero. For this kind of Riemann problem, the $u_1$-wave separates a mixture region where the three phases coexist from a single phase region with the remaining phase.

\medskip
The solution is composed of a $\{u_3-c_3\}$-shock wave in the left-hand side (LHS) region where only phase 3 is present. This region is separated by a $u_1$-contact discontinuity from the right-hand side (RHS) region where the three phases are mixed. In this RHS region, the solution is composed of a $\{u_1+c_1\}$-shock, a $\{u_2+c_2\}$-shock and a $\{u_3+c_3\}$-rarefaction wave.

\begin{table}[ht!]
\centering
\begin{tabular}{|c|cccc|}
\hline
		& Region $L$ 	& Region $-$	& Region $+$	& Region $R$ \\
\hline
$\alpha_1$ 	&$0.0$		&$0.0$		&$0.4$		&$0.4$	   \\
$\alpha_2$ 	&$0.0$		&$0.0$		&$0.2$		&$0.2$	   \\
$\rho_1$	&$-$		&$-$		&$1.35516$	&$0.67758$  \\
$u_1$		&$-$		&$-$		&$0.3$		&$-0.96764$   \\
$\rho_2$	&$-$		&$-$		&$1.0$		&$0.5$	   \\
$u_2$		&$-$		&$-$		&$0.3$		&$-2.19213$	   \\
$\rho_3$	&$0.5$		&$1.0$		&$0.99669$	&$1.24587$	   \\
$u_3$		&$2.03047$	&$0.2$		&$0.04917$	&$0.70127$	   \\
\hline
\end{tabular}
\protect \parbox[t]{13cm}{\caption{Test-case 2: left, right and intermediate states of the exact solution.\label{Table_TC2}}}
\end{table}

\medskip
In practice, the numerical method requires values of $\alpha_{1,L}$ and $\alpha_{2,L}$ that lie strictly in the interval $(0,1)$. Therefore, in the numerical implementation, we take $\alpha_{1,L}=\alpha_{2,L}= 10^{-10}$. The aim here is to give a qualitative comparison between the numerical approximation and the exact solution. Moreover, there is theoretically no need to specify left initial values for the phase 1 and phase 2 quantities since this phase is not present in the LHS region. For the sake of the numerical simulations however, one must provide such values. We choose to set $\rho_{1,L}$, $u_{1,L}$, $\rho_{2,L}$, $u_{2,L}$ to the values on the right of the $u_1$-contact discontinuity, which is coherent with the preservation of the Riemann invariants of this wave, and avoids the formation of fictitious acoustic waves for phases 1 and 2 in the LHS region.  For the relaxation scheme, this choice enables to avoid oscillations of the phases 1 and 2 density and velocity in the region where these phases are not present, as seen in Figure \ref{Figcase2}. However, some tests have been conducted that assess that taking other values of $(\rho_{1,L},u_{1,L},\rho_{2,L},u_{2,L})$ has little impact on the phase 3 quantities as well as on the phases 1 and 2 quantities where these phases are present.

\medskip
As expected, we can see that for the same level of refinement, the relaxation method is more accurate than Rusanov's scheme. As regards the region where phases 1 and 2 do not exist, we can see that the relaxation scheme is much more stable than Rusanov's scheme. Indeed, the relaxation scheme behaves better than Rusanov's scheme when it comes to divisions by small values of $\alpha_1$ or $\alpha_2$, since the solution approximated by Rusanov's scheme develops quite large oscillations. Moreover, the amplitude of these oscillations increases with the mesh refinement !

\begin{figure}[ht!]
\begin{center}
\begin{tabular}{ccc}
$\alpha_1$ & $\alpha_2$ & \\[1ex]
\includegraphics[width=5.5cm,height=4.3cm]{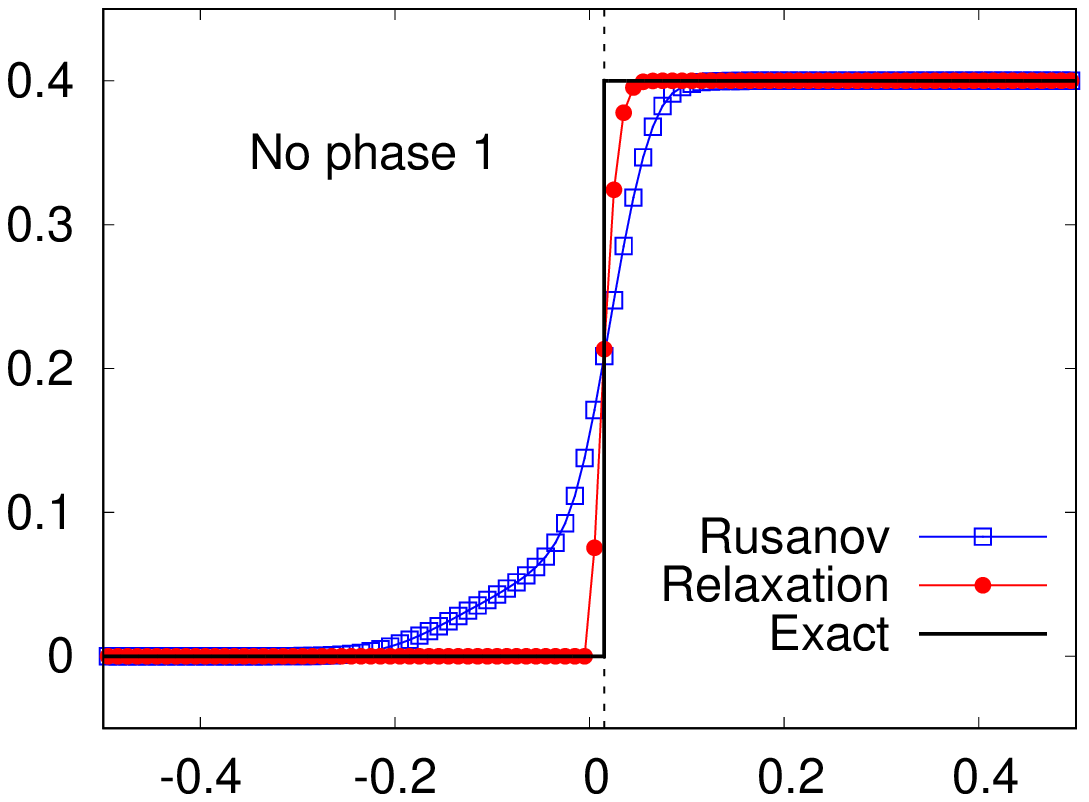}&
\includegraphics[width=5.5cm,height=4.3cm]{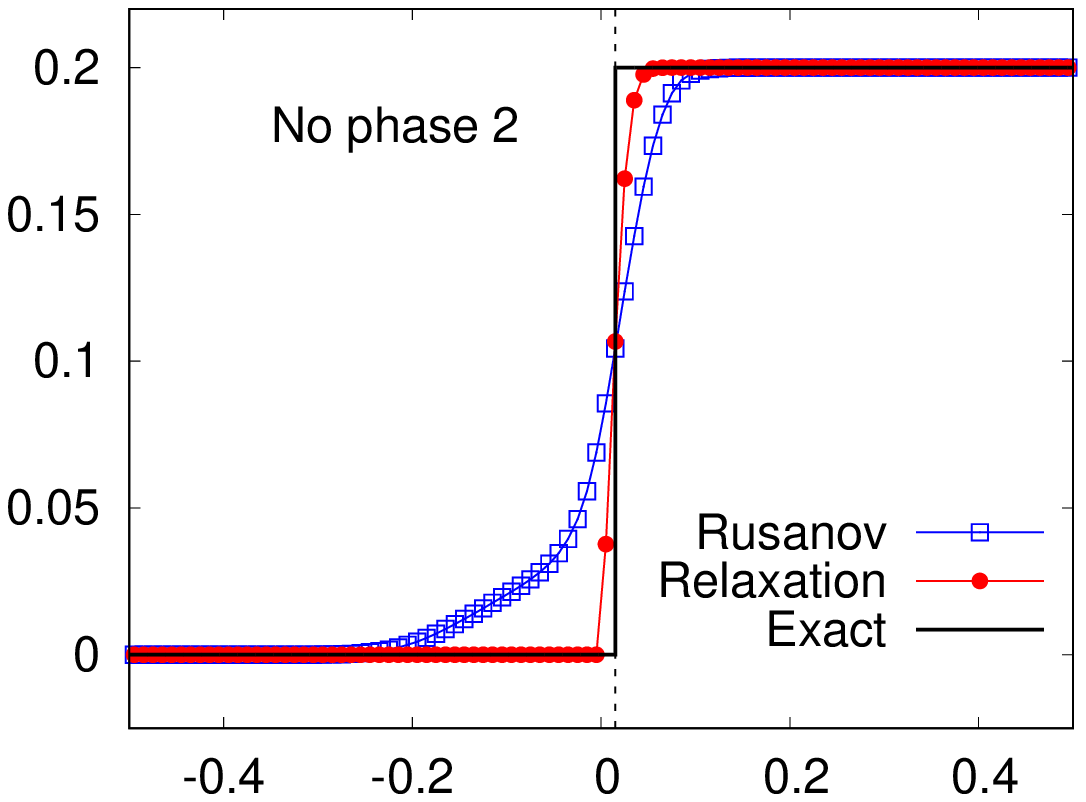}&  \\[3ex]
$u_1$ & $u_2$ & $u_3$ \\[1ex]
\includegraphics[width=5.5cm,height=4.3cm]{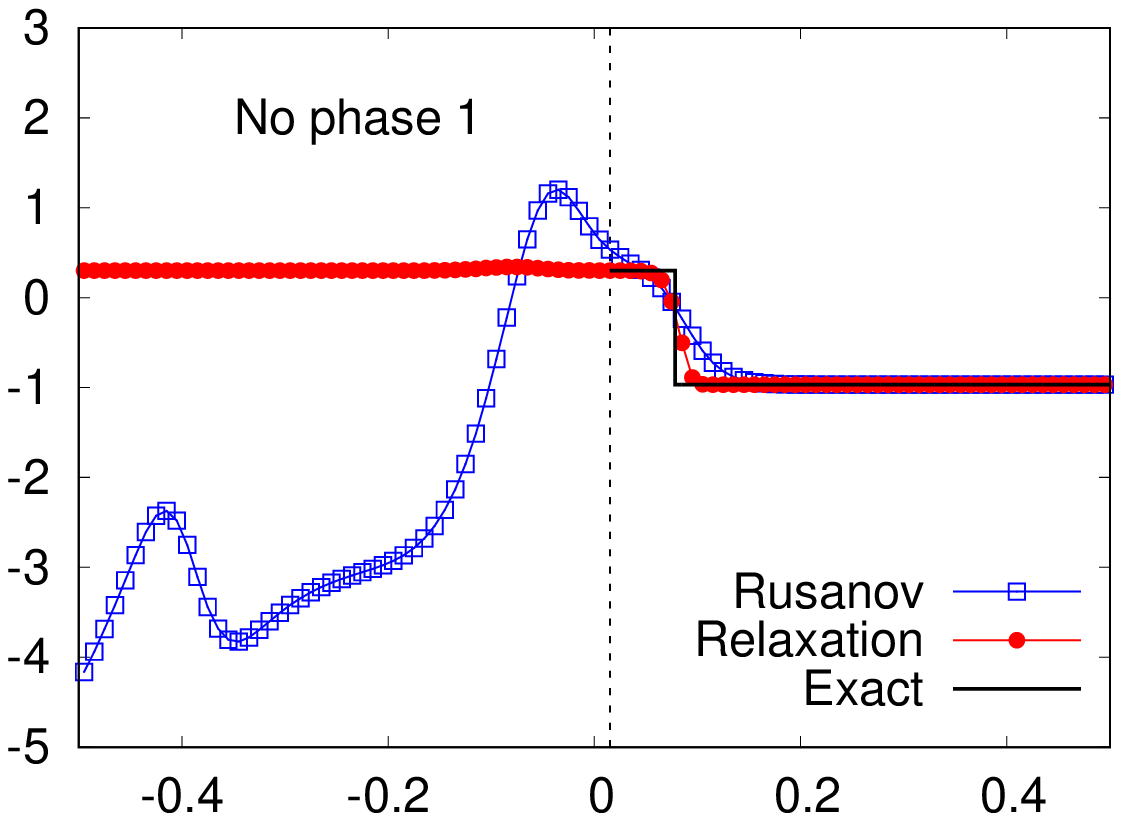} &
\includegraphics[width=5.5cm,height=4.3cm]{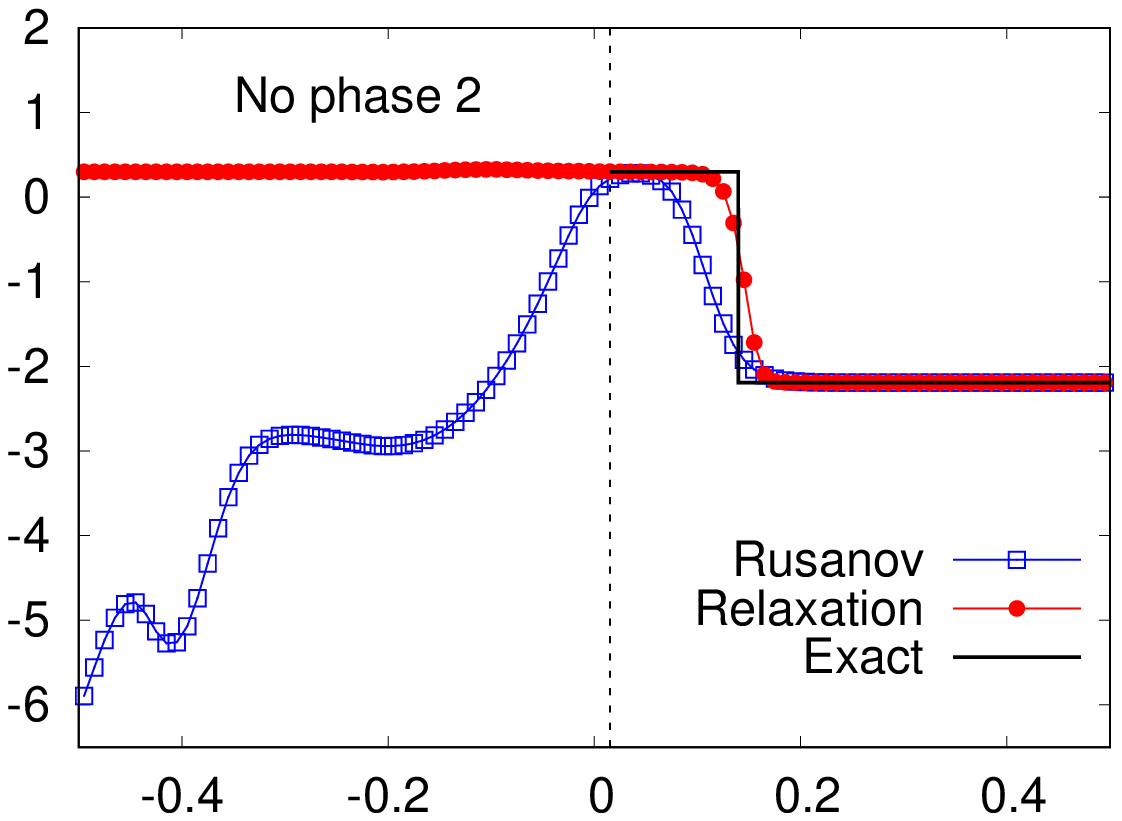} &
\includegraphics[width=5.5cm,height=4.3cm]{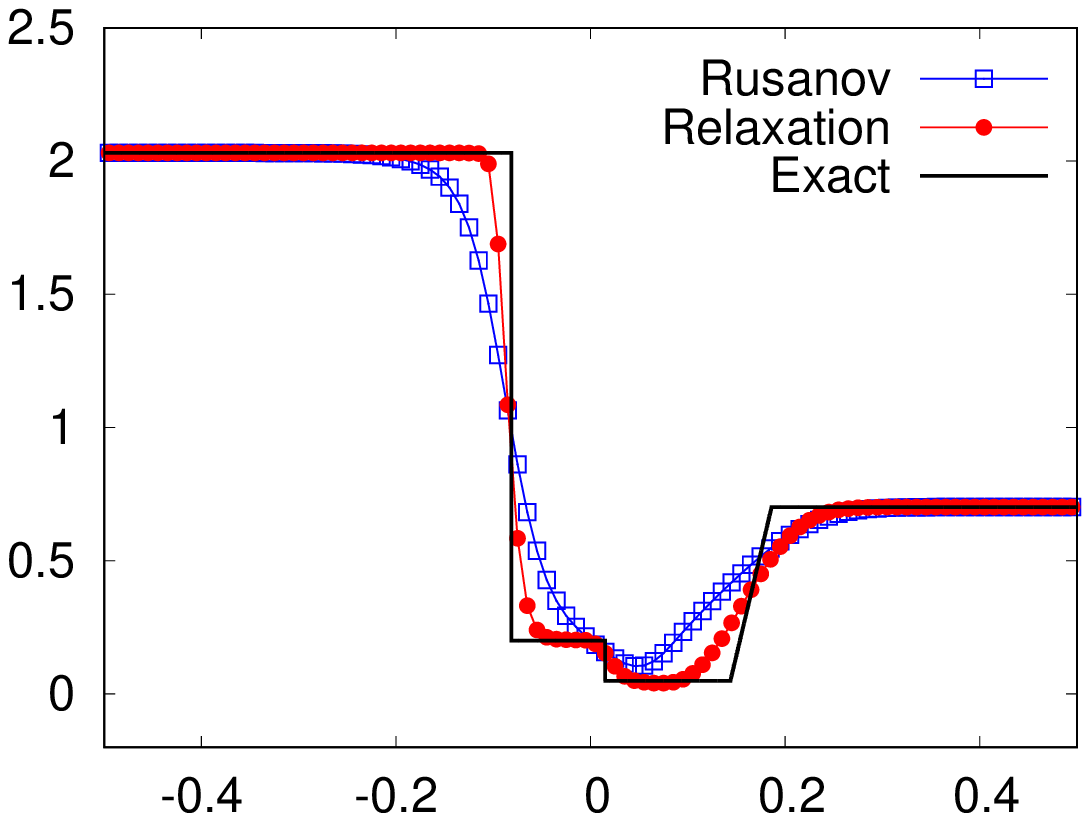}\\[3ex] 
$\rho_1$ & $\rho_2$ & $\rho_3$  \\[1ex]
\includegraphics[width=5.5cm,height=4.3cm]{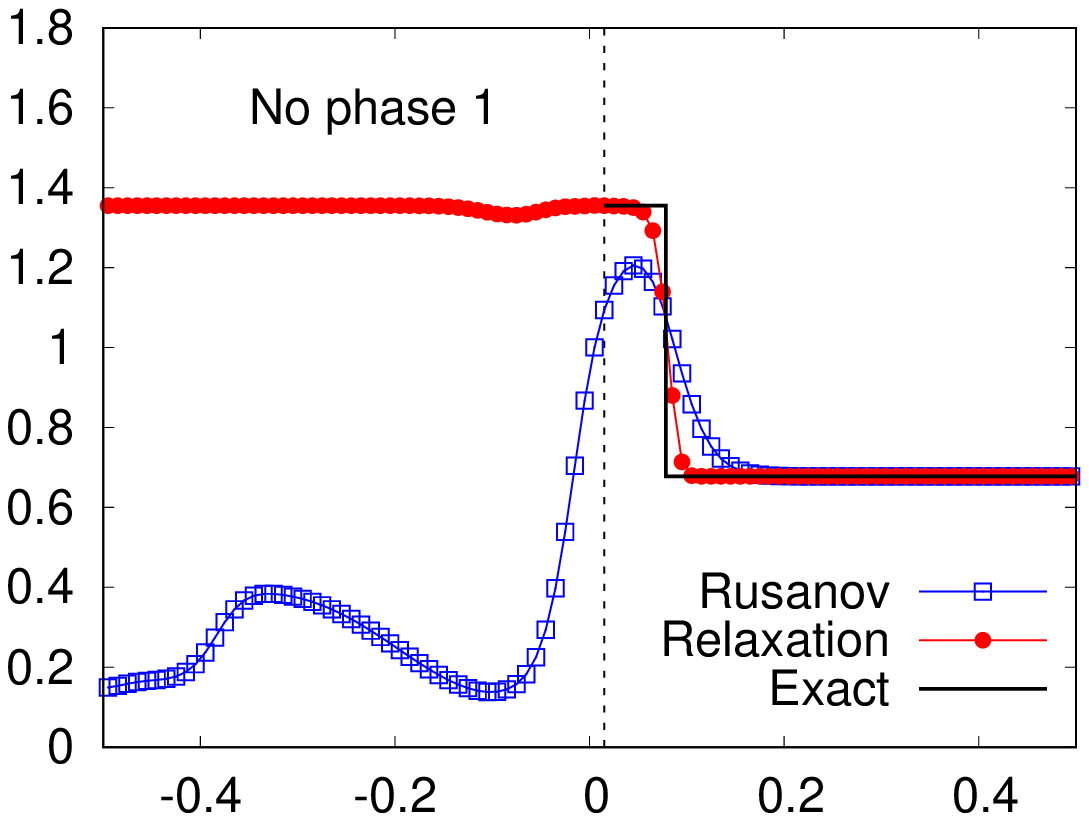} &
\includegraphics[width=5.5cm,height=4.3cm]{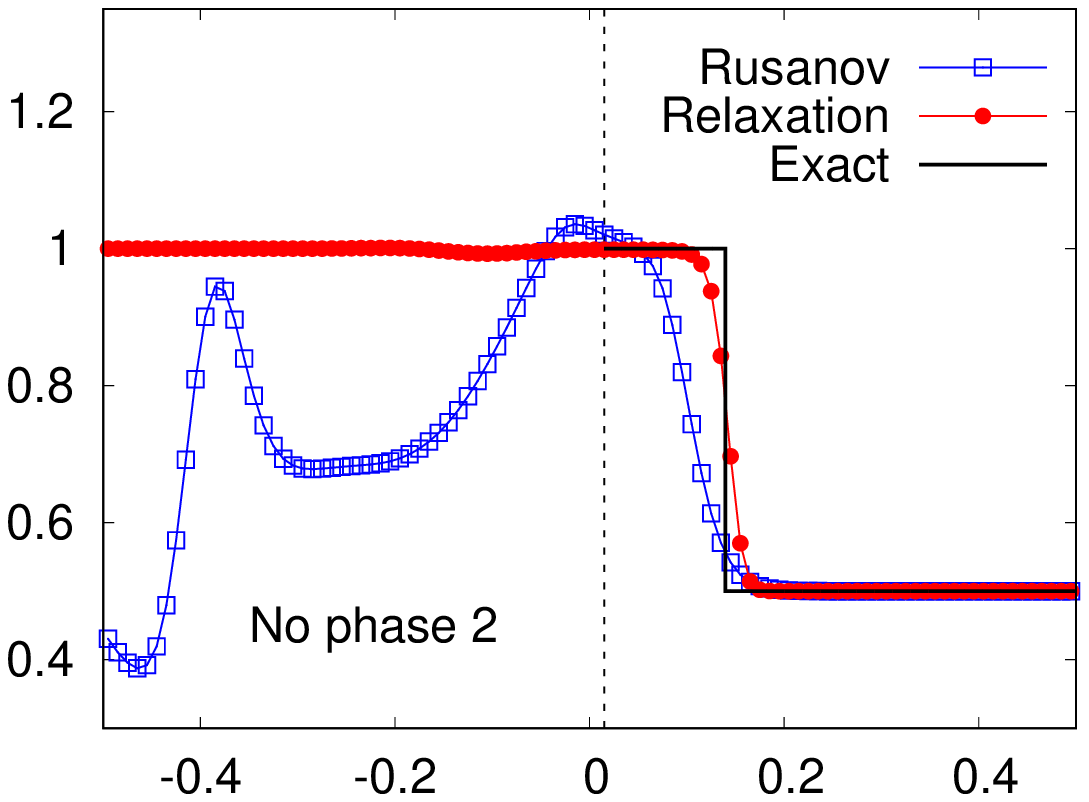} &
\includegraphics[width=5.5cm,height=4.3cm]{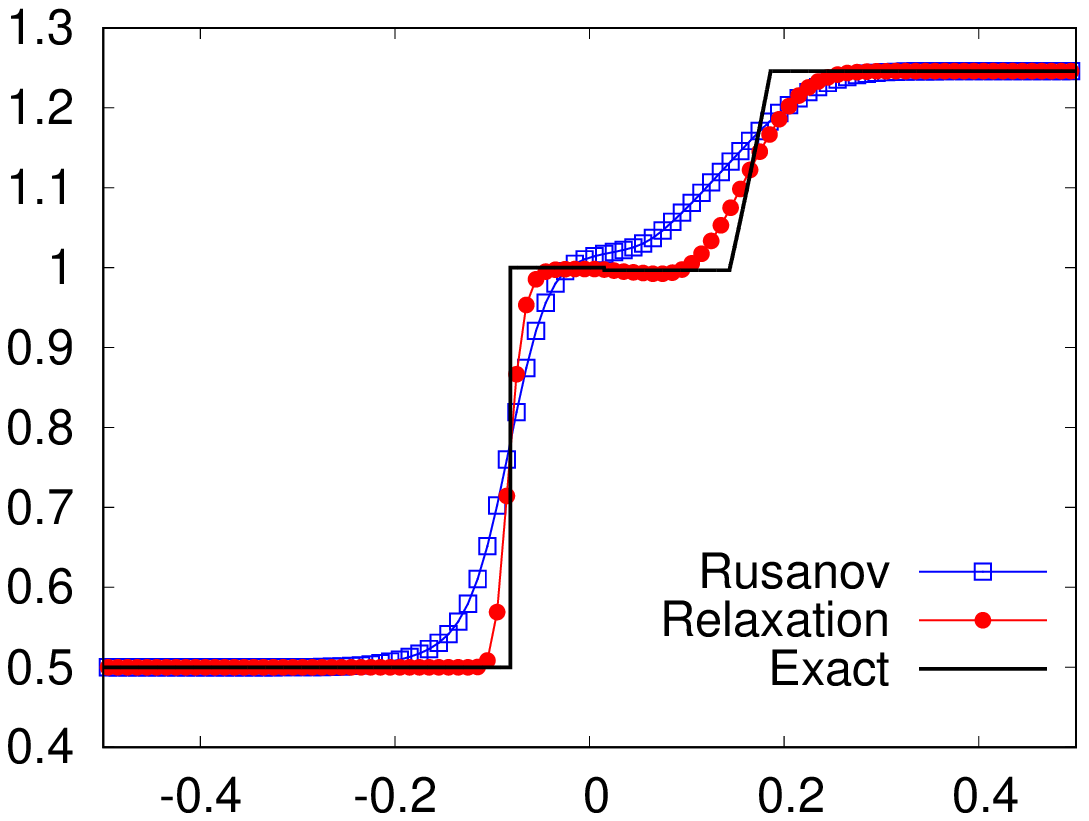} 
\end{tabular}
\protect \parbox[t]{17cm}{\caption{Test-case 2: space variations of the physical variables at the final time $T_{\rm max}=0.05$. Mesh size: $100$ cells.\label{Figcase2}}}
\end{center}
\normalsize
\end{figure}

\clearpage

\subsection{Test-case 3: interaction of a gas shock wave with a lid of rigid particles}

In this test-case, we consider a configuration where an incoming gas shock wave hits a cloud of spherical rigid particles. A sketch of the shock tube apparatus is given in Figure \ref{shock-tube}. The tube ranges from $x=0\,m$ to $x=3.75\,m$ and is closed at both ends. At the initial time $t=0\,s$, the cloud of particles lies between $x=2.97\,m$ and $x=3.37\,m$. This test-case is adapted from the experimental setup presented in \cite{chau-12-et,chau-11-exp}.

\begin{figure}[ht!]
\centering
\begin{center}
\begin{tikzpicture}[scale=3.5]
\fill [gray!20] (0.75,0) -- (0.75,0.5) -- (0,0.5) -- (0,0); 
\foreach \n in {1,2,...,10}
{
  \foreach \p in {1,2,...,8}
  {
      \draw [fill=black] (2.97+\p*0.05-0.02,\n*0.05-0.02) circle (0.02);
  }
}
\foreach \p in {0.1, 0.2, 0.3, 0.4, 0.5} {\draw [very thick,- ] (-0.2,\p-0.1) -- (0,\p); \draw [very thick,- ] (3.75+0.2,\p-0.1) -- (3.75,\p);}

\draw (0.37,0.25) node {\small High pressure};
\draw (2.0,0.25) node {\small Low pressure};

\draw [thick,-] (0,0.1) -- (0,-0.05) node[below=2pt] {$x=0$};
\draw [thick,-] (0.75,0.0) -- (0.75,-0.05) node[below=2pt] {$0.75$};
\draw [thick,-] (2.97,0.0) -- (2.97,-0.05) node[below=2pt] {$2.97$};
\draw [thick,-] (3.37,0.0) -- (3.37,-0.05) node[below=2pt] {$3.37$};
\draw [thick,-] (3.75,0.0) -- (3.75,-0.05) node[below=2pt] {$3.75$};


\draw [thick,<-] (2.7,0.5) -- (2.7,0.65) node [above=1pt] {$\substack{S_1 \\ x=2.7}$};
\draw [thick,<-] (3.0,0.5) -- (3.0,0.65) node [above=1pt] {$\substack{S_2 \\ x=3.0}$};
\draw [thick,<-] (3.2,0.5) -- (3.2,0.65) node [right=10pt,above=1pt] {$\substack{S_3 \\ x=3.2}$};
\draw [thick,<-] (3.7,0.5) -- (3.7,0.65) node [above=1pt] {$\substack{S_4 \\ x=3.7}$};

\draw [thick, dashed] (0.75,0) -- (0.75,0.5);
\draw [very thick,->] (-0.5,0)--(4.25,0) node[right=3pt] {$x$};
\draw [very thick,- ] (0,0) -- (0,0.5) -- (3.75,0.5) -- (3.75,0);
\end{tikzpicture}
\end{center}
\protect \parbox[t]{14cm}{\caption{Test-case 3: sketch of the experimental shock tube apparatus at $t=0$.\label{shock-tube}}}
\end{figure}
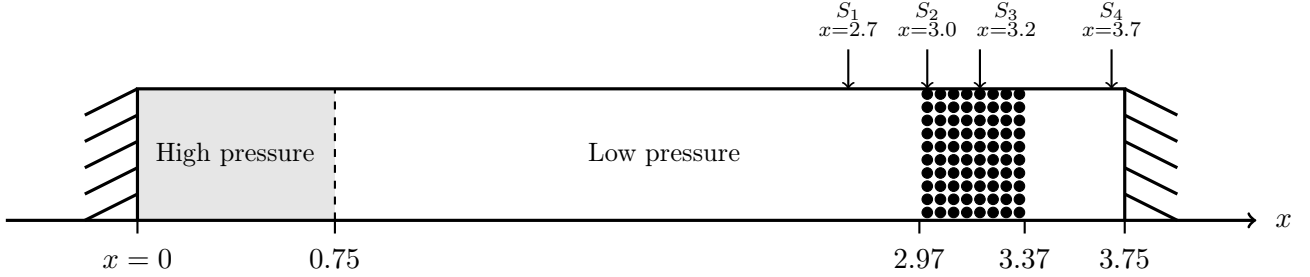

\medskip
A gas pressure disequilibrium is initiated at $x=0.75\,m$.
This initial pressure disequilibrium produces a left going gas rarefaction wave and a right-going gas shock wave. The rarefaction wave is soon reflected by the left wall while the shock first hits the lid of particles (producing a small left-going wave) and is then reflected by the right wall after crossing the lid of particles. In order to retrieve this behavior, four pressure transducers are located at stations $S_n$ for $n=1,..,4$ (see Figure \ref{shock-tube}).

\medskip
We decide to simulate this experiment with Rusanov's scheme and the relaxation scheme for the barotropic three-phase flow model. The particle phase has label 1 while the gas phase has label 2. Phase 3 is an ``absent'' phase, the statistical fraction of which is set to $\alpha_3=10^{-10}$ everywhere. At time $t=0$, the particle phase is residual outside the lid of particles and we set $\alpha_1=10^{-10}$ for $x\notin(2.97,3.37)$. In the lid of particles, \emph{i.e.} for $x\in (2.97,3.37)$ we set $\alpha_1=0.0104$.
At the initial time, all phases are at rest and in relative pressure equilibrium, with a space pressure disequilibrium at $x=0.75\,m$, thus for $k=1,..,3$:
\[
\begin{aligned}
& u_k(x,t=0)=0\, m.s^{-1}, \\[1ex]
& p_k(x,t=0)= \left \lbrace \begin{array}{ll}
                                7.10^5\,Pa, & \text{for $x<0.75\,m$,}\\
                                1.10^5\,Pa, & \text{for $x>0.75\,m$.}
                             \end{array}
                    \right.
\end{aligned}
\]
The particle phase follows a barotropic stiffened gas equation of state:
\begin{equation*}
 p_1(\rho_1) = c_1^2 \rho_1 + P_{1,{\rm ref}}, 
\end{equation*}
with $c_1=1500\,m.s^{-1}$ and $P_{1,{\rm ref}}<0$ is such that $p_1(\rho_{1,\rm{ref}})=10^5\, Pa$ where $\rho_{1,\rm{ref}}=10^3\, kg.m^{-3}$. The gas phase follows a barotropic ideal gas pressure law:
\begin{equation*}
 p_2(\rho_2) = \kappa_2 \rho_2^{\gamma_2},
\end{equation*}
with $\gamma_2=7/5$ and $\kappa_2$ is such that $p_2(\rho_{2,\rm{ref}})=10^5\, Pa$ where $\rho_{2,\rm{ref}}=1.27\, kg.m^{-3}$. In the computations, the same \emph{e.o.s.} is chosen for the residual phase with label 3. 

\begin{figure}[ht!]
\begin{center}
\begin{tabular}{cc}
Rusanov's scheme & Relaxation scheme  \\[1ex]
\includegraphics[width=8.5cm,height=6cm]{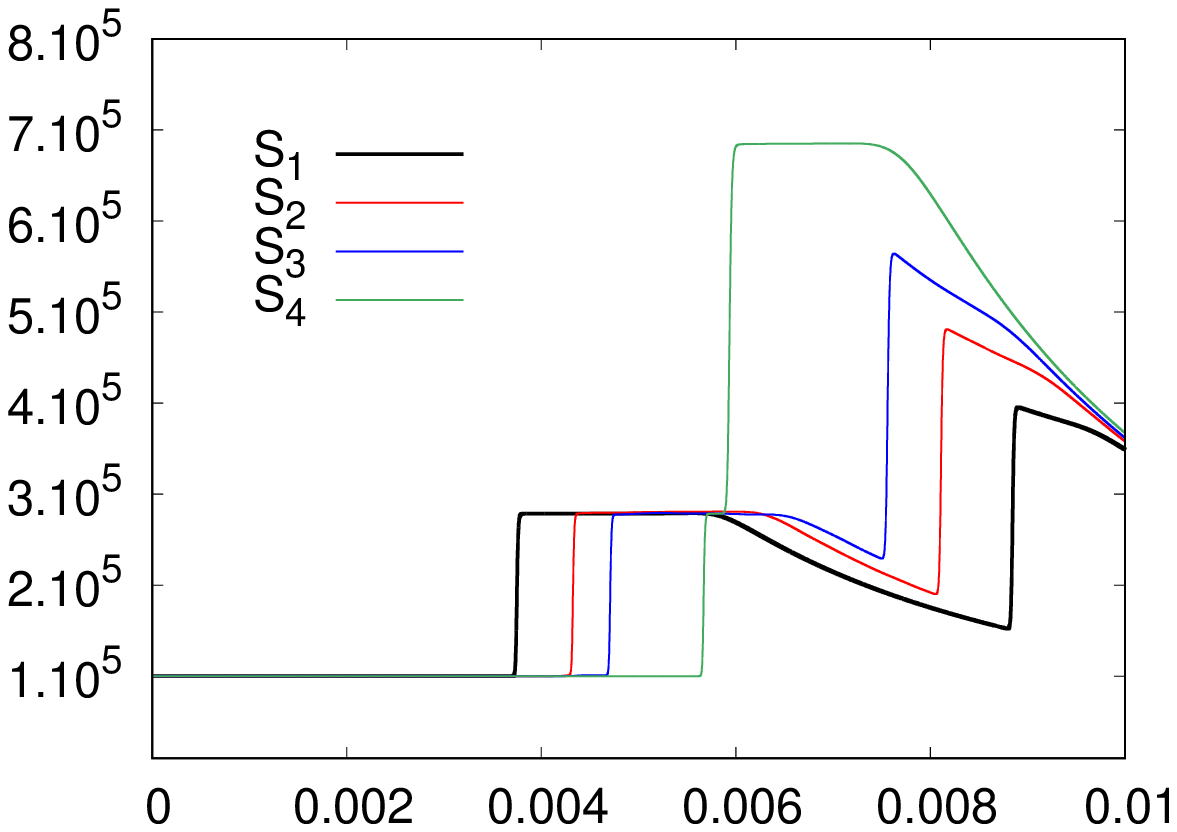}&
\includegraphics[width=8.5cm,height=6cm]{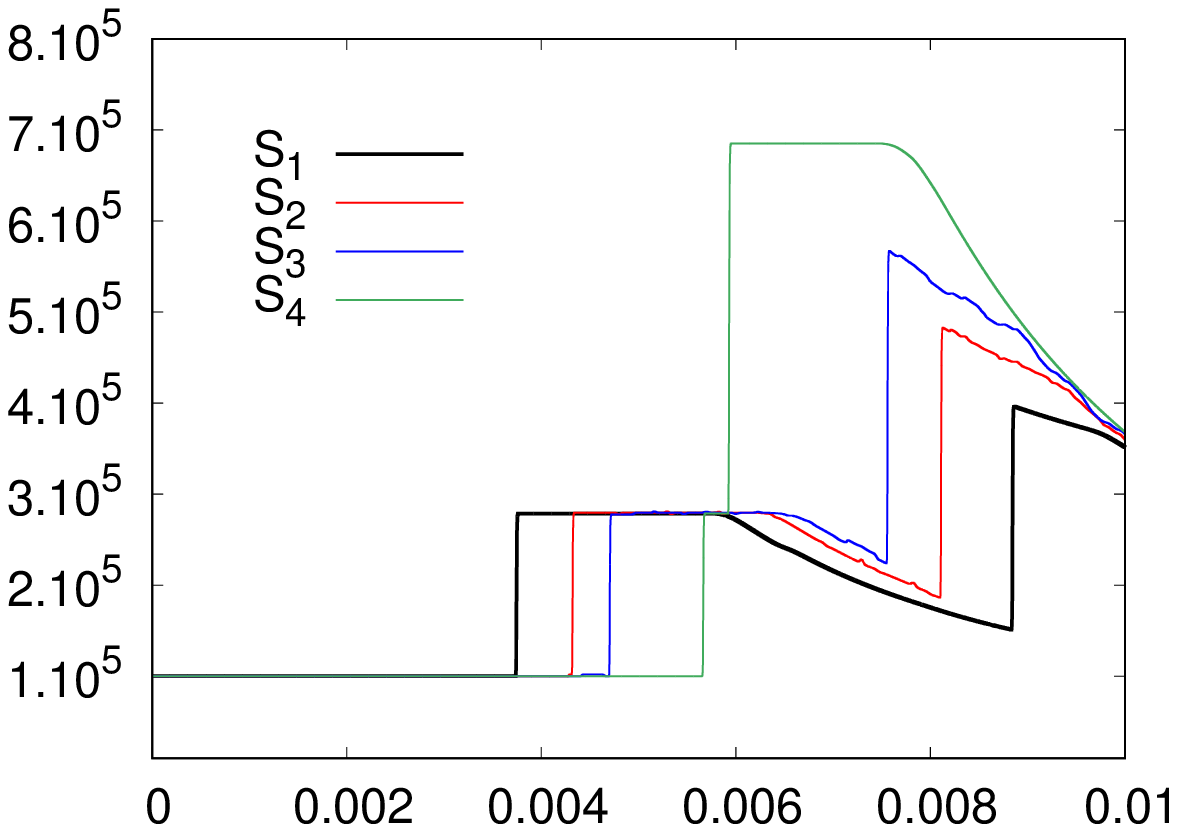}
\end{tabular}
\protect \parbox[t]{17cm}{\caption{Test-case 3: Mean pressure signals (in $Pa$) \emph{w.r.t} time (in $s$) at stations $S_n$, $n=1,..,4$. \label{Figcase6}}}
\end{center}
\end{figure}

\medskip
Figure \ref{Figcase6} displays the numerical approximations on the time interval $(0,0.01)$ of the total mean pressure $P=\sum_{k=1}^3 \alpha_k p_k$ at stations $S_n$ for $n=1,..,4$. The obtained computations are run with a 5000 cell mesh with both the relaxation scheme and Rusanov's scheme. For both schemes, the observed curves reflect the expected behavior of the total mean pressure. At station $S_1$, the pressure first jumps to a value $P^*$
when the right-going gas shock wave reaches $x=2.7\,m$ around time $t=0.0037\,s$. The pressure at this location then remains steady until the reflection of the left-going rarefaction wave meets this position around time $t=0.0058 \,s$ which causes a decrease of the pressure until a second jump occurs (around time $t=0.0088\,s$) due to the reflection of the right-going gas shock wave which has hit the right wall boundary after crossing the lid of particles. Similar features are observed at stations $S_2$ and $S_3$ with the (expected) difference that the more rightward the station is located, the later the first pressure jump occurs and the sooner the second pressure jump occurs. If we now turn to station $S_4$, a first pressure jump to the value $P^*$ is recorded around time $t=0.0056 \,s$ due to the right-going gas shock wave, soon followed by a second pressure jump (around $t=0.0059\,s$) to a value $P^{**}$ due to the reflection of this wave on the right wall boundary. The pressure then remains steady until the location of station $S_4$ is reached by the reflection of the left-going rarefaction wave on the left wall boundary, causing a decrease in the pressure.

\medskip
In Table \ref{Table_pressure}, we compare the values of $P^*$ and $P^{**}$ obtained respectively in the experiment \cite{chau-12-et,chau-11-exp}, to those obtained with Rusanov's scheme and with the Relaxation scheme. We can see that the pressure values obtained with the barotropic three-phase flow model are slightly over estimated, a behavior that has already been observed in \cite{bou-18-rel} on the same test-case. 

\begin{table}[ht!]
\centering
\begin{tabular}{|c|ccc|}
\hline
				& Experiment	& Rusanov's scheme & Relaxation scheme	 \\
\hline
$P^* \ (\times 10^5\, Pa)$	& $\approx 2.4 $	& $2.78$	& $2.78$    \\
$P^{**} (\ \times 10^5\, Pa)$	& $\approx 5.0$		& $6.85$       	& $6.85$\\
\hline
\end{tabular}
\protect \parbox[t]{17cm}{\caption{Total mean pressure behind the right-going gas shock $P^*$ and behind the reflection of this shock wave on the right wall boundary $P^{**}$.\label{Table_pressure}}}
\end{table}



\section{Conclusion}
We have extended to the $N$-phase compressible model developed in \cite{her-16-cla} the relaxation finite volume scheme designed in \cite{coq-13-rob} for the barotropic Baer-Nunziato two phase flow model. The obtained scheme inherits the main properties of the scheme designed for the two phase framework. It applies to general barotropic equations of state (see Remark \ref{rem:any eos}). It is able to cope with arbitrarily small values of the statistical phase fractions. The approximated phase fractions and phase densities are proven to remain positive and a discrete energy inequality is proven under a classical CFL condition.

\medskip
For $N=3$, three test cases have been implemented which assess the good behaviour of the relaxation scheme. 
For the same level of refinement, the relaxation scheme is shown to be much more accurate than Rusanov's scheme, and for a given level of approximation error, the relaxation scheme is shown to perform much better in terms of computational cost than Rusanov's scheme.
Moreover, contrary to Rusanov's scheme which develops strong oscillations when approximating vanishing phase solutions, the numerical results show that the relaxation scheme remains stable in such regimes. Given that Rusanov's scheme is the only numerical scheme presently available for the considered three phase flow model, the present paper therefore constitutes an improvement in this area.

\medskip
Several natural extensions to this work can be considered. First of all, the scheme can be easily extended to the multidimensional framework. Indeed, the multidimensional version of the $N$-phase model (see \cite{bou-18-rel} for the multi-D three phase model) is invariant under Galilean transformation and under frame rotation. Thus, the one-dimensional relaxation Riemann solver can still be used to obtain a finite volume scheme on two and three dimensional unstructured grids. An update of the cell unknown is obtained through a convex combination of 1D updates associated with the cell faces. These 1D updates are computed with the relaxation 1D scheme by considering local 1D Riemann problems in the orthogonal direction to the grid faces. Thanks to the convex combination, the positivity and entropy properties of the scheme are still valid for the multidimensional scheme under a natural CFL condition. We refer to \cite{sal-12-ana} where this extension is detailed for the Baer-Nunziato two phase flow model, and where 2D-test cases have already been conducted. Another natural generalization is the extension of the scheme to higher order. A formally order two scheme can be obtained by considering a classical minmod reconstruction on the symmetrizing variable and a second order Runge-Kutta time scheme (see \cite{cro-13-app} for the two phase model). Such a procedure however does not ensure the preservation of the discrete energy inequality. Designing entropy satisfying second order numerical schemes is an open topic which is still under investigation. Finally, the extension of the relaxation numerical scheme to the multiphase flow model with non barotropic equations of state will be considered in a forthcoming paper. The key property of the relaxation scheme which allows this relatively simple extension is the existence of the discrete energy inequalities \eqref{ener_prop_1}-\eqref{ener_prop_k}. Thanks to these and to the second principle of thermodynamics which connects the phasic energies and the transported phasic entropies, one is able to extend the present Riemann solver to the full model with general \textit{e.o.s.} through minor adaptations. This work has already been done for the Baer-Nunziato two phase model in \cite{coq-17-pos}. The obtained scheme was shown to compare well with two of the most popular existing available schemes, namely Schwendeman-Wahle-Kapila's Godunov-type scheme \cite{SWK} and Tokareva-Toro's HLLC scheme \cite{TT}.

\appendix

\section{Eigenstructure of the relaxation system}
\label{app:sec:spectre:multi:relax}
\begin{prop}
\label{app:prop:spectre:multi:relax}
System \eqref{sys:multi:relax} is weakly hyperbolic on $\Omega_\W$ in the following sense. It admits the following $4N-1$ real eigenvalues: $\sigma_1(\W)=..=\sigma_{N-1}(\W)=u_1$, $\sigma_{N-1+k}(\W)=u_k-a_k \tau_k$, $\sigma_{2N-1+k}(\W)=u_k+a_k\tau_k$ and $\sigma_{3N-1+k}(\W)=u_k$ for $k=1,..,N$. All the characteristic fields are linearly degenerate and the corresponding right eigenvectors are linearly independent if, and only if,
\begin{equation}
\label{app:multi:hypfail}
\alpha_k \neq 0, \quad \forall k=1,..,N \qquad \text{and} \qquad |u_1-u_k| \neq a_k\tau_k, \quad \forall k=2,..,N.
\end{equation}
\end{prop}

\medskip
\begin{proof}
Denoting $\w=(\alpha_2,..,\alpha_N,\rho_1,u_1,\T_1,..,\rho_N,u_N,\T_N)^T$, the smooth solutions of system \eqref{sys:multi0} satisfy the following equivalent system:
\[
 \dv_t\w + \mathscr{A}(\w)\dv_x\w=0,
\]
where $\mathscr{A}(\w)$ is the block matrix:
\[
 \mathscr{A}(\w) = \begin{pmatrix}
  A 			&   \rvline   & \bigzero \\
  \hline
  \begin{matrix}
   B_1 \\
   \vdots \\
   B_N
  \end{matrix}
			&   \rvline   &  \begin{matrix}
			                  C_1  &      &         &   \\
			                       &      & \ddots  &    \\
			                       &      &         & C_N 
			                 \end{matrix}
 \end{pmatrix}.
\]
We write $\pi_k$ instead of $\pi_k(\tau_k,\T_k)$ in order to ease the notations.
Defining $M_k=(u_k-u_1)/(a_k\tau_k)$, for $k=2,..,N$, the matrices $A$, $B_1,..,B_N$ and $C_1,..,C_N$ are given as follows. $A$ is a $(N-1)\times(N-1)$ diagonal matrix with $u_1$ on the diagonal. $B_1,..,B_N$ are $3\times(N-1)$ matrices and $C_1,..,C_N$ are $3\times3 $ matrices defined by:
\[
\begin{aligned}
  &B_1=\left( \frac{1}{\alpha_1\rho_1} \sum_{k=2}^N (\pi_k-\pi_1)\delta_{i,2}\,\delta_{j+1,k} \right )_{\substack{1\leq i \leq 3 \\ 1 \leq j \leq N-1}},\\
 \quad
  &B_k = \left( \frac{\rho_k\,M_k\,a_k\,\tau_k}{\alpha_k} \delta_{i,1}\,\delta_{j+1,k} \right )_{\substack{1\leq i \leq 3 \\ 1 \leq j \leq N-1}},\quad \text{for $k=2,..,N$},
 \\
 &C_k = \begin{pmatrix}
        u_k & \rho_k & 0 \\
        a_k^2\tau_k^3 & u_k & (p_k'(\T_k)+a_k^2)\tau_k \\
        0 & 0 & u_k
       \end{pmatrix}, \quad \text{for $k=1,..,N$},
\end{aligned}
\]
where $\delta_{p,q}$ is the Kronecker symbol: for $p,q\in\N$, $\delta_{p,q}=1$ if $p=q$ and $\delta_{p,q}=0$ otherwise. Since $A$ is diagonal and $C_k$ is $\R$-diagonalizable with eigenvalues $u_k-a_k\tau_k$ and $u_k+a_k\tau_k$ and $u_k$, the matrix $\mathscr{A}(\w)$ admits the eigenvalues $u_1$ (with multiplicity $N$), $u_1-a_1\tau_1$, $u_1+a_1\tau_1$ and $u_k-a_k\tau_k$, $u_k+a_k\tau_k$ and $u_k$ for $k=2,..,N$. In addition, $\mathscr{A}(\w)$ is $\R$-diagonalizable provided that the corresponding right eigenvectors span $\R^{4N-1}$. The right eigenvectors are the columns of the following block matrix:
\[
 \mcal{R}(\w) = \begin{pmatrix}
  A' 			&   \rvline   & \bigzero \\
  \hline
  \begin{matrix}
   B_1' \\
   \vdots \\
   B_N'
  \end{matrix}
			&   \rvline   &  \begin{matrix}
			                  C_1'  &      &         &   \\
			                       &      & \ddots  &    \\
			                       &      &         & C_N'
			                 \end{matrix}
 \end{pmatrix},
\]
where $A'$ is a $(N-1)\times(N-1)$ diagonal matrix defined by $A'={\rm diag}(1-M_2^2,..,1-M_N^2)$. $B_1',..,B_N'$ are $3\times(N-1)$ matrices and $C_1',..,C_N'$ are $3\times3 $ matrices defined by:
\[
\begin{aligned}
& B_1'=\left( -\frac{1}{\alpha_1 a_1^2\tau_1^2} \sum_{k=2}^N (\pi_k-\pi_1)(1-M_k^2)\delta_{i,1}\,\delta_{j+1,k} \right )_{\substack{1\leq i \leq 3 \\ 1 \leq j \leq N-1}},
\\
& B_k' = \left( \Big (\frac{M_k^2 \rho_k}{\alpha_k} \delta_{i,1}-\frac{M_k a_k \tau_k}{\alpha_k}\delta_{i,2} \Big) \,\delta_{j+1,k} \right )_{\substack{1\leq i \leq 3 \\ 1 \leq j \leq N-1}}, \quad \text{for $k=2,..,N$},
\\
& C_k = \begin{pmatrix}
       \rho_k & \rho_k  & 0\\
       -a_k\tau_k   & a_k\tau_k & 0 \\
       0 & 0 & 1
      \end{pmatrix}, \quad \text{for $k=1,..,N$}.
\end{aligned}
\]
The first $N-1$ columns are the eigenvectors associated with the eigenvalue $u_1$. For $k=1,..,N$, the $\big(N+2(k-1)\big)$-th, $\big(N+(2k-1)\big)$-th and $\big(N+2k\big)$-th columns are the eigenvectors associated with $u_k-a_k\tau_k$, $u_k+a_k\tau_k$ and $u_k$ respectively. We can see that $\mcal{R}(\w)$ is invertible if and only if $M_k\neq1$ for all $k=2,..,N$ \emph{i.e.} if and only if $|u_k-u_1|\neq a_k\tau_k$ for all $k=2,..,N$. In particular, if for some $k=2,..,N$, one has $u_k=u_1$, $\mcal{R}(\w)$ is still invertible.
Denote $(\r_j(\w))_{1\leq j \leq 4N-1}$ the columns of $\mcal{R}(\w)$. If $1\leq j \leq N-1$, we can see that the $(N+1)$-th component of $\r_j(\w)$ is zero. This implies that for all  $1\leq j \leq N-1$ $\r_j(\w) \cdot \gradi_{\w}(u_1)=0$. Hence, the field associated with the eigenvalue $u_1$ is linearly degenerated. Now we observe that for all $k=1,..,N$:
\[
 \begin{aligned}
  \r_{N+2(k-1)}(\w) \cdot \gradi_{\w} (u_k-a_k \tau_k) =  0,\\[2ex] 
  \r_{N+(2k-1)}(\w) \cdot \gradi_{\w} (u_k+a_k \tau_k) =  0,\\[2ex]
  \r_{N+2k)}(\w) \cdot \gradi_{\w} (u_k) =  0.
 \end{aligned}
\]
This means that all the other fields are also linearly degenerated.
\end{proof}

\section{Practical computation of the numerical fluxes}
\label{sec:flux}

\subsection{Computation of the relaxation parameters $(a_k)_{k=1,..,N}$}
\label{sec:choixak}
In the numerical scheme \eqref{fvscheme}, at each interface $x_{j+\frac 12}$, one must determine the relaxation parameters $(a_k)_{k=1,..,N}$ in order to compute the Riemann solution $\W_{\rm Riem}(\xi;\W_L,\W_R)$ where $\W_L=\mathscr{M}(\U_j^n)$ and $\W_R=\mathscr{M}(\U_{j+1}^n)$.
The parameters $(a_k)_{k=1,..,N}$, must be chosen large enough so as to satisfy several requirements:
\begin{itemize}
 \item In order to ensure the stability of the relaxation approximation, $a_k$ must satisfy Whitham's condition (\ref{whithambis}). For simplicity however, we do not impose Whitham's condition everywhere in the solution of the Riemann problem \eqref{sys:multi:relax}-\eqref{sys:multi:relax:CI} (which is possible however), but only for the left and right initial data at each interface:
 \begin{equation}
  \text{for $k=1,..,N$}, \quad a_k > \max \left (\rho_{k,L}\,c_{k}(\rho_{k,L}),\rho_{k,R}\,c_{k}(\rho_{k,L}) \right ),
 \end{equation}
 where $c_k(\rho_k)$ is the speed of sound in phase $k$. In practice, no instability was observed during the numerical simulations due to this simpler Whitham-like condition.
 
\item In order to compute the solution of the relaxation Riemann problem, the specific volumes $\tdd_{k,L}$ and $\tdd_{k,R}$ defined in \eqref{diese} must be positive. The expressions of $\tdd_{k,L}$ and $\tdd_{k,R}$
are second order polynomials in $a_k^{-1}$ whose constant terms are respectively $\tau_{k,L}$ and $\tau_{k,R}$. Hence, by taking $a_k$ large enough, one can guarantee that $\tdd_{k,L}>0$ and $\tdd_{k,R}>0$, since the initial specific volumes $\tau_{k,L}$ and $\tau_{k,R}$ are positive. 

\item Finally, in order for the relaxation Riemann problem \eqref{sys:multi:relax}-\eqref{sys:multi:relax:CI} to have a positive solution, the parameters $(a_k)_{k=1,..,N}$ must be chosen so as to meet condition \eqref{TheCondition} of Theorem \ref{multi:THEtheorem}. As explained in Remark \ref{rem:akgrand}, assumption \eqref{TheCondition} is always satisfied if the parameters $(a_k)_{k=1,..,N}$ are taken large enough.
\end{itemize}
Thereafter, we propose an algorithm for the computation of the parameters $(a_k)_{k=1,..,N}$ at each interface.
We begin with choosing $\eta$ a (small) parameter in the interval $(0,1)$. In our numerical computations, we took $\eta=0.01$. 

\bigskip

\hspace{1.5cm}
\vrule
\begin{minipage}{20cm}
 \begin{itemize} \itemsep2pt \parskip0pt \parsep0pt
 \item For $k=1,..,N$ initialize $a_{k}$:
\begin{description}
\item \qquad $a_{k} := (1+\eta) \max \left( \rho_{k,L}\, c_{k}(\rho_{k,L}), \rho_{k,R}\, c_{k}(\rho_{k,R})  \right )$.
\end{description}
\item For $k=1,..,N$:
\begin{description}
\item \qquad do $\lbrace a_{k} := (1+\eta)a_{k} \rbrace$  while $ \big ( \tdd_{k,L} \leq 0$ or $\tdd_{k,R} \leq 0 \big )$.
\end{description}
\item Do $\lbrace \text{for k=1,..,N, } a_{k} := (1+\eta)a_{k} \rbrace $ while \big (not \eqref{TheCondition}\big).
\end{itemize}
\end{minipage}

\bigskip
This algorithm always converges in the sense that there is no infinite looping due to the while-conditions. Moreover, this algorithm provides reasonable values of $(a_k)_{k=1,..,N}$, since in all the numerical simulations, the time step obtained through the CFL condition (\ref{cfl}) remains reasonably large and does not go to zero.
In fact, the obtained values of $(a_k)_{k=1,..,N}$ are quite satisfying since the relaxation scheme compares favorably with Rusanov's scheme, in terms of numerical diffusion and CPU-time performances (see Figure \ref{Figcase1ter}).

\subsection{Construction of the solution to the Riemann problem \eqref{sys:multi:relax}-\eqref{sys:multi:relax:CI}}
\label{constr_sol}

Now, given $(\W_L,\W_R)\in\Omega_{\W}$ and $(a_k)_{k=1,..,N}$ such that the conditions of Theorem \ref{multi:THEtheorem} are met, we give the expression of the piecewise constant solution of the Riemann problem  \eqref{sys:multi:relax}-\eqref{sys:multi:relax:CI}:
\[
\xi \mapsto \W_{\rm Riem}(\xi;\W_{L},\W_{R}).
\]
We recall the following notations built on the initial states $(\vect{W}_L,\vect{W}_R)$ and on the relaxation parameters $(a_k)_{k=1,..,N}$, which are useful for the computation of the solution.
For $k=1,..,N$: 
\begin{equation}
\label{diese1}
\begin{aligned}
\udd_k &= \dfrac{1}{2} \left (u_{k,L}+u_{k,R} \right )-\dfrac{1}{2a_k} \left (\pi_k(\tau_{k,R},\T_{k,R}) - \pi_k(\tau_{k,L},\T_{k,L}) \right ),
\\
\pidd_k &= \dfrac{1}{2} \left ( \pi_k(\tau_{k,R},\T_{k,R}) + \pi_k(\tau_{k,L},\T_{k,L}) \right )- \dfrac{a_k}{2} \left (u_{k,R}- u_{k,L} \right ),
\\
\tdd_{k,L} &= \tau_{k,L} + \dfrac{1}{a_k}(\udd_k - u_{k,L}), 
\\
\tdd_{k,R} &= \tau_{k,R} - \dfrac{1}{a_k}(\udd_k - u_{k,R}).
\end{aligned}
\end{equation} 

The relaxation Riemann solution $\W_{\rm Riem}(\xi;\W_{L},\W_{R})$ is determined through the following steps:
\begin{enumerate}
\item Define the function $\theta_1$ by:
\[
\theta_1(u) =  a_1\paren{\alpha_{1,L}+\alpha_{1,R}} (u -  \udd_1). 
\]
\item Define for $(\nu,\omega)\in\R_+^{*}\times\R_+^*$ the two-variable function:
\[
\M_0(\nu,\omega) = \dfrac{1}{2} \left( \dfrac{1+\omega^2}{1-\omega^2} \left(
1+
\dfrac{1}{\nu} \right ) - \sqrt{\left(
\dfrac{1+\omega^2}{1-\omega^2}\right )^2 \left( 1+ \dfrac{1}{\nu}
\right )^2 -\dfrac{4}{\nu} }\right ), 
\]
which can be extended by continuity to $\omega=1$ by setting $\M_0(\nu,1)=0$.
\item For $k=2,..,N$ define the function $\theta_k$ by:
\begin{multline*}
\theta_k(u) =  a_k\paren{\alpha_{k,L}+\alpha_{k,R}} (u -  \udd_k) \\
+2a_k^2
\left \lbrace
\begin{array}{ll}
\alpha_{k,L} \, \tdd_{k,L} \, \M_0\paren{\frac{\alpha_{k,L}}{\alpha_{k,R}},\frac{1-\Me_k}{1+\Me_k}}, & \quad \text{with \ $\Me_k=\frac{\udd_k-u}{a_k\tdd_{k,L}}$  \ if \ $\udd_k\geq u$},\\[3ex] 
\alpha_{k,R} \, \tdd_{k,R} \, \M_0\paren{\frac{\alpha_{k,R}}{\alpha_{k,L}},\frac{1-\Me_k}{1+\Me_k}}, & \quad \text{with \ $\Me_k=\frac{\udd_k-u}{a_k\tdd_{k,R}}$ \ if \ $\udd_k\leq u$}.
\end{array}
\right.
\end{multline*}
\item Define the function $\Theta$ by:
\[
 \Theta(u) = \theta_1(u)+...+\theta_N(u).
\]
\item Assuming \eqref{TheCondition}, use an iterative method (\textit{e.g.} Newton's method or a dichotomy (bisection) method) to compute the unique $u_1^*$ in the interval $\Big(\max\limits_{k=1,..,N} \left \lbrace u_{k,L}-a_{k}\tau_{k,L}\right \rbrace,\min\limits_{k=1,..,N} \left \lbrace u_{k,R}+a_{k}\tau_{k,R}\right \rbrace \Big)$ that satisfies:
\begin{equation}
\label{fp:app}
\Theta(u_1^*)=\sum_{k=2}^N (\pidd_1-\pidd_k) \Delta \alpha_k. 
\end{equation}
\item The intermediate states for phase 1 are given by:
\begin{center}
\begin{tikzpicture}[scale=3]
\small
\tikzstyle{axes}=[thin,>=latex]
\begin{scope}[axes]
        \draw[->] (-1,0)--(1,0) node[right=3pt] {$x$};
        \draw[->] (0,0)--(0,1) node[left=3pt] {$t$};
	\draw [very thick,color=red] (0,0) -- (30:1cm) node [color=black, above=1pt] {$\qquad u_{1,R}+a_1\tau_{1,R}$};
        \draw [very thick,color=red] (0,0) -- (80:0.8cm) node [color=black, above=2pt] {$u_1^*$};
	\draw [very thick,color=red] (0,0) -- (150:1cm) node [color=black, above=1pt] {$u_{1,L}-a_1\tau_{1,L}$};
	\draw (-0.8,0.2) node[blue] {$\substack{\alpha_{1,L}\\ \tau_{1,L} \\ u_{1,L} \\ \T_{1,L}}$};
	\draw (-0.25,0.6) node[blue] {$\substack{\alpha_{1}^-\\ \tau_{1}^- \\ u_{1}^- \\ \T_{1}^-}$};
	\draw (0.3,0.6) node[blue] {$\substack{\alpha_{1}^+\\ \tau_{1}^+ \\ u_{1}^+ \\ \T_{1}^+}$};
	\draw (0.7,0.15) node[blue] {$\substack{\alpha_{1,R}\\ \tau_{1,R} \\ u_{1,R} \\ \T_{1,R}}$};
\end{scope}
\normalsize
\end{tikzpicture}
\end{center}
where:
\[
\begin{array}{llll}
\alpha_{1}^- = \alpha_{1,L},
&\quad \tau_1^- = \tdd_{1,L} +\frac{1}{a_1} (u_1^*-\udd_1),
&\quad u_1^+=u_1^*,
&\quad\T_1^- = \T_{1,L}, \\[2ex]
\alpha_{1}^+ = \alpha_{1,R},
&\quad \tau_1^+ = \tdd_{1,R} -\frac{1}{a_1} (u_1^*-\udd_1),
&\quad u_1^+=u_1^*,
&\quad \T_1^+ = \T_{1,R}.
\end{array}
\]
\item There are three possibilities for the wave configuration of phase $k$ depending on the sign of $\udd_k-u_1^*$:
\begin{itemize}
 \item If $\udd_k>u_1^*$, defining $\nu_k=\alpha_{k,L}/\alpha_{k,R}$, $\Me_k=(\udd_k-u_1^*)/(a_k\tdd_{k,L})$ and $\M_k=\M_0\paren{\nu_k,\frac{1-\Me_k}{1+\Me_k}}$, the intermediate states for phase $k$ are given by:
\begin{center}
\begin{tikzpicture}[scale=3]
\small
\tikzstyle{axes}=[thin,>=latex]
\begin{scope}[axes]
	\draw[->] (-1,0)--(1,0) node[right=3pt] {$x$};
        \draw[->] (0,0)--(0,1) node[left=3pt] {$t$};
	\draw [very thick,color=red] (0,0) -- (30:1cm) node[color=black,above] {$u_{k,R}+a_k\tau_{k,R}$};
        \draw [very thick,color=red] (0,0) -- (80:1cm) node[color=black,right] {$u_k^*=u_k^-=u_k^+$};
	\draw [dashed,very thick,color=red] (0,0) -- (115:1cm) node[color=black,above] {$u_1^*$};
	\draw [very thick,color=red] (0,0) -- (155:1cm) node[color=black,above] {$u_{k,L}-a_k\tau_{k,L}$};
	\draw (-0.8,0.15) node[blue] {$\substack{\alpha_{k,L}\\ \tau_{k,L} \\ u_{k,L} \\ \T_{k,L}}$};
	\draw (-0.4,0.5) node[blue] {$\substack{\alpha_{k}^-\\ \tau_{k}^- \\ u_{k}^- \\ \T_{k}^-}$};
	\draw (-0.1,0.6) node[blue] {$\substack{\alpha_{k}^+\\ \tau_{k}^+ \\ u_{k}^+ \\ \T_{k}^+}$};
	\draw (0.3,0.5) node[blue] {$\substack{\alpha_{k,R*}\\ \tau_{k,R*} \\ u_{k,R*} \\ \T_{k,R*}}$};
	\draw (0.7,0.15) node[blue] {$\substack{\alpha_{k,R}\\ \tau_{k,R} \\ u_{k,R} \\ \T_{k,R}}$};
	\draw (0,-0.2) node[] {Wave ordering $u_k>u_1$};%
\end{scope}
\normalsize
\end{tikzpicture}
\end{center}
where:
\[
\begin{array}{llll}
\alpha_k^-=\alpha_{k,L},
&\tau_k^- = \tdd_{k,L} \dfrac{ 1 - \Me_k}{1 -\M_k},   
&u_k^- = u_1^*+ a_k\M_k \tau_k^-,
&\T_k^- = \T_{k,L},  \\[2ex]
\alpha_k^+=\alpha_{k,R},
&\tau_k^+ =\tdd_{k,L} \dfrac{ 1 + \Me_k }{1 + \nu_k \M_k},   
&u_k^+ = u_1^*+ \nu_k a_k \M_k \tau_k^+, 
&\T_k^+ = \T_{k,L},  \\[2ex]
\alpha_{k,R*}=\alpha_{k,R},
&\tau_{k,R*} = \tdd_{k,R} + \tdd_{k,L} \dfrac{\Me_k - \nu_k \M_k}{1 + \nu_k \M_k},
&u_{k,R*} = u_1^*+ \nu_k a_k \M_k \tau_k^+,  
&\T_{k,R*} =\T_{k,R}. 
\end{array}
\]
 \item If $\udd_k<u_1^*$, defining $\nu_k=\alpha_{k,R}/\alpha_{k,L}$, $\Me_k=-(\udd_k-u_1^*)/(a_k\tdd_{k,R})$ and $\M_k=\M_0\paren{\nu_k,\frac{1-\Me_k}{1+\Me_k}}$, the intermediate states for phase $k$ are given by:
\begin{center}
\begin{tikzpicture}[scale=3]
\small
\tikzstyle{axes}=[thin,>=latex]
\begin{scope}[axes]
	\draw[->] (-1,0)--(1,0) node[right=3pt] {$x$};
        \draw[->] (0,0)--(0,1) node[above=3pt] {$t$};
	\draw [very thick,color=red] (0,0) -- (30:1cm) node[color=black,above] {$u_{k,R}+a_k\tau_{k,R}$};
        \draw [very thick,color=red] (0,0) -- (80:1cm) node[color=black,above] {$u_1^*$};
	\draw [dashed,very thick,color=red] (0,0) -- (115:1cm) node[color=black,above] {$u_k^*=u_k^-=u_k^+$};
	\draw [very thick,color=red] (0,0) -- (155:1cm) node[color=black,above] {$u_{k,L}-a_k\tau_{k,L}$};
	\draw (-0.8,0.15) node[blue] {$\substack{\alpha_{k,L}\\ \tau_{k,L} \\ u_{k,L} \\ \T_{k,L}}$};
	\draw (-0.4,0.5) node[blue] {$\substack{\alpha_{k,L*}\\ \tau_{k,L*} \\ u_{k,L*} \\ \T_{k,L*}}$};
	\draw (-0.1,0.6) node[blue] {$\substack{\alpha_{k}^-\\ \tau_{k}^- \\ u_{k}^- \\ \T_{k}^-}$};
	\draw (0.3,0.5) node[blue] {$\substack{\alpha_{k}^+\\ \tau_{k}^+ \\ u_{k}^+ \\ \T_{k}^+}$};
	\draw (0.7,0.15) node[blue] {$\substack{\alpha_{k,R}\\ \tau_{k,R} \\ u_{k,R} \\ \T_{k,R}}$};
	\draw (0,-0.2) node[] {Wave ordering $u_k<u_1$};%
\end{scope}
\normalsize
\end{tikzpicture}
\end{center}
where:
\[
\begin{array}{llll}
\alpha_{k}^+ = \alpha_{k,R}, 
&\tau_k^+ = \tdd_{k,R} \dfrac{ 1 - \Me_k}{1 -\M_k},   
&u_k^+ =   u_1^*- a_k\M_k \tau_k^+, 
&\T_k^+ =  \T_{k,R},  \\[2ex]
\alpha_{k}^- = \alpha_{k,L}, 
&\tau_k^- =\tdd_{k,R} \dfrac{ 1 + \Me_k }{1 + \nu_k \M_k},   
&u_k^- =  u_1^*- \nu_k a_k \M_k \tau_k^-, 
&\T_k^- =  \T_{k,R},  \\[2ex]
\alpha_{k,L*} = \alpha_{k,L}, 
&\tau_{k,L*} = \tdd_{k,L} + \tdd_{k,R} \dfrac{\Me_k - \nu_k \M_k}{1 + \nu_k \M_k},
&u_{k,L*} =  u_1^*- \nu_k a_k \M_k \tau_k^-,  
&\T_{k,L*} =  \T_{k,L}. 
\end{array}
\]
\item If $\udd_k=u_1^*$, the intermediate states for phase $k$ are given by:
\begin{center}
\begin{tikzpicture}[scale=3]
\small
\tikzstyle{axes}=[thin,>=latex]
\begin{scope}[axes]
	\draw[->] (-1,0)--(1,0) node[right=3pt] {$x$};
        \draw[->] (0,0)--(0,1) node[left=3pt] {$t$};
	\draw [very thick,color=red] (0,0) -- (30:1cm) node[color=black,above] {$u_{k,R}+a_k\tau_{k,R}$};
	\draw [dashed,very thick,color=red] (0,0) -- (115:1cm) node[color=black,above] {$u_1^*=u_k^*$};
	\draw [very thick,color=red] (0,0) -- (155:1cm) node[color=black,above] {$u_{k,L}-a_k\tau_{k,L}$};
	\draw (-0.8,0.15) node[blue] {$\substack{\alpha_{k,L}\\ \tau_{k,L} \\ u_{k,L} \\ \T_{k,L}}$};
	\draw (-0.4,0.5) node[blue] {$\substack{\alpha_{k}^-\\ \tau_{k}^- \\ u_{k}^- \\ \T_{k}^-}$};
	\draw (+0.1,0.6) node[blue] {$\substack{\alpha_{k}^+\\ \tau_{k}^+ \\ u_{k}^+ \\ \T_{k}^+}$};
	\draw (0.7,0.15) node[blue] {$\substack{\alpha_{k,R}\\ \tau_{k,R} \\ u_{k,R} \\ \T_{k,R}}$};
	\draw (0,-0.2) node[] {Wave ordering $u_k=u_1$};%
\end{scope}
\normalsize
\end{tikzpicture}
\end{center}
where:
\[
\begin{array}{llll}
\alpha_k^-=\alpha_{k,L},
&\tau_k^- = \tdd_{k,L},
&u_k^- = u_k^*=\udd_k
&\T_k^- = \T_{k,L},  \\[2ex]
\alpha_k^+=\alpha_{k,R},
&\tau_k^+ =\tdd_{k,R},   
&u_k^+ = u_k^*=\udd_k,
&\T_k^+ = \T_{k,R}. 
\end{array}
\]
\end{itemize}
\end{enumerate}


\subsection{Calculation of the numerical fluxes} 

Given the solution $\xi \mapsto \W_{\rm Riem}(\xi;\W_L,\W_R)$ where $\W_L=\mathscr{M}(\U_L)$ and $\W_R=\mathscr{M}(\U_R)$ of the relaxation Riemann problem, the numerical fluxes are computed as follows:
\begin{equation*}
\mathbf{F}^{\pm}(\vect{U}_{L},\vect{U}_{R})=\left [
\begin{matrix}
  (u_1^*)^{\pm} \Delta\alpha_{1}\\
 \vdots \\
 (u_1^*)^{\pm} \Delta\alpha_{N-1} \\
 \alpha_1 \rho_1 u_1 \\
 \vdots \\
 \alpha_N \rho_N u_N\\[1ex]
 \alpha_1\rho_1 u_1^2 + \alpha_1 \pi_1+\frac{(u_1^*)^\pm}{u_1^*}\, \sum_{l=2}^N\pi_l^*\Delta\alpha_l\\[1ex]
 \alpha_1\rho_1 u_1^2 + \alpha_1 \pi_1 - \frac{(u_1^*)^\pm}{u_1^*}\,\pi_2^*\Delta\alpha_2\\
  \vdots \\
 \alpha_N\rho_N u_N^2 + \alpha_N \pi_N-\frac{(u_1^*)^\pm}{u_1^*} \, \pi_N^*\Delta\alpha_N
\end{matrix}  \right ]\Big (\W_{\rm Riem} \left (0^\pm;\mathscr{M}(\vect{U}_L),\mathscr{M}(\vect{U}_R)\right) \Big ),
\end{equation*}
where $u_1^*$ is the solution of the fixed-point problem \eqref{fp:app} and for $k=2,..,N$, $\pi_k^* \Delta \alpha_k =  \pidd_k \Delta \alpha_k + \theta_k(u_1^*)$ with $\pidd_k$ defined in \eqref{diese1} and $\alpha_k=\alpha_{k,R}-\alpha_{k,L}$.
In the above expression of the numerical fluxes, we have denoted $(u_1^*)^+=\max(u_1^*,0)$, $(u_1^*)^-=\min(u_1^*,0)$ and the functions $x\mapsto \frac{(x)^\pm}{x}$ are extended by $0$ at $x=0$.

\section*{Acknowledgements} The author would like to thank Jean-Marc H\'erard for the productive discussions and his thorough reading of this paper which has constantly led to improving the text. 

\small
\bibliographystyle{plain}
\bibliography{bibfile}
\normalsize

\end{document}